\documentclass[12pt, reqno]{amsart}
\usepackage{amsmath}
\usepackage{amssymb}
\usepackage{enumerate}
\usepackage{mathrsfs}
\usepackage{amsfonts}
\usepackage{color}
\usepackage{vmargin}
\usepackage{amsthm}
\usepackage{graphicx} 
\usepackage{esint}
\usepackage{mathtools}
\usepackage{MnSymbol}
\usepackage[shortcuts]{extdash}
\usepackage{hyperref}

\hypersetup{
  colorlinks=true,
  linkcolor=blue,
  citecolor=blue,
  urlcolor=blue
}

\makeatletter
\renewcommand{\tocsection}[3]{%
  \indentlabel{\@ifnotempty{#2}{\bfseries\ignorespaces#1 #2\quad}}\bfseries#3}
\renewcommand{\tocsubsection}[3]{%
  \indentlabel{\@ifnotempty{#2}{\ignorespaces#1 #2\quad}}#3}
\renewcommand{\tocsubsubsection}[3]{%
  \indentlabel{\@ifnotempty{#2}{\ignorespaces#1 #2\quad}}#3}

\newcommand\@dotsep{4.5}
\def\@tocline#1#2#3#4#5#6#7{\relax
  \ifnum #1>\c@tocdepth % then omit
  \else
    \par \addpenalty\@secpenalty\addvspace{#2}%
    \begingroup \hyphenpenalty\@M
    \@ifempty{#4}{%
      \@tempdima\csname r@tocindent\number#1\endcsname\relax
    }{%
      \@tempdima#4\relax
    }%
    \parindent\z@ \leftskip#3\relax \advance\leftskip\@tempdima\relax
    \rightskip\@pnumwidth plus1em \parfillskip-\@pnumwidth
    #5\leavevmode\hskip-\@tempdima{#6}\nobreak
    \leaders\hbox{$\m@th\mkern \@dotsep mu\hbox{.}\mkern \@dotsep mu$}\hfill
    \nobreak
    \hbox to\@pnumwidth{\@tocpagenum{\ifnum#1=1\bfseries\fi#7}}\par% <-- \bfseries for \section page
    \nobreak
    \endgroup
  \fi}
\AtBeginDocument{%
\expandafter\renewcommand\csname r@tocindent0\endcsname{0pt}
}
\def\l@subsection{\@tocline{2}{0pt}{2.5pc}{5pc}{}}
\def\l@subsubsection{\@tocline{2}{0pt}{4.6pc}{1pc}{}}
\makeatother

\allowdisplaybreaks
\usepackage[T1]{fontenc}
\usepackage{setspace}

\def\XXint#1#2#3{{\setbox0=\hbox{$#1{#2#3}{\int}$ }
\vcenter{\hbox{$#2#3$ }}\kern-.6\wd0}}

\long\def\symbolfootnote[#1]#2{\begingroup%
\def\thefootnote{\fnsymbol{footnote}}\footnote[#1]{#2}\endgroup}

\setmarginsrb{20mm}{20mm}{20mm}{20mm}{10mm}{10mm}{10mm}{10mm}

\newtheoremstyle{remark}
{}{}{}{}{\bfseries}{.}{.5em}{{\thmname{#1 }}{\thmnumber{#2}}{\thmnote{ (#3)}}}
\theoremstyle{remboldstyle}

\RequirePackage{amsthm}

\newtheorem{tw}{Theorem}[section]
\newtheorem{defi}[tw]{Definition}

\newtheorem{lemma}[tw]{Lemma}

\newtheorem{prop}[tw]{Proposition}
\newtheorem{remark}[tw]{Remark} 
\newtheorem{example}[tw]{Example}

\theoremstyle{definition}

\newcommand{\lip}{\textrm{lip} \,}

\newcommand{\diam}{\textnormal{diam} \,}

\DeclareMathOperator*{\esssup}{ess\,sup}
\DeclareMathOperator*{\essinf}{ess\,inf}

\numberwithin{equation}{section}

\def\={\hspace{-3mm}&=&\hspace{-3mm}}

\title[\bf Embeddings of variable Sobolev, Besov, and Triebel-Lizorkin spaces]{\bf Embeddings of variable Sobolev, Besov, and Triebel-Lizorkin spaces on metric measure spaces}
\author{Ryan Alvarado, Micha{\l} Dymek,  Przemys{\l}aw G\'{o}rka, Nijjwal Karak}



\begin{document}
\begin{abstract}
Sobolev-type embeddings on metric measure spaces encode a subtle interaction between the analytic regularity of functions and the geometry of the underlying domain space. In this paper we develop an embedding theory for variable Haj{\l}asz-type smoothness spaces on metric measure spaces whose ``dimension'' is allowed to vary pointwise through a bounded exponent $Q(\cdot)$ that governs a lower Ahlfors growth condition on the measure. We introduce variable exponent Haj{\l}asz-Sobolev spaces $M^{s(\cdot),p(\cdot)}$, Haj{\l}asz-Triebel-Lizorkin spaces $M^{s(\cdot)}_{p(\cdot),q(\cdot)}$, and Haj{\l}asz-Besov spaces $N^{s(\cdot)}_{p(\cdot),q(\cdot)}$, and establish Sobolev, Morrey, and Moser-Trudinger type embeddings into variable exponent Lebesgue and H\"older spaces. These embeddings are proved both locally (on balls) under a lower Ahlfors $Q(\cdot)$-regularity condition on the measure and regularity assumptions on the exponents (notably log-H\"older continuity), and globally under additional geometric hypotheses such as geometric doubling and mild uniform bounds on the measure of unit balls. We also identify geometric conditions that are not only sufficient but, in appropriate forms, necessary for the validity of these embeddings, showing in particular that such inequalities force a lower growth bound on the measure of order $r^{Q(x)}$. Several results related to global embeddings are novel, even in the context of spaces with constant exponents.

\end{abstract}

\maketitle
\tableofcontents

\section*{Introduction}

Sobolev-type embeddings are among the central structural results in analysis, describing how certain bounds on the derivative of a function can translate into it having improved integrability and regularity properties (measured on the Lebesgue and H\"older scales); see, for example, \cite{sob36,sob38,gag,Niren,trud,AF03}. A persistent and fascinating theme in this theory is the interplay between analytic information (viewed through these embeddings) and the geometric and measure-theoretic properties of the underlying domain space; see \cite{AF03,BK96,hajlaszkt1,AGH20,HHL06,Haj96}.
In the Euclidean setting $\mathbb{R}^n$, the classical Sobolev, Morrey, and Moser-Trudinger embeddings for Sobolev spaces (including fractional ones) and, more generally, for Triebel-Lizorkin and Besov spaces are by now standard, and their numerous variants permeate nonlinear PDE, geometric analysis, and harmonic analysis; see \cite{AF03,Tri}. In the last two decades, two directions in this vein have proved particularly fruitful and technically demanding: (i) the passage from constant to variable integrability and smoothness, and (ii) the passage from $\mathbb{R}^n$ to metric measure spaces, where the geometry is encoded in measure growth conditions and covering properties rather than in coordinates; see \cite{DHHR11,AH,HHL06,Coifman}.

Regarding the passage to variable integrability and smoothness, function spaces with variable exponents and related parameters were originally motivated 
by models with non-standard growth conditions in PDE and by applications in elasticity, image processing, and fluid mechanics; see, for example, \cite{Z87,R00,LLP10}.
Since then, a broad variable function space theory
has emerged, including the theory of variable exponent Lebesgue and Sobolev spaces $L^{p(\cdot)}$ and $W^{1,p(\cdot)}$, where certain structural regularity assumptions on the exponents (notably log-H\"older continuity)  guarantee boundedness of fundamental operators and allow one to recover many tools from the constant exponent theory;
see \cite{CUF13,DHHR11,Gaczkowski} and the references therein.  A parallel development concerns smoothness spaces of Triebel-Lizorkin and Besov type with variable integrability and smoothness parameters. On $\mathbb{R}^n$ these spaces admit several equivalent descriptions (Littlewood-Paley, local means, differences), and a substantial embedding theory has been established; see, for example, \cite{AH,AS09,Rn}.

As concerns the passage from $\mathbb{R}^n$ to metric measure spaces, by the late 1970s it had been well recognized that much of contemporary real analysis required little structure on the ambient. In fact, one significant development in the theory of function spaces emerged in the 1990s with the introduction of Sobolev spaces defined on metric measure spaces, where several robust substitutes for classical Sobolev spaces that do not rely on the notion of a derivative have been developed; see, for example, \cite{Haj96,cheeger,SMP,SMP2,shanmugalingam}. In particular, the Haj{\l}asz pointwise gradient approach in \cite{Haj96} provides a flexible framework for defining Sobolev spaces on very general metric measure spaces, even in settings without rectifiable curves, and for proving Sobolev-type inequalities under natural geometric hypotheses; see \cite{Haj03,Haj96,AGH20}. The Haj{\l}asz pointwise approach has also led to metric-space analogues of the classical Triebel-Lizorkin and Besov spaces (often called Haj{\l}asz-Triebel-Lizorkin and Haj{\l}asz-Besov spaces), with embeddings and characterizations developed in many settings; see, for instance, \cite{KYZ11,Rn,AWYY21,AYY22,AYY24,zwartestale,K20,K19}.

The present paper lies at the intersection of these two directions, and additionally incorporates a third feature: variable dimension. Following the viewpoint initiated in \cite{HHL06}, we work on metric measure spaces  for which the Ahlfors-regularity condition on the measure is governed by a variable exponent $Q(\cdot)$. This framework accommodates spaces that, from the viewpoint of measure and geometry, behave like objects of different dimensions glued together,  and it allows one to formulate and prove Sobolev embeddings that reflect this pointwise dimensional information. 
While \cite{HHL06} shows that variable dimension naturally leads to embeddings for first-order Sobolev spaces into variable exponent Lebesgue spaces, many aspects of the Sobolev theory remain unsettled on very general metric measure spaces (for example, without doubling, boundedness, or connectedness assumptions), and in this level of generality the interaction between variable smoothness and variable integrability has not yet been developed for Triebel--Lizorkin and Besov spaces.

The main aim of this paper is to develop a systematic embedding theory for variable Haj{\l}asz-type smoothness spaces on very general metric measure spaces having variable dimension, and to identify necessary and sufficient geometric conditions on the underlying metric measure space ensuring that these embeddings hold.
More specifically, in the setting of variable-dimension metric measure spaces, for measurable exponents $s(\cdot)$, $p(\cdot)$, and $q(\cdot)$, we introduce variable Haj{\l}asz-Sobolev spaces $M^{s(\cdot),p(\cdot)}$,  Haj{\l}asz-Triebel-Lizorkin spaces $M^{s(\cdot)}_{p(\cdot),q(\cdot)}$, and the Haj{\l}asz-Besov spaces $N^{s(\cdot)}_{p(\cdot),q(\cdot)}$, defined in the spirit of \cite{Rn}, and establish  Sobolev-type embeddings of these spaces into variable exponent Lebesgue  and H\"older-type spaces.
These embeddings are proved both in local form (on balls) under Ahlfors-type lower bounds on the measure and regularity assumptions on the exponents (notably log-H\"older continuity, as in \cite{DHHR11,GKP}), and in global form under additional geometric hypotheses such as geometric doubling and mild uniform bounds on the measure of unit balls.

In our second set of results, we address the geometric side of the theory by identifying conditions on the metric measure space that are not merely sufficient but, in appropriate forms, also necessary for the validity of the embeddings under consideration. In particular, we show that the validity of Sobolev-type embeddings proven in this paper forces a lower Ahlfors-regularity growth condition on the measure of order $r^{Q(x)}$. 
In this sense, our embedding results provide an intimate link between functional inequalities for variable Haj{\l}asz--Sobolev, Triebel--Lizorkin, and Besov spaces and quantitative geometric properties of the underlying variable-dimension metric measure spaces.

The paper is organized as follows. In Section~\ref{sect:preliminaries}, we fix notation and provide some basic background on metric measure spaces, including a few technical lemmas that will be used in subsequent sections. Section~\ref{sect:varexptheory}
establishes fundamental aspects of the variable exponent theory, including variable exponent Lebesgue and H\"older spaces and log-H\"older continuity.
In Section~\ref{sect:variableexponentspaces}, we introduce the variable exponent Haj{\l}asz-Sobolev, Haj{\l}asz-Triebel-Lizorkin, and Haj{\l}asz-Besov spaces and establish technical norm estimates for Lipschitz cut-off functions. 
Section~\ref{sect:mainresults} contains the core results of the paper, detailing structural embeddings and the local and global Sobolev, Morrey, and Moser--Trudinger type inequalities. Finally, Section~\ref{sect:necessity} addresses the necessity of the geometric conditions by showing that the aforementioned functional embeddings force a lower growth bound on the measure.

Lastly, we note here that, although we formulate our main results in the setting of metric measure spaces, these results could also be equally well formulated in quasi-metric measure spaces (as was done in, e.g., \cite{AYY24}), but we omit such generalizations.

\section{Notation and Preliminaries}\label{sect:preliminaries}

In this section, we fix our notation and collect some foundational material used throughout the paper. More specifically, we recall our standing assumptions on metric measure spaces, including  the concept of variable dimension. Additionally, we record several auxiliary lemmas concerning geometric doubling and covering properties that serve as prerequisites for our main results. Since much of this material is included for completeness and ease of reference, the reader may return to this section as needed.

\subsection{Notation}
Open (metric) balls in a given metric space, $(X,d)$, shall be denoted by $B(x,r)=\{y\in X:\, d(x,y)<r\}$ while the notation $\overline{B}(x,r)=\{y\in X:\, d(x,y)\leq r\}$ will be used 
for closed balls. Note that in general $\overline{B}(x,r)$ is not necessarily equal to the topological closure of $B(x,r)$. We allow the radius of a ball to  equal zero. If $r=0$, then
$B(x,r)=\emptyset$, but $\overline{B}(x,r)=\{ x\}$. We use $\mathbb{Z}$ and $\mathbb Q$ to denote the sets of integers and rational numbers, respectively. We also write
$\mathbb{N}:=\{1,2,\dots\}$, set $\mathbb{N}_0:=\mathbb{N}\cup\{ 0\}$, and denote by $\overline{\mathbb R}$ the extended real line $[-\infty,\infty]$. The characteristic function of a set $E$ will be denoted by $\chi_E$, and $\# E$ will stand for the cardinality of $E$.
The integral average of a  measurable function $u$ on a measurable set
$E\subseteq X$ with $\mu(E)\in(0,\infty)$ will be denoted by
$$
u_E:=\fint_Eu\, d\mu :=\frac{1}{\mu(E)}\int_E u\, d\mu,
$$	
whenever the integral is well defined, and for any non-empty set $E\subseteq X$, we let
$$
\diam E:=\sup\{d(x,y):\, x,y\in E\}.
$$

\subsection{Metric measure spaces}
We say that the triple $(X,d,\mu)$ is a \texttt{metric measure space} if $(X,d)$ is a metric space and $\mu$ is a non-negative  Borel measure on $X$ such that for every $x\in X$ and $r\in (0,\infty)$,
\begin{equation}\label{ballsmeas}
	0< \mu(B(x,r)) < \infty.
\end{equation}
It is widely known (see \cite{G}) that the condition \eqref{ballsmeas} is equivalent to the separability of the metric space $(X,d)$. Thus, all metric measure spaces considered in this article are separable.

Now we shall introduce the concept of metric measure spaces with variable dimension. Such spaces were first studied in \cite{HHL06}. Let $(X,d,\mu)$ be a metric measure space and assume that $Q: X \rightarrow [0,\infty)$ is a bounded measurable function. We say that the measure $\mu$ is \texttt{lower Ahlfors $Q(\cdot)$-regular}, if there exists $b_1\in(0,\infty)$ such that for all $x\in X$ and $r\in (0,1]$,
\begin{equation*}
	\mu\left(B(x,r)\right) \geq b_1r^{Q(x)}.
\end{equation*}
Similarly, we say that $\mu$ is \texttt{upper Ahlfors $Q(\cdot)$-regular} if there is $b_2\in(0,\infty)$ such that
\begin{equation*}
	b_2r^{Q(x)} \geq \mu\left(B(x,r)\right)
\end{equation*}
for every $x\in X$ and $r\in (0,1]$. Finally, we say that $\mu$ is \texttt{$Q(\cdot)$-Ahlfors regular} if it is both upper and lower Ahlfors $Q(\cdot)$-regular. Note that if $Q$ is is essentially bounded away from 0 and $\infty$ then the upper bound 1 for the radii in the lower and upper Ahlfors $Q(\cdot)$-regularity conditions can be replaced by any strictly positive and finite threshold; see Remark~\ref{measurethreshold}.

Next, let us recall the notion of doubling metric spaces. If $(X,d,\mu)$ is a metric measure space, we say that the measure $\mu$ is \texttt{doubling} if there exists $C_d\in[1,\infty)$ such that
\begin{equation*}
	\mu\left(B(x,2r)\right)\leq C_d\mu\left(B(x,r)\right),
\end{equation*}
for every ball $B(x,r)\subseteq X$.
Furthermore, we say that metric space $(X,d)$ is \texttt{geometrically doubling} if there exists constant $M\in(0,\infty)$, called \texttt{the doubling constant of $X$}, such that every ball $B(x,r) \subseteq X$ can be covered by at most $M$ balls having radius $r/2$. It is well known that a metric space equipped with a doubling measure is geometrically doubling (see, e.g., \cite{Coifman}). On the other hand, in \cite{lukkainen} it was proven that every complete geometrically doubling metric space carries a doubling measure.

Now we prove a covering characterization of geometrically doubling metric spaces, which shall be useful in the proofs of the main results in this paper.

\begin{lemma}\label{covering}
Let $(X,d)$ be a metric space. Then, the following two statements are equivalent.
\begin{enumerate}
\item[(i)] The space $(X,d)$ is geometrically doubling.
\item[(ii)] There exist constants $A, B\in[0,\infty)$ such that for every $r\in (0,\infty)$, there are $I\subseteq \mathbb N$ and a sequence $\left\{x_i\right\}_{i\in I}\subseteq X$ such that the balls $\left\{B(x_i,r)\right\}_{i\in I}$ cover $X$ and have the additional property that, for every $R\in (r,\infty)$, each point $x\in X$ belongs to at most $\displaystyle A\left(R/r\right)^{B}$ balls $B(x_i,R)$.
\end{enumerate}
\end{lemma}
\begin{proof}
	To begin, we will assume that $(X,d)$ possesses the geometric doubling property. Fix $r\in (0,\infty)$. As previously mentioned, it is well known that geometrically doubling spaces are separable. Let $\left\{x_i\right\}_{i\in I}$ be a maximal $r/2$-separated subset of $X$. Then
	\begin{equation*}
		X=\bigcup_{i\in I} B(x_i,r).
	\end{equation*}
	Now, let $R\in (r,\infty)$ and $x\in X$. We aim to prove that $x$ lies within at most $\displaystyle M^3\left(R/r \right)^{\log_2 M}$ balls $B(x_i,R)$, where $M>0$ is a doubling constant of $X$. Since $\left\{x_i\right\}_{i\in I}$ is $r/2$-separated, it suffices to derive an upper bound for the cardinality of any $r/2$-separated subset of $B(x,R)$. Let $p\in\mathbb{N}$ be the smallest positive integer such that $R/r<2^p$. The geometric doubling property of $X$ ensures that 
	\begin{equation*}
		B(x,R) \subseteq \bigcup_{i=1}^{M^{p+2}} B\left(y_i,r/4\right)
	\end{equation*}
	for some $y_i\in X$ and $i=1,\dots,M^{p+2}$. Hence, the cardinality of every $r/2$-separated subset of $B(x,R)$ is at most $M^{p+2}$. Moreover, since $R/r \geq 2^{p-1}$, we conclude that $p-1\leq \log_2(R/r)$. Thus, it follows
	\begin{equation*}
		M^{p+2}\leq M^3 M^{\log_2(R/r)}=M^3 e^{\log M \frac{\log (R/r)}{\log 2}}=M^3\left(R/r\right)^{\log_2 M},
	\end{equation*}
	which completes the first part of the proof.
	
	We now proceed to show that $(ii)$ implies $(i)$. Let $r\in (0,\infty)$ and $x\in X$  be fixed and let $A,B$ be the constants as in condition $(ii)$. Then, we can find $I \subseteq \mathbb N$ and the sequence $\left\{x_i\right\}_{i\in I}\subseteq X$ such that the balls $\left\{B(x_i,r)\right\}_{i\in I}$ cover $X$. Additionally, for $R:=3r$ there is a set $S\subseteq I$ such that $\# S \leq A3^B$ and 
	\begin{equation*}
		x \notin \bigcup_{i\in I \setminus S} B(x_i,3r).
	\end{equation*}
	We claim that
	\begin{equation*}
		B(x,2r) \subseteq \bigcup_{i\in S} B(x_i,r).
	\end{equation*}
	Indeed, let $y$ be any point in $B(x,2r)$. If $y\notin \displaystyle\bigcup_{i\in S} B(x_i,r)$, then since the collection $\left\{B(x_i,r)\right\}_{i\in I}$ cover $X$, we can find the index $i\in I\setminus S$ such that $d(y,x_i)<r$. But on the other hand
	\begin{equation*}
		d(x,x_i) \leq d(x,y)+d(y,x_i)<2r+r=3r,
	\end{equation*}
	which contradicts the fact that $d(x,x_i)\geq 3r$. Therefore, the inclusion stated in our claim holds, confirming that $X$ is geometrically doubling with the constant $M\leq A3^B$.
\end{proof}

\subsection{Auxiliary lemmas}
\begin{lemma}\label{pokrycieprodukt}
	Let $(X,d)$ be separable metric space and $\delta\in(0,\infty)$. Then, there exists a sequence $\left\{z_i\right\}_{i=1}^\infty\subseteq X$ such that the balls $\left\{B(z_i,\delta/2)\right\}_{i=1}^{\infty}$ cover $X$ and
	\begin{equation*}
		\left\{ (x,y)\in X \times X: 0<d(x,y)<\delta/4 \right\} \subseteq \bigcup_{i=1}^{\infty} \left\{ (x,y)\in X \times X: x,y\in B\left(z_i,\delta/2\right) \textnormal{ and } x \neq y\right\}.
	\end{equation*}
\end{lemma}

\begin{proof}
	By separability of $X$, there exists sequence $\left\{z_i\right\}_{i=1}^{\infty} \subseteq X$ such that
	\begin{equation}\label{cover}
		X=\bigcup_{i=1}^{\infty} B\left(z_i,\delta/4\right).
	\end{equation}
	Let $x,y\in X$ be such that $\displaystyle 0<d(x,y)<\delta/4$. Due to \eqref{cover}, there exists $j\in \mathbb N$ such that $\displaystyle x\in B\left(z_j,\delta/4\right)$. Moreover, by triangle inequality
	\begin{equation*}
		d(y,z_j)\leq d(y,x)+d(x,z_j) < \frac{\delta}{4}+\frac{\delta}{4}=\frac{\delta}{2},
	\end{equation*}
	which means that $\displaystyle x,y\in B\left(z_j,\delta/2\right)$, where $x\neq y$, and proof is complete.
\end{proof}

\begin{lemma}\cite[Lemma 16]{AGH20}\label{iterative lemma}
	Suppose $0<a<b<\infty,$ $0<p<q<\infty$ and $\rho,\tau\in (0,\infty).$ If a sequence $\{a_j\}_{j=1}^\infty$ satisfies
	$$a\leq a_j\leq b\quad\text{and}\quad a_{j+1}^{\frac{1}{q}}\leq \rho\tau^ja_j^{\frac{1}{p}},$$
	for all $j\in\mathbb{N}$. Then $a_1^{1-p/q}\rho^p\tau^{pq/(q-p)}\geq 1.$
\end{lemma}

The following result may be well known but for the sake of clarity we state it with the proof. 
\begin{lemma}\label{zero}
	Let $(X,d,\mu)$ be a metric measure space and fix $B_0:=B(x_0, r_0)$, where $x_0 \in X$ and $r_0 \in (0,\infty)$. If $u:X \to \overline{\mathbb R}$ is finite $\mu$-almost everywhere, then there exists a null set $N\subseteq B_0$ such that, for every $c\in \mathbb R$ and every $x, y \in B_0 \setminus N$,
	\begin{equation*}
		\left|u(x)-u(y)\right|\leq 2\left\|u-c\right\|_{L^{\infty}(B(x,2d(x,y)))}.
	\end{equation*}
\end{lemma}
\begin{proof}Since the measure of every ball is positive and finite we have that $X$ is separable (see \cite{G}). Therefore for every $j \in \mathbb{N}$ there exists $\left\{x_{i,j}\right\}_{i=1}^\infty$ such that
	\begin{equation*}
		B_0 =\bigcup_{i=1}^\infty B_{i,j},
	\end{equation*}
	where $B_{i,j} :=B(x_{i,j},1/j) \cap B_0$.
	
	Next, for $q \in \mathbb{Q}$, $i, j \in \mathbb{N}$ there exists a null set $N_{q,i,j} \subseteq B_{i,j}$ such that, for all $x \in B_{i,j} \setminus N_{q,i,j}$,
	\begin{equation*}
		\left|u(x)-q\right| \leq \left\|u-q\right\|_{L^{\infty}\left(B_{i,j}\right)}.
	\end{equation*}
	Next, if we define the following null set
	\begin{equation*}
		N:=\bigcup_{q \in \mathbb{Q}}\bigcup_{i, j \in \mathbb{N}} N_{q,i,j} \cup \{x \in B_0: \left|u(x)\right|= \infty\}
	\end{equation*}
	then, for every $q \in \mathbb{Q}$, $ i,j \in \mathbb{N}$, and $x \in B_{i,j} \setminus N$, we have
	\begin{equation*}
		\left|u(x)-q\right| \leq \left\|u-q\right\|_{L^{\infty}(B_{i,j})}.
	\end{equation*}
	Let us fix $x,y \in B_0 \setminus N$, $x\neq y$, and $\varepsilon \in(0,\infty)$. Then, there exists $q \in \mathbb{Q}$ such that
	\begin{equation*}
		\left|c-q\right| < \varepsilon.
	\end{equation*}
	Let $j \in \mathbb{N}$ be such that $1/j < d(x,y)/2$ and $i_x, i_y \in \mathbb{N}$ be such that $x\in B_{i_x,j}$ and $y\in B_{i_y,j}$. Then
	\begin{equation*}
		B_{i_x,j} \cup B_{i_y,j} \subseteq B(x,2d(x,y)).
	\end{equation*}
	Hence,
	\begin{align*}
		\left|u(x) -u(y) \right| &\leq  \left|u(x)-q \right| + \left|u(y)-q \right|\\
		&\leq    \left\|u-q\right\|_{L^{\infty}\left(B_{i_x,j}\right)} + \left\|u-q\right\|_{L^{\infty}\left(B_{i_y,j}\right)}\\
		&\leq  2\left\|u-q\right\|_{L^{\infty}\left(B(x,2d(x,y))\right)}\\
		&\leq 2\left\|u-c\right\|_{L^{\infty}\left(B(x,2d(x,y))\right)} + 2\left|q-c\right|\\
		&\leq 2\left\|u-c\right\|_{L^{\infty}\left(B(x,2d(x,y))\right)} + 2\varepsilon.
	\end{align*}
	Since $\varepsilon$ was arbitrary, the proof follows.
\end{proof}

\section{Variable Exponent Theory}\label{sect:varexptheory}

This section establishes the analytic foundations for spaces with variable integrability. We summarize the properties of semi-modular spaces and define the variable exponent Lebesgue and H\"older spaces. Significant focus is placed on the regularity of exponents, notably log-H\"older continuity, which is essential for the boundedness of fundamental operators. We also discuss the median operator and its behavior within the variable exponent setting. Similar to Section~\ref{sect:preliminaries}, the results recorded here are largely standard and thus, the reader may safely skip to the subsequent results and return to these reference materials on an as-needed basis.

\subsection{Modular spaces}
Let $V$ be a vector space over $\mathbb K \in \left\{ \mathbb R, \mathbb C\right\}.$ The function $\rho: V \to [0,\infty]$ is called a \texttt{semi-modular} on $V$ if the following conditions are satisfied:
\begin{enumerate}
	\item[(i)] $\rho(0)=0$;
	\item[(ii)] $\rho\left(\lambda v\right)=\rho(v)$ for all $v\in V$ and $\lambda \in \mathbb K$ with $\left|\lambda\right|=1$;
	\item[(iii)] If $\rho\left(\lambda v\right)=0$ for all $\lambda\in(0,\infty)$, then $v=0$;
	\item[(iv)] The function $[0,\infty)\ni \lambda \mapsto \rho\left(\lambda v\right)$ is left-continuous for each fixed $v\in V$;
	\item[(v)] The function $[0,\infty) \ni \lambda \mapsto \rho\left(\lambda v\right)$ is non-decreasing for each fixed $v\in V$.
\end{enumerate}
In the literature, condition $\rm (v)$ is  often replaced by a convexity condition on $\rho$, which is a stronger assumption. If $\rho$ is a semi-modular on a vector space $V$, then the space
\begin{equation*}
	V_{\rho}:=\left\{v\in V: \exists_{\lambda\in(0,\infty)}\, \rho\left(\lambda v\right)<\infty \right\}
\end{equation*}
called a \texttt{semi-modular space}. We endow  $V_\rho$ with the functional $\left\|\cdot \right\|_{\rho}: V_{\rho} \to [0,\infty]$, which is given by the formula
\begin{equation*}
\left\|v\right\|_{\rho}:=\inf\left\{\lambda\in(0,\infty): \rho\left(v/\lambda\right) \leq 1\right\},
\end{equation*}
for $v\in V_{\rho}$. If $\rho$ is a convex semi-modular, then $\left\|\cdot\right\|_{\rho}$ defines a norm on $V_{\rho}$ (called the \texttt{Luxemburg norm}).

\begin{prop}\label{ballprop}\cite{DG}
Let $\rho$ be a semi-modular on vector space $V$. Then, for every $v\in V_{\rho}$,
\begin{equation*}
\rho(v) \leq 1\, \text{ if and only if }\, \left\|v\right\|_{\rho} \leq 1.
\end{equation*}
\end{prop}

\subsection{Variable exponent Lebesgue spaces}
Let $(X,\mu)$ be a measure space. Given a measurable function $p: X \to (0,\infty]$ and a measurable set $E\subseteq X$, we shall use notation
\begin{equation*}
	p_E^-:= \essinf_{x\in E} p(x)\quad \text{and}\quad p_E^+:=\esssup_{x\in E} p(x).
\end{equation*}
If $E=X$, then we shall write $p^+:=p_X^+$, $p^-:=p_X^-$. We say that $p$ is a \texttt{variable exponent} on $X$ if $p^->0$. The class of all variable exponents on $X$ shall be denoted by $\mathcal{P}(X)$. For two exponents $p,q\in \mathcal{P}(X)$, we shall write $p \geq q$ if $p(x)\geq q(x)$ for all $x\in X$. Moreover, we define the class of (essentially) bounded exponents as $\mathcal{P}_b(X):=\mathcal{P}(X) \cap L^{\infty}(X,\mu)$. For two exponents $p,q\in \mathcal{P}_b(X)$, we shall write $p \gg q$ if $(p-q)^- >0$.

We now recall the notion of variable exponent Lebesgue spaces $L^{p(\cdot)}(X,\mu)$. Given $p\in \mathcal{P}_b(X)$,  we say that $u\in L^{p(\cdot)}(X,\mu)$ if and only if $u$ is a measurable real-valued function defined on $X$ and the following semi-modular is finite
\begin{equation*}
	\rho_{p(\cdot)}(u):= \int_X \left|u(x)\right|^{p(x)} \mbox{d}\mu(x).
\end{equation*}
Note that functions belonging to $L^{p(\cdot)}(X,\mu)$ are necessarily finite pointwise almost everywhere in $X$. Moreover, the space $L^{p(\cdot)}(X,\mu)$ is a quasi-Banach space when equipped with the Luxemburg quasi-norm
\begin{equation*}
	\left\| u \right\|_{L^{p(\cdot)}(X,\mu)} := \inf\left\{ \lambda\in(0,\infty): \rho_{p(\cdot)}\left(\frac{u}{\lambda}\right) \leq 1\right\},
\end{equation*}
 where  $u\in L^{p(\cdot)}(X,\mu)$.
Throughout this article, we denote by $\kappa_{p(\cdot)}\in[1,\infty)$ the constant appearing in the quasi-triangle inequality satisfied by the above quasi-norm. If $p^- \geq 1$, then $L^{p(\cdot)}(X,\mu)$ is a Banach space. Additionally, if $p$ is constant, then the variable exponent space coincides with the ordinary Lebesgue space $L^{p}(X,\mu)$ consisting of all $p$-integrable functions on $X$. Lastly, we will simply use $L^{p(\cdot)}(X)$ in place of $L^{p(\cdot)}(X,\mu)$ whenever the measure is well understood from the context. 

Below we recall some known and useful properties of the semi-modular $\rho_{p(\cdot)}$ and the corresponding Luxemburg quasi-norm.
\begin{prop}\label{rel}
	Let $(X,\mu)$ be a measure space and let $p\in \mathcal{P}_b(X)$. Then, for every $u\in L^{p(\cdot)}(X,\mu)$,
	\begin{equation*}
		\min \left\{ \rho_{p(\cdot)}(u)^{\frac{1}{p^-}} , \rho_{p(\cdot)}(u)^{\frac{1}{p^+}} \right\} \leq \left\| u \right\|_{L^{p(\cdot)}(X,\mu)} \leq \max \left\{\rho_{p(\cdot)}(u)^{\frac{1}{p^-}} , \rho_{p(\cdot)}(u)^{\frac{1}{p^+}} \right\}.
	\end{equation*}
\end{prop}	

In the variable exponent Lebesgue space, the following H\"older inequality holds.

\begin{prop}\label{holder}
	Let $(X,\mu)$ be a measure space and let $p \in \mathcal{P}_b(X)$ be such that $p^- > 1$. Suppose that $p'\in \mathcal{P}_b(X)$ is a conjugate variable exponent to $p$, that is, for each $x\in X$,
	\begin{equation*}
	\frac{1}{p(x)}+\frac{1}{p'(x)}=1.
	\end{equation*}
	Then, for every $f\in L^{p(\cdot)}(X,\mu)$ and $g\in L^{p'(\cdot)}(X,\mu)$ we have
	\begin{equation*}
		\int_X \left|f(x)g(x)\right|\mbox{d}\mu(x) \leq 2\left\| f\right\|_{L^{p(\cdot)}(X,\mu)} \left\|g\right\|_{L^{p'(\cdot)}(X,\mu)}.
	\end{equation*}
\end{prop}

\begin{lemma}\label{wlozenielp}
Let $(X,\mu)$ be a measure space with finite measure and fix $p,q\in \mathcal{P}_b(X)$ such that $q \gg p$. Then, $L^{q(\cdot)}(X,\mu) \subseteq L^{p(\cdot)}(X,\mu)$ and for every $u\in L^{q(\cdot)}(X,\mu)$, 
\begin{equation*}
\left\| u \right\|_{L^{p(\cdot)}(X,\mu)} \leq 2^{\frac{1}{p^-}}\max\left\{\left\|1 \right\|_{L^{t'(\cdot)}(X,\mu)}^{\frac{1}{p^+}}, \left\|1 \right\|_{L^{t'(\cdot)}(X,\mu)}^{\frac{1}{p^-}}\right\} \left\| u \right\|_{L^{q(\cdot)}(X,\mu)},
\end{equation*}
where $\displaystyle t:=q/p$.
\end{lemma}

\begin{proof}
Let $u\in L^{q(\cdot)}(X,\mu)$. We can assume that $\left\|u \right\|_{L^{q(\cdot)}(X,\mu)}=1$. Let $t:=q/p$. Now, by the H\"older inequality (Proposition~\ref{holder}) and Proposition~\ref{rel} we get
\begin{align*}
\rho_{p(\cdot)}(u) &\leq 2\left\| u^{p} \right\|_{L^{t(\cdot)}(X,\mu)} \left\|1\right\|_{L^{t'(\cdot)}(X,\mu)}\leq 2 \max\left\{ \rho_{q(\cdot)}(u)^{\frac{1}{t^-}}, \rho_{q(\cdot)}(u)^{\frac{1}{t^+}}\right\} \left\|1 \right\|_{L^{t'(\cdot)}(X,\mu)} \leq 2\left\|1 \right\|_{L^{t'(\cdot)}(X,\mu)}.
\end{align*}
Hence, $u\in L^{p(\cdot)}(X,\mu)$ and by using  Proposition~\ref{rel} again, we get
\begin{align*}
\left\|u\right\|_{L^{p(\cdot)}(X,\mu)} \leq 2^{\frac{1}{p^-}}\max\left\{\left\|1 \right\|_{L^{t'(\cdot)}(X,\mu)}^{\frac{1}{p^+}}, \left\|1 \right\|_{L^{t'(\cdot)}(X,\mu)}^{\frac{1}{p^-}}\right\},
\end{align*}
which completes the proof.
\end{proof}

\begin{prop}\label{interpolacyjny}
Let $(X,\mu)$ be a measure space and suppose that $E(X,\mu)$ is a quasi-normed subspace of the space $L^0(X,\mu)$ consisting of all (equivalence classes of) measurable functions on $X$ which are finite pointwise almost everywhere. Let $q_0,q,q_1\in \mathcal{P}_b(X)$ be such that $q_0 \leq q \leq q_1$. If the following continuous embeddings hold
\begin{equation*}
	E(X,\mu) \hookrightarrow L^{q_0(\cdot)}(X,\mu), \hspace{15mm} 
	E(X,\mu) \hookrightarrow L^{q_1(\cdot)}(X,\mu),
\end{equation*}
then the continuous embedding
\begin{equation*}
	E(X,\mu) \hookrightarrow L^{q(\cdot)}(X,\mu)
\end{equation*}
also holds.
\end{prop}

\begin{proof}
Let $\left\{u_n\right\}_{n=1}^\infty\subseteq E(X,\mu)$ be any sequence converging to $0$ in $E(X,\mu)$. Since $q_0\leq q\leq q_1$, it follows that for $\mu$-almost every $x\in X$ and each $n\in \mathbb N$, we have
	\begin{equation}\label{interpol}
		\left|u_n(x)\right|^{q(x)}\leq \left|u_n(x)\right|^{q_0(x)}+\left|u_n(x)\right|^{q_1(x)}.
	\end{equation}
Integrating both sides of inequality \eqref{interpol} over $X$, we obtain
	\begin{equation}\label{interpol2}
		\rho_{q(\cdot)}(u_n)\leq \rho_{q_0(\cdot)}(u_n)+\rho_{q_1(\cdot)}(u_n).
	\end{equation}
Given the assumptions, we know that $u_n \to 0$ in $L^{q_0(\cdot)}(X,\mu)$ and $u_n \to 0$ in $L^{q_1(\cdot)}(X,\mu)$ as $n\to \infty$. Hence $\rho_{q_0(\cdot)}(u_n) \to 0$ and $\rho_{q_1(\cdot)}(u_n) \to 0$ as $n\to\infty$. Passing to the limit with $n \to \infty$ in inequality \eqref{interpol2}, we deduce that $\rho_{q(\cdot)}(u_n) \to 0$, and therefore $u_n \to 0$ in $L^{q(\cdot)}(X,\mu)$. 
\end{proof}

\subsection{Variable exponent H\"older spaces}

In this section, we recall the definition of the variable exponent H\"older space. Let $(X,d)$ be a metric space. By $C(X,d)$ we denote the space of continuous functions $u:X \to \mathbb{R}$ such that the norm
\begin{equation*}
	\left\|u\right\|_{C(X,d)}:= \sup_{x\in X}\left|u(x)\right|
\end{equation*}
is finite. Moreover, given a bounded function $\alpha : X \to [0,\infty)$, we denote by $\dot{C}^{0,\alpha(\cdot)}(X,d)$ the \texttt{homogeneous variable exponent H\"older space}, i.e., the space of all continuous functions $u: X \to \mathbb{R}$ such that
\begin{equation*}
	\left[ u \right]_{\alpha(\cdot),X}:= \sup_{\substack{x,y\in X \\ x\neq y}} \frac{\left|u(x)-u(y)\right|}{d(x,y)^{\alpha(x)}} < \infty.
\end{equation*}
Now, we define \texttt{inhomogeneous variable exponent H\"{o}lder space} as
\begin{equation*}
	C^{0,\alpha(\cdot)}(X,d):= C(X,d) \cap \dot{C}^{0,\alpha(\cdot)}(X,d).
\end{equation*} 
The H\"older space $C^{0,\alpha(\cdot)}(X,d)$ is a Banach space when equipped with the norm
\begin{equation*}
	\left\| u  \right\|_{C^{0,\alpha(\cdot)}(X,d)} := \left\| u \right\|_{C(X,d)}  + \left[u\right]_{\alpha(\cdot),X}, 
\end{equation*}
where  $u \in C^{0,\alpha(\cdot)}(X,d)$.

\subsection{Log-H\"older continuity}

\begin{defi}
	Let $(X,d,\mu)$ be a metric measure space and $\Omega\subseteq X.$ We say that $p:\Omega\rightarrow\mathbb{R}$ is \texttt{locally log-H\"{o}lder continuous on $\Omega$} if there exists a constant $C_{\log}(p)>0$ such that
	\begin{equation}\label{logholder}
		\left| p(x) - p(y) \right| \leq \frac{C_{\log}(p)}{\log (e + 1 / d(x,y))} 
	\end{equation}
	for all $x,y\in\Omega.$ 
\end{defi}
We also introduce the following class of variable exponents
\begin{equation*}
	\mathcal{P}^{\log}(\Omega):=\left\{p\in \mathcal{P}(\Omega): 1/p \,\text{ is locally log-H\"older continuous on }\Omega \right\},
\end{equation*}
and  set $\mathcal{P}_b^{\log}(X):=\mathcal{P}^{\log}(X) \cap \mathcal{P}_b(X)$. It is a well known fact  that  $\displaystyle p \in \mathcal{P}_b^{\log}(X)$ if and only if $\displaystyle 1/p \in \mathcal{P}_b^{\log}(X)$.

\begin{lemma}\label{loglemma}
	Let $(X,d,\mu)$ be a metric measure space and $B:=B(z,r)$, where $z \in X$ and $r\in (0,\infty)$. Assume that $t\in \mathcal{P}^{\log}_b(B)$. Then, the following statements are true.
	\begin{enumerate} 			\item[(i)] For any $R\in[2r,\infty)$ and $x \in B$, the following inequality holds
	\begin{equation*}
		e^{-C_{\log}(1/t)}\left(\frac{1}{R}\right)^{\frac{1}{t_B^-}} \leq \left(\frac{1}{R}\right)^{\frac{1}{t(x)}}\leq e^{C_{\log}(1/t)}\left(\frac{1}{R}\right)^{\frac{1}{t_B^+}}.
	\end{equation*}
	\item[(ii)] One has
	\begin{equation*}
		r^{\frac{1}{t_B^+} - \frac{1}{t_B^-}}\leq e^{C_{\log}(1/t)} 2^{\frac{1}{t_B^-}-\frac{1}{t_B^+}}.
	\end{equation*}
	\item[(iii)] There exists a constant $M(r,t)\geq 1$, depending only on $r,t$ and log-H\"older constant of $1/t$, such that for every $x,y\in B$,
	\begin{equation*}
		\frac{1}{M(r,t)}d(x,y)^{\frac{1}{t(y)}} \leq d(x,y)^{\frac{1}{t(x)}} \leq M(r,t)d(x,y)^{\frac{1}{t(y)}}.
	\end{equation*}
	\end{enumerate}
\end{lemma}
\begin{proof}
	$(i)$ Let $R\in[2r,\infty)$. For any $x, y \in B$, we have
	\begin{equation*}
		\left| \frac{1}{t(x)}- \frac{1}{t(y)}\right|\leq \frac{C_{\log}(1/t)}{\log(e+1/d(x,y))} \leq \frac{C_{\log}(1/t)}{\log(e+1/R)}.
	\end{equation*}
	Hence, for $x\in B$ we get
	\begin{equation}\label{logestim}
		\left( \frac{1}{t(x)}- \frac{1}{t_B^+}\right)\log(e+1/R) \leq C_{\log}(1/t),
	\end{equation}
	and we obtain
	\begin{equation*}
		\left(1/R\right)^{\left( \frac{1}{t(x)}- \frac{1}{t_B^+}\right)} \leq \left(e+1/R\right)^{\left( \frac{1}{t(x)}- \frac{1}{t_B^+}\right)} \leq e^{C_{\log}(1/t)}.
	\end{equation*}
	Similarily, we have
	\begin{equation*}
		\left( \frac{1}{t_B^-}- \frac{1}{t(x)}\right)\log(e+1/R) \leq C_{\log}(1/t),
	\end{equation*}
	and finally
	\begin{equation*}
		\left(1/R\right)^{\left( \frac{1}{t_B^-}- \frac{1}{t(x)}\right)} \leq \left(e+1/R\right)^{\left( \frac{1}{t_B^-}- \frac{1}{t(x)}\right)} \leq e^{C_{\log}(1/t)}.
	\end{equation*}
	The desired estimates in $(i)$ now follow.
	
	$(ii)$ From \eqref{logestim} for $R:=2r$ we get 
	\begin{equation*}
		\left( \frac{1}{t_B^-}- \frac{1}{t_B^+}\right)\log(e+1/2r) \leq C_{\log}(1/t).
	\end{equation*}
	Hence,
	\begin{equation*}
		\left(1/r\right)^{\left( \frac{1}{t_B^-}- \frac{1}{t_B^+}\right)} \leq e^{C_{\log}(1/t)} 2^{ \frac{1}{t_B^-}- \frac{1}{t_B^+}}.
	\end{equation*}
	
	$(iii)$ By the assumption $t\in \mathcal{P}^{\log}_b(B)$, we know that for all $x,y\in B$, $x\neq y$, it holds that
	\begin{equation*}
		\left|\frac{1}{t(x)}-\frac{1}{t(y)}\right|\leq \frac{C_{\textnormal{log}}(1/t)}{\log\left(e+1/d(x,y)\right)}.
	\end{equation*}
	Hence, for all $x,y\in B$ we have that
	\begin{equation*}
		\left(\frac{1}{d(x,y)}\right)^{\left|\frac{1}{t(x)}-\frac{1}{t(y)}\right|}\leq\left(e+\frac{1}{d(x,y)}\right)^{\left|\frac{1}{t(x)}-\frac{1}{t(y)}\right|}\leq e^{C_{\textnormal{log}}(1/t)},
	\end{equation*}
	and since $d(x,y)\leq 2r$, it follows that
	\begin{equation*}
		e^{-C_{\textnormal{log}}(1/t)}\leq d(x,y)^{\left|\frac{1}{t(x)}-\frac{1}{t(y)}\right|}\leq \left(2r\right)^{\left|\frac{1}{t(x)}-\frac{1}{t(y)}\right|}\leq \max\left\{1,\left(2r\right)^{\frac{2}{t_B^-}}\right\}.
	\end{equation*}
	Let $M(r,t):=\max\left\{1,\left(2r\right)^{\frac{2}{t_B^-}},e^{C_{\log}(1/t)}\right\}$. Then, for all $x,y\in B$ such that $t(x)\leq t(y)$ we obtain
	\begin{equation*}
		M(r,t)^{-1}d(x,y)^{\frac{1}{t(y)}}\leq d(x,y)^{\frac{1}{t(x)}}\leq M(r,t) d(x,y)^{\frac{1}{t(y)}}. 
	\end{equation*}
	Therefore, having in mind the symmetry of the metric $d$, the proof follows.
\end{proof}

\subsection{Median operator in variable exponent spaces}

\begin{defi}\label{med}
	Let $(X,\mu)$ be a measure space and let $E\subseteq X$ be a measurable set such that $\mu(E)\in (0,\infty)$. Let $u: E \to \overline{\mathbb R}$ be a measurable function. The \texttt{median of $u$ on $E$} is defined as
	\begin{equation*}
		m_u(E):=\sup\left\{t\in \mathbb R: \mu\left(\left\{x\in E: u(x)<t\right\}\right)\leq \frac{\mu(E)}{2}\right\}.
	\end{equation*}
\end{defi}

\begin{prop}\cite{GS}
	Let $(X,\mu)$ be a measure space and let $E\subseteq X$ be a measurable set such that $\mu(E)\in (0,\infty)$. Let $u: E \to \overline{\mathbb{R}}$ be finite $\mu$-almost everywhere. Then,
	\begin{equation*}
		m_u(E)=\max\left\{t\in \mathbb R: \mu\left(\left\{x\in E: u(x)<t\right\}\right)\leq \frac{\mu(E)}{2}\right\}.
	\end{equation*}
\end{prop}

\begin{prop}\label{med1}\cite{GS}
	Let $(X,\mu)$ be a measure space and let $E\subseteq X$ be a measurable set such that $\mu(E)\in (0,\infty)$. Let $u:E \to \overline{\mathbb R}$ be measurable which is finite $\mu$-almost everywhere. Then, the following statements are true.
	\begin{enumerate}
		\item[(i)] If $c\in \mathbb R$, then $m_u(E)-c=m_{u-c}(E)$.
		\item[(ii)] If $c\in(0,\infty)$, then $c\cdot m_u(E)=m_{c\cdot u}(E)$.
		\item[(iii)] $\left|m_u(E)\right| \leq m_{\left|u\right|}(E)$.
	\end{enumerate}
\end{prop}

\begin{prop}\label{medianlemma} Let $(X,\mu)$ be a measure space and let $E\subseteq X$ be a measurable set such that $\mu(E)\in (0,\infty)$. Let $u:E \to \overline{\mathbb R}$ be a measurable function which is finite $\mu$-almost everywhere. Then, for every $p\in \mathcal{P}_b(X)$ and $c\in \mathbb R$, one has
	\begin{equation*}
		\left|m_u(E)-c\right|\leq \max\left\{2,\left(\frac{2}{\mu(E)}\right)^{\frac{1}{p^-_E}}\right\}\left\|u-c\right\|_{L^{p(\cdot)}(E)}. 
	\end{equation*}
\end{prop}

\begin{proof} We can assume that $u\in L^{p(\cdot)}(E)$, since otherwise the claim is obvious. In view of Proposition~\ref{med1} it suffices to prove 
	\begin{equation*}
		m_{\left|u-c\right|}(E) \leq \max\left\{2,\left(\frac{2}{\mu(E)}\right)^{\frac{1}{p^-_E}}\right\}\left\|u-c\right\|_{L^{p(\cdot)}(E)}.
	\end{equation*}
	Let $\displaystyle \eta>\max\left\{2,\left(\frac{2}{\mu(E)}\right)^{\frac{1}{p^-_E}}\right\}$ and $\displaystyle \delta:=\left\|u-c\right\|_{L^{p(\cdot)}(E)}.$ We may assume that $\delta \in (0,\infty)$. 
	By the virtue of the Chebyshev inequality and Proposition~\ref{rel} we obtain
	\begin{align*}
		\mu\left( \left\{x\in E: \left|u(x)-c\right| \geq \eta \delta \right\}\right) \leq \rho_{p(\cdot)}\left(\frac{\left|u-c\right|}{\eta \delta}\chi_{E}\right) \leq \left\|\frac{u-c}{\eta\delta}\right\|_{L^{p(\cdot)}(E)}^{p^-_E}=\frac{1}{\eta^{p^-_E}}<\frac{\mu(E)}{2}.
	\end{align*} 
	Hence
	\begin{equation*}
		m_{\left|u-c\right|}(E)\leq \eta \left\|u-c\right\|_{L^{p(\cdot)}(E)}.
	\end{equation*}
	Passing to the limit with $\displaystyle \eta \to \max\left\{2,\left(\frac{2}{\mu(E)}\right)^{\frac{1}{p^-_E}}\right\}$ we obtain the claim of Proposition~\ref{medianlemma}.
\end{proof}

\section{Variable Exponent Sobolev, Triebel-Lizorkin, and Besov Spaces}\label{sect:variableexponentspaces}

In this section, we introduce the primary function spaces of interest: the variable exponent Haj\l{}asz--Sobolev, Haj\l{}asz--Triebel--Lizorkin, and Haj\l{}asz--Besov spaces. We first discuss mixed Lebesgue-sequence spaces and then provide the formal definitions for the aforementioned variable exponent spaces.

\subsection{Mixed Lebesgue sequence spaces}
Let $(X,\mu)$ be a measure space and $p\in \mathcal{P}_b(X)$, $q\in \mathcal{P}(X)$. For every sequence $\left\{u_k\right\}_{k\in \mathbb Z}\subseteq L^{p(\cdot)}(X,\mu)$ we define the following semi-modular\footnote{See \cite{AH}.}
\begin{equation*}
	\rho_{\ell^{q(\cdot)}(L^{p(\cdot)}(X,\mu))}\left(\left\{u_k\right\}\right):=\sum_{k\in \mathbb Z} \inf\left\{\lambda_k>0: \rho_{p(\cdot)}\left(\frac{u_k}{\lambda_k^{\frac{1}{q(\cdot)}}}\right) \leq 1\right\},
\end{equation*}
where we apply the convention $\lambda^{1/\infty}=1$.	Then, we define the \texttt{mixed Lebesgue sequence space} $\ell^{q(\cdot)}(L^{p(\cdot)}(X,\mu))$ as follows
\begin{equation*}
	\ell^{q(\cdot)}(L^{p(\cdot)}(X,\mu)):=\left\{\left\{u_k\right\}_{k\in \mathbb Z}\subseteq L^{p(\cdot)}(X,\mu): \exists_{\lambda\in(0,\infty)} \hspace{1mm} \rho_{\ell^{q(\cdot)}(L^{p(\cdot)})}\left(\frac{\left\{u_k\right\}_{k\in \mathbb Z}}{\lambda}\right) <\infty \right\}.
\end{equation*}
This is a quasi-normed space when equipped with the quasi-norm
\begin{equation*}
	\left\| \left\{u_k\right\} \right\|_{\ell^{q(\cdot)}(L^{p(\cdot)}(X,\mu))} := \inf\left\{\lambda\in(0,\infty): \rho_{\ell^{q(\cdot)}(L^{p(\cdot)}(X,\mu))}\left(\frac{\left\{ u_k \right\}_{k\in \mathbb Z}}{\lambda}\right) \leq 1 \right\},
\end{equation*}
where $\left\{u_k\right\}_{k\in \mathbb Z} \in \ell^{q(\cdot)}(L^{p(\cdot)}(X,\mu))$.
On the other hand, we say that the sequence $\left\{u_k\right\}$ of measurable functions belongs to the \texttt{mixed Lebesgue sequence space} $L^{p(\cdot)}(\ell^{q(\cdot)}(X,\mu))$ if 
\begin{equation*}
	\left\| \left\{u_k\right\}_{k\in \mathbb Z} \right\|_{\ell^{q(\cdot)}} \in L^{p(\cdot)}(X,\mu),
\end{equation*}
where, for each $x\in X$, we define
\begin{equation*}
	\left\| \left\{u_k\right\}_{k\in \mathbb Z} \right\|_{\ell^{q(x)}}:= \left\{\begin{array}{ll} \displaystyle\sup_{k\in \mathbb Z} \left|u_k(x)\right|,& \textnormal{ if } q(x)=\infty,\\ \displaystyle \left(\sum_{k\in \mathbb Z} \left|u_k(x)\right|^{q(x)}\right)^{\frac{1}{q(x)}},& \textnormal{ if } q(x)<\infty.\end{array}\right.
\end{equation*}
The space $L^{p(\cdot)}(\ell^{q(\cdot)}(X,\mu))$ is also a quasi-normed space with the following quasi-norm
\begin{equation*}
	\left\| \left\{u_k\right\}_{k\in \mathbb Z} \right\|_{L^{p(\cdot)}(\ell^{q(\cdot)}(X))}:= \left\| \left\| \left\{u_k\right\}_{k\in \mathbb Z} \right\|_{\ell^{q(\cdot)}} \right\|_{L^{p(\cdot)}(X,\mu)},
\end{equation*}
where $\left\{u_k\right\}_{k\in \mathbb Z} \in L^{p(\cdot)}(\ell^{q(\cdot)}(X,\mu))$.

\begin{prop}
Let $(X,\mu)$ be a measure space. For every $p\in \mathcal{P}_b(X)$ and $q\in \mathcal{P}(X)$, the spaces $\ell^{q(\cdot)}(L^{p(\cdot)}(X,\mu))$ and $L^{p(\cdot)}(\ell^{q(\cdot)}(X,\mu))$ are quasi-Banach.
\end{prop}

The proof of completeness of $\ell^{q(\cdot)}(L^{p(\cdot)}(X,\mu))$ in the setting of $X=\mathbb R^n$ can be found in \cite{GG}. The general case can be proven by slightly modifying the proof for $\mathbb R^n$. We leave the proof in the case of $L^{p(\cdot)}(\ell^{q(\cdot)}(X,\mu))$ as a simple exercise for the reader.

\begin{prop}\label{monotonicity}
	Let $(X,\mu)$ be a measure space and $p\in \mathcal{P}_b(X)$, $q_1,q_2\in \mathcal{P}(X)$ be such that $q_1\leq q_2$. Then, for every sequence $g:=\left\{g_k\right\}_{k\in \mathbb Z}\subseteq L^{p(\cdot)}(X,\mu)$, we have
	\begin{align*}
		&\left\| g\right\|_{\ell^{q_2(\cdot)}\left(L^{p(\cdot)}(X,\mu)\right)}\leq \left\|g\right\|_{\ell^{q_1(\cdot)}\left(L^{p(\cdot)}(X,\mu)\right)},\\
		&\left\|g\right\|_{L^{p(\cdot)}\left(\ell^{q_2(\cdot)}\left(X,\mu\right)\right)}\leq \left\|g\right\|_{L^{p(\cdot)}\left(\ell^{q_1(\cdot)}\left(X,\mu\right)\right)}.
	\end{align*}
	In particular,
	\begin{align*}
	&\ell^{q_1(\cdot)}\left(L^{p(\cdot)}\left(X,\mu\right)\right) \hookrightarrow \ell^{q_2(\cdot)}\left(L^{p(\cdot)}\left(X,\mu\right)\right)\quad\text{and}\quad L^{p(\cdot)}\left(\ell^{q_1(\cdot)}\left(X,\mu\right)\right) \hookrightarrow L^{p(\cdot)}\left(\ell^{q_2(\cdot)}\left(X,\mu\right)\right).	 
	\end{align*} 
\end{prop}

\begin{proof}
	We being by proving the first inequality. If $\left\|g\right\|_{\ell^{q_1(\cdot)}\left(L^{p(\cdot)}\left(X,\mu\right)\right)}=\infty$, then the claim is obvious. Thus, we can assume $\left\|g\right\|_{\ell^{q_1(\cdot)}\left(L^{p(\cdot)}\left(X,\mu\right)\right)}<\infty$. Moreover, without loss of generality we can assume that $\left\|g\right\|_{\ell^{q_1(\cdot)}\left(L^{p(\cdot)}\left(X,\mu\right)\right)}=1$. It suffices to prove that
	\begin{equation*}
		\left\|g\right\|_{\ell^{q_2(\cdot)}\left(L^{p(\cdot)}\left(X,\mu\right)\right)}\leq 1.
	\end{equation*}
	Since $\left\|g\right\|_{\ell^{q_1(\cdot)}\left(L^{p(\cdot)}(X,\mu)\right)}=1$, by Proposition~\ref{ballprop} we know that $\rho_{\ell^{q_1(\cdot)}\left(L^{p(\cdot)}(X,\mu)\right)}\left(g\right)\leq 1$. It suffices to prove
	\begin{equation*}
		\rho_{\ell^{q_2(\cdot)}\left(L^{p(\cdot)}(X,\mu)\right)}\left(g\right) \leq \rho_{\ell^{q_1(\cdot)}\left(L^{p(\cdot)}(X,\mu)\right)}\left(g\right).
	\end{equation*}
	Now, notice that for each fixed $k\in \mathbb Z$, we have
	\begin{equation*}
		\left\{\lambda_k\in (0,1]: \rho_{p(\cdot)}\left(\frac{g_k}{\lambda_k^{\frac{1}{q_1(\cdot)}}}\right)\leq 1 \right\}\subseteq \left\{\lambda_k \in (0,1]: \rho_{p(\cdot)}\left(\frac{g_k}{\lambda_k^{\frac{1}{q_2(\cdot)}}}\right)\leq 1 \right\},
	\end{equation*}
	since the assumption $q_1\leq q_2$ implies $\displaystyle \left(\frac{1}{\lambda_k}\right)^{\frac{1}{q_1}} \geq \left(\frac{1}{\lambda_k}\right)^{\frac{1}{q_2}}$ for every $\lambda_k\in (0,1]$. Hence,
	\begin{align*}
		\rho_{\ell^{q_2(\cdot)}(L^{p(\cdot)}(X,\mu))}\left(g\right)&\leq \sum_{k\in \mathbb Z} \inf\left\{\lambda_k\in (0,1]: \rho_{p(\cdot)}\left(\frac{g_k}{\lambda_k^{\frac{1}{q_2(\cdot)}}}\right)\leq 1 \right\}\\
		&\leq  \sum_{k\in \mathbb Z} \inf\left\{\lambda_k\in (0,1]: \rho_{p(\cdot)}\left(\frac{g_k}{\lambda_k^{\frac{1}{q_1(\cdot)}}}\right)\leq 1 \right\}.
	\end{align*}
	Since $\rho_{\ell^{q_1(\cdot)}\left(L^{p(\cdot)}(X,\mu)\right)}\left(g\right)\leq 1$, we can write
	\begin{align*}
		&\sum_{k\in \mathbb Z} \inf\left\{\lambda_k\in (0,1]: \rho_{p(\cdot)}\left(\frac{g_k}{\lambda_k^{\frac{1}{q_1(\cdot)}}}\right)\leq 1 \right\}\\
		&\qquad= \sum_{k\in \mathbb Z} \inf\left\{\lambda_k\in(0,\infty): \rho_{p(\cdot)}\left(\frac{g_k}{\lambda_k^{\frac{1}{q_1(\cdot)}}}\right)\leq 1 \right\}=\rho_{\ell^{q_1(\cdot)}\left(L^{p(\cdot)}(X,\mu)\right)}\left(g\right).
	\end{align*}
	Hence, we have proved that
	\begin{equation*}
		\rho_{\ell^{q_2(\cdot)}\left(L^{p(\cdot)}(X,\mu)\right)}\left(g\right)\leq \rho_{\ell^{q_1(\cdot)}\left(L^{p(\cdot)}(X,\mu)\right)}\left(g\right)\leq 1,
	\end{equation*}
	as wanted.
Proposition~\ref{ballprop} provides that $\left\|g\right\|_{\ell^{q_2(\cdot)}\left(L^{p(\cdot)}\left(X,\mu\right)\right)}\leq 1$ and the proof of the first inequality is complete.
	
	Now, we shall prove the second inequality. Since $q_2 \geq q_1$, then for almost every $x\in X$
	\begin{equation*}
		\left\|g\right\|_{\ell^{q_2(x)}}\leq \left\|g\right\|_{\ell^{q_1(x)}}.
	\end{equation*}
	Taking $L^{p(\cdot)}$-norm we obtain the desired inequality, completing the proof of Proposition~\ref{monotonicity}.
\end{proof}

\begin{lemma}\label{mieszane}\cite{AH}
Let $(X,\mu)$ be a measure space and $p\in \mathcal{P}_b(X)$ and $q\in \mathcal{P}(X)$.
\begin{enumerate}
	\item[(i)] If $q$ is constant, then
	\begin{equation*}
		\left\| \left\{f_k\right\}_{k\in \mathbb Z} \right\|_{\ell^{q}(L^{p(\cdot)}(X,\mu))}= \left\| \left\{ \left\|f_k\right\|_{L^{p(\cdot)}(X,\mu)}\right\}_{k\in \mathbb Z} \right\|_{\ell^q}
	\end{equation*}
	for every $\left\{f_k\right\}_{k\in \mathbb Z} \subseteq L^{p(\cdot)}(X,\mu)$.
	\item[(ii)] If $q^+ <\infty$, then
	\begin{equation*}
		\rho_{\ell^{q(\cdot)}(L^{p(\cdot)}(X,\mu))}\left( \left\{f_k\right\}_{k\in \mathbb Z}\right)= \sum_{k\in \mathbb Z} \left\| \left|f_k\right|^{q(\cdot)} \right\|_{L^{\frac{p(\cdot)}{q(\cdot)}}(X,\mu)}
	\end{equation*}
	for every $\left\{f_k\right\}_{k\in \mathbb Z}\subseteq L^{p(\cdot)}(X,\mu)$. Moreover, $\left\{f_k\right\}_{k\in \mathbb Z}\in \ell^{q(\cdot)}(L^{p(\cdot)}(X,\mu))$ if and only if
	\begin{equation*}
	\sum_{k\in \mathbb Z} \left\| \left|f_k\right|^{q(\cdot)} \right\|_{L^{\frac{p(\cdot)}{q(\cdot)}}(X,\mu)}<\infty.
	\end{equation*}
\end{enumerate}
\end{lemma}

\subsection{Variable exponent Sobolev, Besov, and Triebel-Lizorkin spaces}

\indent Let $(X,d)$ be a metric space equipped with a non-negative Borel measure $\mu$ and let $u: X \to \overline{\mathbb{R}}$ be a measurable function which is finite $\mu$-almost everywhere. Let $s\in \mathcal{P}_b(X)$. We say that a non-negative function $g: X \to \mathbb \overline{\mathbb{R}}$ is a \texttt{scalar $s(\cdot)$-gradient of $u$} if there is a measure zero set $N_u\subseteq X$ such that
	\begin{equation}\label{pointwise}
		\vert u(x)-u(y)\vert\leq d(x,y)^{s(x)}g(x)+d(x,y)^{s(y)}g(y),
	\end{equation}
	for all $x,y\in X\setminus N_u$. The collection of all scalar $s(\cdot)$-gradients of $u$ shall be denoted by $\mathcal{D}^{s(\cdot)}(u)$.
	Moreover, we say that a sequence of non-negative measurable functions $\left\{g_k\right\}_{k\in \mathbb Z}$ defined on $X$ is a \texttt{vector $s(\cdot)$-gradient of $u$} if there exists a measure zero set $N_u \subseteq X$ such that
	\begin{equation*}
		\left|u(x)-u(y)\right|\leq d(x,y)^{s(x)}g_k(x)+d(x,y)^{s(y)}g_k(y),
	\end{equation*}
	for all $x,y\in X\setminus N_u$ satisfying $2^{-k-1}\leq d(x,y)<2^{-k}$. The collection of all vector $s(\cdot)$-gradients of $u$ shall be denoted by $\mathbb D^{s(\cdot)}(u)$.
		
	For $s,p\in \mathcal{P}_b(X)$ we define the \texttt{homogeneous Haj{\l}asz-Sobolev space} $\dot{M}^{s(\cdot),p(\cdot)}(X,d,\mu)$ as the space of all measurable functions $u:X \to \overline{\mathbb{R}}$ such that
	\begin{equation*}
		\left\|u\right\|_{\dot{M}^{s(\cdot),p(\cdot)}(X,d,\mu)}:=\inf_{g\in \mathcal{D}^{s(\cdot)}(u)}\left\|g\right\|_{L^{p(\cdot)}(X,\mu)}<\infty,
	\end{equation*}
	and we define the \texttt{inhomogeneous Haj{\l}asz-Sobolev space} $M^{s(\cdot),p(\cdot)}(X,d,\mu)$ as
	\begin{equation*}
		M^{s(\cdot),p(\cdot)}(X,d,\mu):=\dot{M}^{s(\cdot),p(\cdot)}(X,d,\mu) \cap L^{p(\cdot)}(X,\mu).
	\end{equation*}
Here and thereafter, we make the agreement that  $\inf\emptyset:=\infty$.		
		
The space $M^{s(\cdot),p(\cdot)}(X,d,\mu)$ is a quasi-Banach space with the following quasi-norm
\begin{equation*}
	\left\|u\right\|_{M^{s(\cdot),p(\cdot)}(X,d,\mu)}:=\left\|u\right\|_{L^{p(\cdot)}(X,\mu)}+\left\|u\right\|_{\dot{M}^{s(\cdot),p(\cdot)}(X,d,\mu)},
\end{equation*}
where $u\in M^{s(\cdot),p(\cdot)}(X,d,\mu)$.

Now we shall present the definition of variable exponent Haj{\l}asz-Triebel-Lizorkin and Haj{\l}asz-Besov spaces. 
For $s,p\in \mathcal{P}_b(X)$ and $q\in \mathcal{P}(X)$, we define
	\begin{enumerate}
		\item[$(i)$] the \texttt{homogeneous Haj{\l}asz-Triebel-Lizorkin space} $\dot{M}^{s(\cdot)}_{p(\cdot),q(\cdot)}(X,d,\mu)$ as the space of all measurable functions $u:X \to \overline{\mathbb{R}}$ such that
		\begin{equation*}
			\left\|u\right\|_{\dot{M}^{s(\cdot)}_{p(\cdot),q(\cdot)}\left(X,d,\mu\right)}:=\inf_{g\in \mathbb D^{s(\cdot)}(u)} \left\| g \right\|_{L^{p(\cdot)}\left(\ell^{q(\cdot)}\left(X,\mu\right)\right)}<\infty.
		\end{equation*}
		\item[(ii)] the \texttt{homogeneous Haj{\l}asz-Besov space} $\dot{N}^{s(\cdot)}_{p(\cdot),q(\cdot)}(X,d,\mu)$ as the space of all measurable functions $u: X\to \overline{\mathbb{R}}$ such that
		\begin{equation*}
			\left\|u\right\|_{\dot{N}^{s(\cdot)}_{p(\cdot),q(\cdot)}(X,d,\mu)}:=\inf_{g\in \mathbb D^{s(\cdot)}(u)} \left\| g \right\|_{\ell^{q(\cdot)}\left(L^{p(\cdot)}(X,\mu)\right)}<\infty.
		\end{equation*}
		
	\end{enumerate} 
		
For $s,p\in \mathcal{P}_b(X)$ and $q\in \mathcal{P}(X)$, we define \texttt{inhomogeneous Haj{\l}asz-Triebel-Lizorkin space} $M^{s(\cdot)}_{p(\cdot),q(\cdot)}(X,d,\mu)$ and \texttt{inhomogeneous Haj{\l}asz-Besov space} $N^{s(\cdot)}_{p(\cdot),q(\cdot)}(X,d,\mu)$ as
	\begin{align*}
		M^{s(\cdot)}_{p(\cdot),q(\cdot)}(X,d,\mu)&:= \dot{M}^{s(\cdot)}_{p(\cdot),q(\cdot)}(X,d,\mu) \cap L^{p(\cdot)}(X,\mu),\\
		N^{s(\cdot)}_{p(\cdot),q(\cdot)}(X,d,\mu)&:= \dot{N}^{s(\cdot)}_{p(\cdot),q(\cdot)}(X,d,\mu) \cap L^{p(\cdot)}(X,\mu).
	\end{align*}
		
We endow the spaces $M^{s(\cdot)}_{p(\cdot),q(\cdot)}(X,d,\mu)$ and $N^{s(\cdot)}_{p(\cdot),q(\cdot)}(X,d,\mu)$ with the following quasi-norms
\begin{align*}
	\left\|u\right\|_{M^{s(\cdot)}_{p(\cdot),q(\cdot)}(X,d,\mu)}&:=\left\|u\right\|_{L^{p(\cdot)}(X,\mu)}+\left\|u\right\|_{\dot{M}^{s(\cdot)}_{p(\cdot),q(\cdot)}(X,d,\mu)}, \textnormal{ where } u\in M^{s(\cdot)}_{p(\cdot),q(\cdot)}(X,d,\mu)\\
	\left\|u\right\|_{N^{s(\cdot)}_{p(\cdot),q(\cdot)}(X,d,\mu)}&:=\left\|u\right\|_{L^{p(\cdot)}(X,\mu)}+\left\|u\right\|_{\dot{N}^{s(\cdot)}_{p(\cdot),q(\cdot)}(X,d,\mu)}, \textnormal{ where } u\in N^{s(\cdot)}_{p(\cdot),q(\cdot)}(X,d,\mu).
\end{align*}
Both Haj{\l}asz-Triebel-Lizorkin and Haj{\l}asz-Besov spaces are quasi-Banach spaces when equipped with these quasi-norms.

For $A=M, N$, we will simply use $\dot{M}^{s(\cdot),p(\cdot)}(X)$, ${M}^{s(\cdot),p(\cdot)}(X)$, $\dot{A}^{s(\cdot)}_{p(\cdot),q(\cdot)}(X)$,  and $A^{s(\cdot)}_{p(\cdot),q(\cdot)}(X)$
respectively, in place of $\dot{M}^{s(\cdot),p(\cdot)}(X,d,\mu)$, ${M}^{s(\cdot),p(\cdot)}(X,d,\mu)$, ${A}^{s(\cdot)}_{p(\cdot),q(\cdot)}(X,d,\mu)$, 
 and $A^{s(\cdot)}_{p(\cdot),q(\cdot)}(X,d,\mu)$,  whenever the metric and the measure are well understood from the context. 

\begin{remark}
	The spaces $M^{s(\cdot)}_{p(\cdot),q(\cdot)}$ and $N^{s(\cdot)}_{p(\cdot),q(\cdot)}$ were defined in \cite{Rn} using  slightly different quasi-norms. Let us recall that if $u:X \to \mathbb R$ is a measurable function, then a sequence $\left\{g_k\right\}_{k\in\mathbb Z}\in \mathbb{D}^0(u)$ if the functions $g_k: X\to[0,\infty]$ are measurable and there is a null set $N_u\subseteq X$ such that
	\begin{equation*}
		\left|u(x)-u(y)\right| \leq g_k(x)+g_k(y),
	\end{equation*}
	for every $x\in X \setminus N_u$ satisfying $2^{-k-1} \leq d(x,y) <2^{-k}$.
	
	In  \cite{Rn}, the authors defined the space $ \tilde{A}^{s(\cdot)}_{p(\cdot),q(\cdot)}(X,d,\mu)$ as the space of all measurable functions $u: X \to \mathbb R$ such that
	\begin{equation}\label{equivalent}
		\left\|u\right\|_{\tilde{A}^{s(\cdot)}_{p(\cdot),q(\cdot)}(X,d,\mu)}:=\left\|u\right\|_{L^{p(\cdot)}(X,\mu)}+\inf_{\left\{g_k\right\}_{k\in\mathbb Z}\in \mathbb{D}^0(u)} \left\|\left\{2^{ks}g_k\right\}_{k\in\mathbb Z} \right\| <\infty,
	\end{equation} 
	where $\left\|\cdot \right\|$ is the $L^{p(\cdot)}(\ell^{q(\cdot)})$-quasi-norm if $A=M$ and the $\ell^{q(\cdot)}(L^{p(\cdot)})$-quasi-norm if $A=N$.
	
	We shall prove that $A^{s(\cdot)}_{p(\cdot),q(\cdot)}(X,d,\mu)=\tilde{A}^{s(\cdot)}_{p(\cdot),q(\cdot)}(X,d,\mu)$ with equivalent quasi-norms.
\end{remark}

\begin{proof} Let $u\in A^{s(\cdot)}_{p(\cdot),q(\cdot)}(X,d,\mu)$ and $\left\{g_k\right\}_{k\in\mathbb Z}\in \mathbb{D}^{s(\cdot)}(u)$ with $\left\| \left\{g_k\right\} \right\| < \infty$. Let $N_u \subseteq X$ be the null set such that
	\begin{equation}\label{n1}
		\left|u(x)-u(y)\right| \leq d(x,y)^{s(x)}g_k(x)+d(x,y)^{s(y)}g_k(y),
	\end{equation}
	whenever $x,y\in X \setminus N_u$ satisfy $2^{-k-1} \leq d(x,y) <2^{-k}$. Then, \eqref{n1} yields
	\begin{equation*}
		\left|u(x)-u(y)\right|\leq 2^{-ks(x)}g_k(x)+2^{-ks(y)}g_k(y)=\tilde{g}(x)+\tilde{g}(y),
	\end{equation*}
	where $\tilde{g}_k(x):=2^{-ks(x)}g_k(x)$. Then, $\left\{ \tilde{g}_k \right\}_{k\in\mathbb Z}\in \mathbb{D}^0(u)$ and
	\begin{equation*}
		\left\| \left\{2^{ks}\tilde{g}_k\right\}_{k\in\mathbb Z} \right\| = \left\| \left\{g_k\right\}_{k\in\mathbb Z} \right\| <\infty.
	\end{equation*}
	Therefore $u\in \tilde{A}^{s(\cdot)}_{p(\cdot),q(\cdot)}(X,d,\mu)$ and $\left\| u \right\|_{\tilde{A}^{s(\cdot)}_{p(\cdot),q(\cdot)}(X,d,\mu)} \leq \left\| u \right\|_{A^{s(\cdot)}_{p(\cdot),q(\cdot)}(X,d,\mu)}.$
	
	Conversely, assume that $u\in \tilde{A}^{s(\cdot)}_{p(\cdot),q(\cdot)}(X,d,\mu)$ and let $\left\{g_k\right\}_{k\in\mathbb Z}\in \mathbb{D}^0(u)$ with $\left\|2^{ks} \left\{ g_k \right\}_{k\in\mathbb Z}\right\| <\infty$. Then, there exists a null set $N_u \subseteq X$ such that, for all $x,y\in X \setminus N_u$ with $2^{-k-1}\leq d(x,y) <2^{-k}$, we have
	\begin{align*}
		\left|u(x)-u(y)\right|& \leq g_k(x)+g_k(y)= 2^{\left(-k-1\right)s(x)}\tilde{g}_k(x)+2^{\left(-k-1\right)s(y)}\tilde{g}_k(y) \\ & \leq d(x,y)^{s(x)}\tilde{g}_k(x)+d(x,y)^{s(y)}\tilde{g}_k(y),
	\end{align*}
	where $\tilde{g}_k(x):=2^{\left(k+1\right)s(x)}g_k(x)$. Hence $\left\{ \tilde{g}_k \right\}_{k\in\mathbb Z}\in \mathbb{D}^{s(\cdot)}(u)$ and
	\begin{equation*}
		\left\| \left\{\tilde{g}_k\right\}_{k\in\mathbb Z} \right\|= \left\|\left\{ 2^{\left(k+1\right)s}g_k \right\}_{k\in\mathbb Z} \right\| \leq 2^{s^+} \left\| \left\{ 2^{ks} g_k\right\}_{k\in\mathbb Z} \right\| <\infty.
	\end{equation*}
	Therefore $u\in A^{s(\cdot)}_{p(\cdot),q(\cdot)}(X,d,\mu)$ and $\left\| u \right\|_{A^{s(\cdot)}_{p(\cdot),q(\cdot)}(X,d,\mu)} \leq 2^{s^+} \left\| u \right\|_{\tilde{A}^{s(\cdot)}_{p(\cdot),q(\cdot)}(X,d,\mu)}.$ This finishes the proof of our claim.
\end{proof}

\begin{lemma}\label{gradientzero}
Let $(X,d)$ be a metric space equipped with a non-negative Borel measure $\mu$  and let $p,s\in \mathcal{P}_b(X)$. Suppose that the function $u\in \dot{M}^{s(\cdot),p(\cdot)}(X,d,\mu)$ is such that $\left\|u\right\|_{\dot{M}^{s(\cdot),p(\cdot)}(X,d,\mu)}=0$. Then, $u$ is constant pointwise almost everywhere.	
\end{lemma}
\begin{proof}
By the assumption, there exists sequence $\left\{g_n\right\}_{n\in\mathbb Z}\subseteq \mathcal{D}^{s(\cdot)}(u)\cap L^{p(\cdot)}(X,\mu)$ such that $g_n \to 0$ in $L^{p(\cdot)}(X,\mu)$ as $n\to\infty$. Passing to subsequence, we can assume that $g_n \to 0$ pointwise almost everywhere as $n\to\infty$. Let $N\subseteq X$ be the null set such that $\left\{g_n\right\}_{n\in\mathbb Z}$ converges pointwise to zero on $X\setminus N$. By the very definition of a scalar $s(\cdot)$-gradient, for each $n\in \mathbb Z$ we can find the null set $E_n\subseteq X$ such that for all $x,y\in X \setminus E_n$,
\begin{equation*}
	\left|u(x)-u(y)\right| \leq d(x,y)^{s(x)}g_n(x)+d(x,y)^{s(y)}g_n(y).
\end{equation*}
We define
\begin{equation*}
	E:=\bigcup_{n=1}^{\infty} E_n \cup N.
\end{equation*}
Then, $\mu(E)=0$ and for every $n\in \mathbb N$ and $x,y\in X \setminus E$ we have
\begin{equation*}
	\left|u(x)-u(y)\right|\leq d(x,y)^{s(x)}g_n(x)+d(x,y)^{s(y)}g_n(y).
\end{equation*}
Passing with $n\to \infty$ we obtain $u(x)=u(y)$. Hence, the function $u$ is constant almost everywhere and the proof is complete.
\end{proof}

\section{Embeddings of Variable Exponent Sobolev, Triebel-Lizorkin, and Besov Spaces}\label{sect:mainresults}

In this section, we present the core results of the paper regarding functional inequalities. We begin by investigating the structural properties and internal embeddings within the variable Haj\l{}asz--type scales. We then establish local Sobolev, Morrey, and Moser--Trudinger type inequalities on balls under lower Ahlfors growth conditions on the measure. These results are subsequently extended to their global counterparts under additional geometric hypotheses, such as geometric doubling and uniform bounds on the measure of balls of a fixed radius.

Throughout this section, we will use the following observation.
\begin{remark}\label{measurethreshold}
Given a metric measure space $(X,d,\mu)$, if $Q: X \rightarrow (0,\infty)$ is a bounded measurable function then the upper bound 1 for the radii in the lower Ahlfors $Q(\cdot)$-regularity condition is not essential and can be replaced by any strictly positive and finite threshold. 
Indeed, if there exist $\delta,b\in(0,\infty)$ such that $\mu(B(x,r)) \geq br^{Q(x)}$ holds for every $x\in X$ and $r\in (0,\delta]$, then for every $\delta' \in [\delta,\infty)$, $x\in X$, and $r\in (0,\delta']$, we have $\mu(B(x,r)) \geq b(\delta/\delta')^{Q^+} r^{Q(x)}.$ Similar considerations apply to the upper Ahlfors $Q(\cdot)$-regularity condition, provided one imposes additional assumptions on $(X,d,\mu)$, such as the measure $\mu$ being doubling.
\end{remark}

\subsection{Embedding properties of variable exponent Haj{\l}asz-type spaces}\label{htlhb}
		
\begin{prop}\label{embeddingsbetween}
	Let $(X,d)$ be a metric space equipped with a non-negative Borel measure $\mu$. The following statements hold.
	\begin{enumerate}
		\item[(i)] If $p,s\in \mathcal{P}_b(X)$ and $q_1,q_2\in \mathcal{P}(X)$ are such that $q_1\leq q_2$, then
		\begin{align*}
			\dot{M}^{s(\cdot)}_{p(\cdot),q_1(\cdot)}(X,d,\mu) \hookrightarrow \dot{M}^{s(\cdot)}_{p(\cdot),q_2(\cdot)}(X,d,\mu),\hspace{8mm}\dot{N}^{s(\cdot)}_{p(\cdot),q_1(\cdot)}(X,d,\mu) \hookrightarrow     \dot{N}^{s(\cdot)}_{p(\cdot),q_2(\cdot)}(X,d,\mu),
		\end{align*}
		and the operator semi-norms of the above embeddings are less than or equal to one.
		\item[(ii)] If $p,s\in \mathcal{P}_b(X)$, then
		\begin{align*}
			\dot{M}^{s(\cdot)}_{p(\cdot),\infty}(X,d,\mu)=\dot{M}^{s(\cdot),p(\cdot)}(X,d,\mu)
		\end{align*}
		with equal quasi-semi-norms.
		\item[(iii)] If $p,s\in \mathcal{P}_b(X)$, it holds that
		\begin{equation*}
			\dot{M}^{s(\cdot)}_{p(\cdot),p(\cdot)}(X,d,\mu)=\dot{N}^{s(\cdot)}_{p(\cdot),p(\cdot)}(X,d,\mu)
		\end{equation*}
		with equal quasi-semi-norms.
		\item[(iv)] If $s,t,p\in \mathcal{P}_b(X)$ and $q_1,q_2\in \mathcal{P}(X)$ are such that $ t \gg s$, then,
		\begin{align*}
			N^{t(\cdot)}_{p(\cdot),q_1(\cdot)}(X,d,\mu)\hookrightarrow N^{s(\cdot)}_{p(\cdot),q_2(\cdot)}(X,d,\mu)\quad\text{and}\quad
			M^{t(\cdot)}_{p(\cdot),q_1(\cdot)}(X,d,\mu)\hookrightarrow M^{s(\cdot)}_{p(\cdot),q_2(\cdot)}(X,d,\mu).
		\end{align*}
		\item[(v)] If $p,s\in \mathcal{P}_b(X)$, $q\in \mathcal{P}(X)$, then it holds that
		\begin{align*}
			\dot{M}^{s(\cdot)}_{p(\cdot),q(\cdot)}(X,d,\mu) \hookrightarrow \dot{M}^{s(\cdot),p(\cdot)}(X,d,\mu),
		\end{align*}
		where the operator semi-norm is less than or equal to one.
		\item[(vi)] If $p,s\in \mathcal{P}_b(X)$ and $q\in \mathcal{P}(X)$ are such that $p\geq q$, then
		\begin{equation*}
			\dot{N}^{s(\cdot)}_{p(\cdot),q(\cdot)}(X,d,\mu) \hookrightarrow \dot{M}^{s(\cdot),p(\cdot)}(X,d,\mu),
		\end{equation*}
		where the operator semi-norm is less than or equal to one.
		\item[(vii)] If $p,s\in \mathcal{P}_b(X)$ and $q\in \mathcal{P}(X)$, then for every $t \in \mathcal{P}_b(X)$ such that $s \gg t$, it holds that
		\begin{equation*}
			N^{s(\cdot)}_{p(\cdot),q(\cdot)}(X,d,\mu) \hookrightarrow M^{t(\cdot),p(\cdot)}(X,d,\mu).
		\end{equation*}
		\item[(viii)] If $p,s\in \mathcal{P}_b(X)$ and $q\in \mathcal{P}(X)$, then
		\begin{equation*}
			\dot{M}^{s(\cdot)}_{p(\cdot),q(\cdot)}(X,d,\mu) \hookrightarrow \dot{N}^{s(\cdot)}_{p(\cdot),\infty}(X,d,\mu).
		\end{equation*}
		\item[(ix)] If $p,s\in \mathcal{P}_b(X)$ and $\delta \in (0,\infty)$, then for every $t\in \mathcal{P}_b(X)$ such that $s \gg t$ there exists constant $\zeta(p,s,t,\delta)\in (0,\infty)$  such that for every ball $B:=B(x_0,r_0)\subseteq X$ with $r_0\in (0,\delta]$ and every $q\in \mathcal{P}(X)$, it holds that
		\begin{equation*}
			\dot{N}^{s(\cdot)}_{p(\cdot),q(\cdot)}(B,d,\mu) \hookrightarrow \dot{M}^{t(\cdot),p(\cdot)}(B,d,\mu)
		\end{equation*}
		and
		\begin{equation*}
			\left\| u \right\|_{\dot{M}^{t(\cdot),p(\cdot)}(B,d,\mu)} \leq \zeta(p,s,t,\delta)\left\|u \right\|_{\dot{N}^{s(\cdot)}_{p(\cdot),q(\cdot)}(B,d,\mu)},
		\end{equation*}
		for all $u\in \dot{N}^{s(\cdot)}_{p(\cdot),q(\cdot)}(B,d,\mu).$
	\end{enumerate}
\end{prop}

\begin{proof}
	$(i)$ This is straightforward due to Proposition~\ref{monotonicity}.
	
	$(ii)$ Let $u\in \dot{M}^{s(\cdot)}_{p(\cdot),\infty}(X,d,\mu)$, and take any $\left\{h_k\right\}_{k\in\mathbb Z}\in \mathbb D^{s(\cdot)}(u)$ such that
	\begin{equation*}
		\left\|\left\{h_k\right\}_{k\in\mathbb Z}\right\|_{L^{p(\cdot)}\left(\ell^{\infty}\left(X,\mu\right)\right)}=\left\| \sup_{k\in \mathbb Z} h_k \right\|_{L^{p(\cdot)}(X,\mu)}<\infty.
	\end{equation*}
	Define $h(x):=\displaystyle \sup_{k\in \mathbb Z} h_k(x)$. We have $h\in \mathcal{D}^{s(\cdot)}(u)\cap L^{p(\cdot)}(X,\mu)$, which implies that $u\in \dot{M}^{s(\cdot),p(\cdot)}\left(X,d,\mu\right)$. Moreover,
	\begin{equation*}
		\left\|u\right\|_{\dot{M}^{s(\cdot),p(\cdot)}(X,d,\mu)}=\inf_{g\in \mathcal{D}^{s(\cdot)}(u)}\left\|g\right\|_{L^{p(\cdot)}(X,\mu)}\leq \left\| h \right\|_{L^{p(\cdot)}(X,\mu)} = \left\|\left\{h_k\right\}_{k\in\mathbb Z}\right\|_{L^{p(\cdot)}\left(\ell^{\infty}\left(X,\mu\right)\right)}.
	\end{equation*}
	Taking the infimum over all $\left\{h_k\right\}_{k\in\mathbb Z}\in \mathbb{D}^{s(\cdot)}(u)$ we obtain
	\begin{equation*}
		\left\|u\right\|_{\dot{M}^{s(\cdot),p(\cdot)}(X,d,\mu)}\leq \left\|u\right\|_{\dot{M}^{s(\cdot)}_{p(\cdot),\infty}(X,d,\mu)}.
	\end{equation*}
	Now, assume that $u\in \dot{M}^{s(\cdot),p(\cdot)}(X,d,\mu)$. Let us take any $h\in \mathcal{D}^{s(\cdot)}(u)\cap L^{p(\cdot)}(X,\mu)$ and define $h_k:=h$ for all $k\in \mathbb Z$. Then $\left\{h_k\right\}_{k\in\mathbb Z}\in \mathbb{D}^{s(\cdot)}(u)$ and
	\begin{equation*}
		\left\|u\right\|_{\dot{M}^{s(\cdot)}_{p(\cdot),\infty}(X,d,\mu)}=\inf_{\left\{g_k\right\}_{k\in\mathbb Z}\in \mathbb{D}^{s(\cdot)}(u)}\left\| \left\{g_k \right\}_{k\in\mathbb Z} \right\|_{L^{p(\cdot)}\left(\ell^{\infty}\left(X,\mu\right)\right)} \leq \left\| \left\{h_k\right\}_{k\in\mathbb Z} \right\|_{L^{p(\cdot)}(\ell^{\infty}(X,\mu))}= \left\|h\right\|_{L^{p(\cdot)}(X,\mu)},
	\end{equation*}
	and thus, $u\in \dot{M}^{s(\cdot)}_{p(\cdot),\infty}(X,d,\mu)$. Moreover, by taking the infimum over all $h\in \mathcal{D}^{s(\cdot)}(u)$, we obtain
	\begin{equation*}
		\left\|u \right\|_{\dot{M}^{s(\cdot)}_{p(\cdot),\infty}(X,d,\mu)} \leq \left\|u\right\|_{\dot{M}^{s(\cdot),p(\cdot)}(X,d,\mu)}
	\end{equation*}
	
	$(iii)$ It suffices to prove that
	\begin{equation*}
		\rho_{\ell^{p(\cdot)}(L^{p(\cdot)}(X,\mu))}\left(g\right)=\rho_{L^{p(\cdot)}(X,\mu)}\left( \left(\sum_{k\in \mathbb Z} \left|g_k\right|^{p}\right)^{\frac{1}{p}} \right)
	\end{equation*}
	for every sequence of measurable functions $g:=\left\{g_k\right\}_{k\in\mathbb Z} \subseteq L^{p(\cdot)}(X,\mu)$. By Lemma~\ref{mieszane} $(ii)$ and the monotone convergence theorem, we get
	\begin{align*}
		\rho_{\ell^{p(\cdot)}(L^{p(\cdot)}(X,\mu))}\left(g\right)&=\sum_{k\in \mathbb Z} \left\| \left|g_k\right|^{p} \right\|_{L^1(X)}=\sum_{k\in \mathbb Z} \int_{X} \left|g_k(x)\right|^{p(x)}\mbox{d}\mu(x) \\ &=\rho_{L^{p(\cdot)}(X,\mu)}\left( \left(\sum_{k\in \mathbb Z} \left|g_k\right|^{p}\right)^{\frac{1}{p}} \right).
	\end{align*}
	and the proof of $(iii)$ is complete.    
	
	$(iv)$ 
	By the recently proved Proposition~\ref{embeddingsbetween} $(i)$, it suffices to prove that 
	\begin{align*}
		N^{t(\cdot)}_{p(\cdot),\infty}(X,d,\mu)&\hookrightarrow N^{s(\cdot)}_{p(\cdot),q_2^-}(X,d,\mu),\\
		M^{t(\cdot)}_{p(\cdot),\infty}(X,d,\mu)&\hookrightarrow M^{s(\cdot)}_{p(\cdot),q_2^-}(X,d,\mu).
	\end{align*}
	Let $u\in N^{t(\cdot)}_{p(\cdot),\infty}(X,d,\mu)$ and $\left\{g_k\right\}_{k\in \mathbb Z} \in \mathbb{D}^{t(\cdot)}(u) \cap \ell^{\infty}\left(L^{p(\cdot)}(X,\mu)\right).$ Denote $\varepsilon:= t-s$. Since $t \gg s$, thus $\varepsilon^-\in(0,\infty)$. Now define
	\begin{equation*}
		h_k(x):=\left\{\begin{array}{ll} 2^{(k+1)s^-}\left|u(x)\right|& \textnormal{ for } k<0,\\ 2^{-k\varepsilon^-}g_k(x)& \textnormal{ for } k\geq 0.\end{array}\right.
	\end{equation*}
	We we will show that $\left\{h_k\right\}_{k\in \mathbb Z}\in \mathbb{D}^{s(\cdot)}(u)$. Let $N_u\subseteq X$ be the null set from the definition of $\left\{g_k\right\}_{k\in\mathbb Z}$ and fix $k\in\mathbb Z$, $x,y\in X\setminus N_u$, and assume $2^{-k-1}\leq d(x,y)<2^{-k}$. Then, if $k\geq 0$,  we have that
	\begin{align*}
		\left|u(x)-u(y)\right| &\leq d(x,y)^{t(x)} g_k(x)+d(x,y)^{t(y)} g_k(y)\\
		&\leq d(x,y)^{s(x)} 2^{-k\varepsilon^-}g_k(x)+d(x,y)^{s(y)} 2^{-k\varepsilon^-}g_k(y)\\ &=d(x,y)^{s(x)}h_k(x)+d(x,y)^{s(y)}h_k(y).
	\end{align*}
	On the other hand, if $k<0$, then we have that
	\begin{align*}
		\left|u(x)-u(y)\right|&\leq 2^{(-k-1)s(x)}2^{(k+1)s(x)}\left|u(x)\right|+2^{(-k-1)s(y)}2^{(k+1)s(y)}\left|u(y)\right| \\ &\leq d(x,y)^{s(x)} h_k(x)+d(x,y)^{s(y)}h_k(y).
	\end{align*}
	Therefore $\left\{h_k\right\}_{k\in\mathbb Z}\in \mathbb{D}^{s(\cdot)}(u).$ Now, since $\left\{h_k\right\}_{k\in\mathbb Z}\subseteq L^{p(\cdot)}(X,\mu)$, if $q_2^-<\infty$, then by Lemma~\ref{mieszane} $(i)$ we have
	\begin{align*}
		\left\| \left\{h_k\right\}_{k\in\mathbb Z} \right\|_{\ell^{q_2^-}\left(L^{p(\cdot)}(X,\mu)\right)}^{q_2^-}&=\left\|u\right\|_{L^{p(\cdot)}(X,\mu)}^{q_2^-}\sum_{k=-\infty}^{-1} 2^{(k+1)s^-q_2^-}+\sum_{k=0}^{\infty} 2^{-kq_2^-\varepsilon^-}\left\|g_k\right\|_{L^{p(\cdot)}(X,\mu)}^{q_2^-}\\ & \leq \frac{1}{1-2^{-s^-q_2^-}}\left\|u\right\|_{L^{p(\cdot)}(X,\mu)}^{q_2^-}+ \frac{1}{1-2^{-q_2^-\varepsilon^-}}\left\| \left\{g_k\right\}_{k\in\mathbb Z} \right\|_{\ell^{\infty}\left(L^{p(\cdot)}(X,\mu)\right)}^{q_2^-}.
	\end{align*}
	Moreover, if $q_2^-=\infty$, then $q_2 \equiv \infty$ and
	\begin{equation*}
	\left\| \left\{h_k\right\}_{k\in \mathbb Z} \right\|_{\ell^{\infty}\left(L^{p(\cdot)}(X,\mu)\right)} \leq \left\|u\right\|_{L^{p(\cdot)}(X,\mu)}+\left\| \left\{g_k\right\}_{k\in \mathbb Z} \right\|_{\ell^{\infty}\left(L^{p(\cdot)}(X,\mu)\right)}.
	\end{equation*}
	Therefore $u\in N^{s(\cdot)}_{p(\cdot),q_2^-}(X,d,\mu)$ and the embedding is continuous.
	
	The proof of the second inclusion is similar. From $(ii)$, we know that $M^{t(\cdot)}_{p(\cdot),\infty}(X,d,\mu)=M^{t(\cdot),p(\cdot)}(X,d,\mu)$ with equal quasi-norms. Let $g\in \mathcal{D}^{t(\cdot)}(u) \cap L^{p(\cdot)}(X)$ and define
	\begin{align*}
		f_k(x):=\left\{\begin{array}{ll} 2^{(k+1)s^-}\left|u(x)\right|& \textnormal{ for } k<0,\\ 2^{-k\varepsilon^-}g(x)& \textnormal{ for } k\geq 0,\end{array}\right. 
	\end{align*}
	where $\varepsilon:= t-s$. Now, let $N_u\subseteq X$ be the null set from the definition of $g$, and fix $k\in\mathbb Z$, $x,y\in X$, and assume that $2^{-k-1}\leq d(x,y)<2^{-k}$. Then if $k\geq 0$, we have that
	\begin{align*}
		\left|u(x)-u(y)\right|&\leq d(x,y)^{t(x)}g(x)+d(x,y)^{t(y)}g(y)\\
		&\leq d(x,y)^{s(x)}f_k(x)+d(x,y)^{s(y)}f_k(y).
	\end{align*}
	On the other hand, if $k<0$ then we have
	\begin{equation*}
		\left|u(x)-u(y)\right|\leq d(x,y)^{s(x)}f_k(x)+d(x,y)^{s(y)}f_k(y),
	\end{equation*}
	which proves that $\left\{f_k\right\}_{k\in\mathbb Z}\in \mathbb{D}^{s(\cdot)}(u)$. 
	Clearly, $\left\{f_k\right\}_{k\in\mathbb Z}\subseteq L^{p(\cdot)}(X,\mu)$. Moreover, if $q_2^- <\infty$, then
	\begin{align*}
		\left\| \left\{f_k\right\}_{k\in\mathbb Z} \right\|_{L^{p(\cdot)}\left(\ell^{q_2^-}\left(X,\mu\right)\right)}&\leq 2^{\frac{1}{q_2^-}}\left\| \left|u(x)\right|\left(\sum_{k=-\infty}^{-1} 2^{(k+1)s^-q_2^-}\right)^{\frac{1}{q_2^-}} + g(x)\left(\sum_{k=0}^{\infty} 2^{-k\varepsilon^-q_2^-}\right)^{\frac{1}{q_2^-}} \right\|_{L^{p(\cdot)}(X,\mu)} \\ 
		& \leq \kappa_{p(\cdot)} 2^{\frac{1}{q_2^-}} \left(\frac{1}{\left(1-2^{-s^-q_2^-}\right)^{\frac{1}{q_2^-}}}\left\|u\right\|_{L^{p(\cdot)}(X,\mu)}+\frac{1}{\left(1-2^{-\varepsilon^-q_2^-}\right)^{\frac{1}{q_2^-}}}\left\|g\right\|_{L^{p(\cdot)}(X,\mu)}\right),
	\end{align*}
	where $\kappa_{p(\cdot)}\in[1,\infty)$ is the constant appearing in the quasi-triangle inequality satisfied by the quasi-norm $\left\| \cdot\right\|_{L^{p(\cdot)}(X,\mu)}$.
	Furthermore, if $q_2^- = \infty$, then $q_2\equiv \infty$ and we have
	\begin{equation*}
		\left\| \left\{f_k\right\}_{k\in\mathbb Z} \right\|_{L^{p(\cdot)}(\ell^{\infty}(X,\mu))} \leq \kappa_{p(\cdot)}\left(\left\| u \right\|_{L^{p(\cdot)}(X,\mu)} + \left\|g\right\|_{L^{p(\cdot)}(X,\mu)}\right).
	\end{equation*}Hence, the claim  in $(iv)$ is proven.
	
	$(v)$ Using $(i)$ and $(ii)$ of the current proposition, we get
	\begin{equation*}
		\dot{M}^{s(\cdot)}_{p(\cdot),q(\cdot)}(X,d,\mu) \hookrightarrow \dot{M}^{s(\cdot)}_{p(\cdot),\infty}(X,d,\mu)=\dot{M}^{s(\cdot),p(\cdot)}(X,d,\mu),
	\end{equation*}
	and hence $(v)$ is proven.
	
	$(vi)$ Using $(i)$, $(iii)$, and $(v)$ of the current proposition, we get
	\begin{equation*}
		\dot{N}^{s(\cdot)}_{p(\cdot),q(\cdot)}(X,d,\mu) \hookrightarrow \dot{N}^{s(\cdot)}_{p(\cdot),p(\cdot)}(X,d,\mu)=\dot{M}^{s(\cdot)}_{p(\cdot),p(\cdot)}(X,d,\mu) \hookrightarrow \dot{M}^{s(\cdot),p(\cdot)}(X,d,\mu),
	\end{equation*}
	which proves $(vi)$. 
	
	$(vii)$ Applying $(iv)$ and $(vi)$ we have
	\begin{equation*}
		N^{s(\cdot)}_{p(\cdot),q(\cdot)}(X,d,\mu) \hookrightarrow N^{t(\cdot)}_{p(\cdot),p(\cdot)/2}(X,d,\mu) \hookrightarrow M^{t(\cdot),p(\cdot)}(X,d,\mu)
	\end{equation*}
	and the proof of $(vii)$ is done.
	
	$(viii)$ By $(v)$ of the current proposition, it suffices to prove that
	\begin{equation*}
		\dot{M}^{s(\cdot),p(\cdot)}(X,d,\mu) \hookrightarrow \dot{N}^{s(\cdot)}_{p(\cdot),\infty}(X,d,\mu).
	\end{equation*}
	Let $u\in \dot{M}^{s(\cdot),p(\cdot)}(X,d,\mu)$ and $g\in \mathcal{D}^{s(\cdot)}(u)\cap L^{p(\cdot)}(X,\mu)$. For every $k\in \mathbb Z$ we define $g_k:=g$. Then, $\left\{g_k\right\}_{k\in \mathbb Z}$ is fractional $s(\cdot)$-gradient of $u$ and
	\begin{equation*}
		\left\| u \right\|_{\dot{N}^{s(\cdot)}_{p(\cdot),\infty}(X,d,\mu)}\leq \left\| \left\{g_k\right\}_{k\in \mathbb Z} \right\|_{\ell^{\infty}\left(L^{p(\cdot)}(X,\mu)\right)} = \sup_{k\in \mathbb Z} \left\| g_k \right\|_{L^{p(\cdot)}(X,\mu)} = \left\|g \right\|_{L^{p(\cdot)}(X,\mu)}.
	\end{equation*}
	Since $g$ was an arbitrary scalar $s(\cdot)$-gradient, we get
	\begin{equation*}
		\left\| u \right\|_{\dot{N}^{s(\cdot)}_{p(\cdot),\infty}(X,d,\mu)} \leq \left\| u \right\|_{\dot{M}^{s(\cdot),p(\cdot)}(X,d,\mu)}
	\end{equation*}
	and $(viii)$ is proven.
	
	$(ix)$ Since by $(i)$ of the current proposition, we have
	\begin{equation*}
		\dot{N}^{s(\cdot)}_{p(\cdot),q(\cdot)}(B,d,\mu) \hookrightarrow \dot{N}^{s(\cdot)}_{p(\cdot),\infty}(B,d,\mu)
	\end{equation*}
	with appropriate control of quasi-norms, it suffices to prove that there  exists constant $\zeta(p,s,t,\delta)\in (0,\infty)$, independent of $B$, such that
	\begin{equation*}
		\left\|u\right\|_{\dot{M}^{t(\cdot),p(\cdot)}(B,d,\mu)} \leq \zeta(p,s,t,\delta) \left\|u\right\|_{\dot{N}^{s(\cdot)}_{p(\cdot),\infty}(B,d,\mu)}
	\end{equation*}
	holds for every $u\in \dot{N}^{s(\cdot)}_{p(\cdot),\infty}(B,d,\mu).$ Let $u\in \dot{N}^{s(\cdot)}_{p(\cdot),\infty}(B,d,\mu)$ be fixed. Without loss of generality, we can assume that $\left\|u\right\|_{\dot{N}^{s(\cdot)}_{p(\cdot),\infty}(B,d,\mu)}<1.$ Thus, it suffices to prove that
	\begin{equation*}
		\left\|u \right\|_{\dot{M}^{t(\cdot),p(\cdot)}(B,d,\mu)} \leq \zeta(p,s,t,\delta).
	\end{equation*}
	We take $\left\{g_k\right\}_{k\in \mathbb Z} \in \mathbb{D}^{s(\cdot)}(u)$ such that $\left\| \left\{g_k\right\}_{k\in\mathbb Z} \right\|_{\ell^{\infty}(L^{p(\cdot)}(B))}\leq 1.$ Let $k_0\in \mathbb Z$ be such that $2^{-k_0-1} \leq 2r_0 <2^{-k_0}.$ Define
	\begin{equation*}
		g(x):=\left[ \sum_{k=k_0}^{\infty} 2^{-k\left(s(x)-t(x)\right)p(x)} g_k(x)^{p(x)} \right]^{\frac{1}{p(x)}}.
	\end{equation*}
	We claim that $g\in \mathcal{D}^{t(\cdot)}(u).$ Indeed, for $y,z\in B$ we have $d(y,z)<2r_0<2^{-k_0}.$ Hence, for every $y,z\in B$ there exists $j\geq k_0$ such that $2^{-j-1} \leq d(y,z) <2^{-j}.$ Since $\left\{g_k\right\}_{k\in\mathbb Z}\in \mathbb{D}^{s(\cdot)}(u)$, there exists a null set $N_u\subseteq B$ such that
	\begin{align*}
	\left|u(y)-u(z)\right| & \leq d(y,z)^{s(y)}g_j(y)+d(y,z)^{s(z)}g_j(z) \leq d(y,z)^{t(y)}g(y)+d(y,z)^{t(z)}g(z),
	\end{align*}
	whenever $y,z\in B \setminus N_u.$ This means that $g\in \mathcal{D}^{t(\cdot)}(u).$ On the other hand, since $k_0 \geq -2-\log_2 r_0$, we get
	\begin{align*}
	\rho_{p(\cdot)}(g)&=\int_B \sum_{k=k_0}^{\infty} 2^{-k\left(s(x)-t(x)\right)p(x)}g_k(x)^{p(x)} \mbox{d}\mu(x)\\ & = \int_B 2^{\left(2+\log_2 r_0\right)\left(s(x)-t(x)\right)p(x)}\sum_{k=k_0}^{\infty} 2^{-\left(k+2+\log_2 r_0\right)\left(s(x)-t(x)\right)p(x)}g_k(x)^{p(x)}\mbox{d}\mu(x) \\ & \leq \int_B (4 \delta)^{\left(s(x)-t(x)\right)p(x)} \sum_{k=k_0}^{\infty} 2^{-\left(k+2+\log_2 r_0\right)\left(s-t\right)_B^- p_B^-}g_k(x)^{p(x)}\mbox{d}\mu(x) \\ & \leq \max\left\{(4\delta)^{\left(s-t\right)^+ p^+}, (4\delta)^{\left(s-t\right)^- p^-}\right\}\frac{2^{-\left(k_0+2+\log_2 r_0\right)p_B^-\left(s-t\right)_B^-}}{1-2^{-p_B^-\left(s-t\right)_B^-}} \leq \frac{\max\left\{(4\delta)^{\left(s-t\right)^+ p^+}, (4\delta)^{\left(s-t\right)^- p^-}\right\}}{1-2^{-p^- \left(s-t\right)^-}}.
	\end{align*}
	Thus,
	\begin{equation*}
	\left\| g\right\|_{L^{p(\cdot)}(B,\mu)} \leq \max\left\{\rho_{p(\cdot)}(g)^{\frac{1}{p_B^-}}, \rho_{p(\cdot)}(g)^{\frac{1}{p_B^+}}\right\} \leq \zeta(p,s,t,\delta),
	\end{equation*}
	which completes the proof of $(ix)$ and hence, the proof of Proposition~\ref{embeddingsbetween}.
\end{proof}

\subsection{Local Sobolev inequalities on balls}\label{elbm}

\begin{tw}\label{sobolevpoincare}
	Let $(X,d,\mu)$ be a metric measure space and assume that there exist $Q\in \mathcal{P}_b^{\log}(X)$, $\delta \in (0,\infty)$, and $b\in (0,1]$\footnote{If $b> 1$, then we can take $\tilde{b}:=\min\left\{1,b\right\}$.} such that for every $x\in X$ and $r\in (0,\delta]$,
	\begin{equation*}
		\mu\left(B(x,r)\right)\geq br^{Q(x)}.
	\end{equation*}
	Let $p,s\in \mathcal{P}_b^{\log}(X)$ and $\sigma \in (1,\infty)$. Then, the following statements are valid.
	\begin{enumerate}
		\item[(i)] If $sp \ll Q$, then, there exists a positive constant $C_{S}$, depending only on $b$, $\sigma$, $p$, $s$, $\delta$ and $Q$, such that, for every ball $B_0:=B(x_0,r_0)\subseteq X$ with $r_0\leq 
		\delta/\sigma$, and every pair of functions $u\in \dot{M}^{s(\cdot),p(\cdot)}(\sigma B_0)$ and $g\in \mathcal{D}^{s(\cdot)}(u)\cap L^{p(\cdot)}(\sigma B_0)$, there holds
		\begin{equation}\label{sobolevloc}
			\inf_{c\in \mathbb R}\left\|u-c\right\|_{L^{\gamma(\cdot)}(B_0)} \leq C_{S}\left(\frac{\mu(B_0)}{r_0^{Q(x_0)}}\right)^{\frac{1}{\gamma^-_{B_0}}} \left\|g\right\|_{L^{p(\cdot)}(\sigma B_0)},
		\end{equation}
		where $\displaystyle \gamma:=\frac{Qp}{Q-sp}$.
		\item[(ii)] If $sp=Q$, then there exist positive constants $C_{MT1}, C_{MT2}$, depending only on $b$, $\sigma$, $p$, $s$, $\delta$ and $Q$, such that for every ball $B_0:=B(x_0,r_0)\subseteq X$ with $r_0\leq \delta/\sigma$, and every pair of functions $u\in \dot{M}^{s(\cdot),p(\cdot)}(\sigma B_0)$ and $g\in \mathcal{D}^{s(\cdot)}(u) \cap L^{p(\cdot)}(\sigma B_0)$ with $ \left\|g\right\|_{L^{p(\cdot)}(\sigma B_0)}> 0$, there holds
		\begin{equation}\label{moser2}
			\fint_{B_0} \exp\left(C_{MT1}\frac{\left|u(x)-u_{B_0}\right|}{\left\|g\right\|_{L^{p(\cdot)}(\sigma B_0)}}\right)\mbox{d}\mu(x) \leq C_{MT2}.
		\end{equation}
		\item[(iii)] If $sp \gg Q$, then there exists a positive constant $C_H$, depending only on $b$, $\sigma$, $p$, $s$, $\delta$ and $Q$, such that for every ball $B_0:=B(x_0,r_0)\subseteq X$ with $r_0 \leq \delta/\sigma$, and every pair of functions $u\in \dot{M}^{s(\cdot),p(\cdot)}(\sigma B_0)$ and $g\in \mathcal{D}^{s(\cdot)}(u)\cap L^{p(\cdot)}(\sigma B_0)$, there holds
		\begin{equation}\label{morrey}
			\left\|u-u_{B_0}\right\|_{L^{\infty}(B_0)}\leq C_Hr_0^{\alpha(x_0)}\left\|g\right\|_{L^{p(\cdot)}(\sigma B_0)},
		\end{equation}
		where $\displaystyle \alpha:=s-\frac{Q}{p}$.
		Moreover, the function $u$ has a continuous representative $\tilde{u}$ on $B_0$ satisfying that, for all $x,y\in B_0$,
		\begin{equation}\label{holderembedding}
			\left|\tilde{u}(x)-\tilde{u}(y)\right| \leq D_H(r_0) \|g\|_{L^{p(\cdot)}(\sigma B_0)}d(x,y)^{\alpha(x)},
		\end{equation}
		where
		\begin{equation*}
				D_H(r_0):=2^{\alpha^++1}C_H \left(\frac{\sigma \delta}{(\sigma -1)r_0}\right)^{\alpha^+}.
			\end{equation*}
	\end{enumerate}
\end{tw}

\begin{proof} 
Let $\beta_p:=4+2^{\frac{1}{p^-}}.$ We begin by making a series of observations. First of all, by a standard scaling argument we can assume that $\delta=1$. Moreover, let us observe that to prove \eqref{sobolevloc}, \eqref{moser2}, and \eqref{morrey}, it is enough to prove that 
	\begin{equation}\label{one1}
		\inf_{c\in \mathbb R}\left\|u-c\right\|_{L^{\gamma(\cdot)}(B_0)}\leq \frac{C_{S}}{\beta_p}\left(\frac{\mu(B_0)}{r_0^{Q(x_0)}}\right)^{\frac{1}{\gamma_{B_0}^-}}, 
	\end{equation}
	\begin{equation}\label{one1moser}
		\fint_{B_0} \exp\left(\beta_p C_{MT1}\left|u(x)-u_{B_0}\right|\right)\mbox{d}\mu(x) \leq C_{MT2},
	\end{equation}
	and
	\begin{equation}\label{twored1}
		\left\|u-u_{B_0}\right\|_{L^{\infty}(B_0)}\leq \frac{C_H}{\beta_p}r_0^{\alpha(x_0)}
	\end{equation}
	hold for all $u\in \dot{M}^{s(\cdot),p(\cdot)}(\sigma B_0)$ and $g \in \mathcal{D}^{s(\cdot)}(u)$ satisfying
	\begin{equation}\label{1/6}
		\frac{1}{\beta_p} \leq \left\|g\right\|_{L^{p(\cdot)}(\sigma B_0)}\leq 1
	\end{equation}
	and
	\begin{equation}\label{lowerbound}
		g(x)^{p(x)}\geq \frac{1}{2}\fint_{\sigma B_0}g(y)^{p(y)}\mbox{d}\mu(y) >0,
	\end{equation}
	for every $x \in \sigma B_0$.
	Indeed, let $u\in \dot{M}^{s(\cdot),p(\cdot)}(\sigma B_0)$ and set
	\begin{equation*}
		A:=\inf_{g\in \mathcal{D}^{s(\cdot)}(u)} \left\|g\right\|_{L^{p(\cdot)}(\sigma B_0)}.    
	\end{equation*}
	If $A=0$, then $u$ is constant by Lemma~\ref{gradientzero}, and hence \eqref{sobolevloc}, \eqref{moser2}, and \eqref{morrey} are satisfied. Therefore, we can assume that $A>0$. Let $g\in \mathcal{D}^{s(\cdot)}(u)\cap L^{p(\cdot)}(\sigma B_0)$ and let us define 
	\begin{equation*}
		\tilde{u}:=\frac{u}{4A+2^{\frac{1}{p^-}}\left\|g\right\|_{L^{p(\cdot)}(\sigma B_0)}}.
	\end{equation*}
	Then $\displaystyle \tilde{g}:=\frac{g}{4A+2^{\frac{1}{p^-}}\|g\|_{L^{p(\cdot)}(\sigma B_0)}} \in \mathcal{D}^{s(\cdot)}(\tilde{u})$ and 
	\begin{equation*}
		\frac{1}{\beta_p} \leq \|\tilde{g}\|_{L^{p(\cdot)}(\sigma B_0)} \leq \frac{1}{2^{\frac{1}{p^-}}} < 1.
	\end{equation*}
	Next, let us define $\hat{g}$ as follows
	\begin{equation*}
		\hat{g}(x)^{p(x)}:=\tilde{g}(x)^{p(x)}+\fint_{\sigma B_0}\tilde{g}(y)^{p(y)}\mbox{d}\mu(y).
	\end{equation*}
	Then, $\hat{g} \in \mathcal{D}^{s(\cdot)}(\tilde{u})$,  $\hat{g}(x)\geq \tilde{g}(x)$ for all $x\in \sigma B_0$, and since 
	\begin{equation*}
		\int_{\sigma B_0}\tilde{g}(y)^{p(y)}\mbox{d}\mu(y) \leq    \left\|\tilde{g}\right\|_{L^{p(\cdot)}(\sigma B_0)}^{p^-} \leq \frac{1}{2},
	\end{equation*}
	which is a consequence of Proposition~\ref{ballprop} and the fact that $\left\|\tilde{g}\right\|_{L^{p(\cdot)}(\sigma B_0)} \leq 1$,
	we have  that $\displaystyle\int_{\sigma B_0}\hat{g}(x)^{p(x)}\mbox{d}\mu(x)\leq 1.$ Therefore, 
	\begin{equation*}
		\frac{1}{\beta_p} \leq \left\|\hat{g}\right\|_{L^{p(\cdot)}(\sigma B_0)}\leq 1 \, \, \, \, \textnormal{ and } \, \, \, \,
		\hat{g}(x)^{p(x)}\geq \frac{1}{2}\fint_{\sigma B_0}\hat{g}(y)^{p(y)}\mbox{d}\mu(y) >0,
	\end{equation*}
	for every $x \in \sigma B_0$.
	Hence, replacing $u$ with $\tilde{u}$ in \eqref{one1}, \eqref{one1moser}, and \eqref{twored1}, respectively, we get 
	\begin{align*}
		\inf_{c\in \mathbb R}\left\|u-c\right\|_{L^{\gamma(\cdot)}(B_0)}&\leq \frac{C_{S}}{\beta_p}\left(\frac{\mu(B_0)}{r_0^{Q(x_0)}}\right)^{\frac{1}{\gamma_{B_0}^-}} \left(4A+2^{\frac{1}{p^-}} \left\| g \right\|_{L^{p(\cdot)}(\sigma B_0)}\right)\\
		&\leq \frac{C_{S}}{\beta_p}\left(\frac{\mu(B_0)}{r_0^{Q(x_0)}}\right)^{\frac{1}{\gamma_{B_0}^-}} \beta_p \left\| g \right\|_{L^{p(\cdot)}(\sigma B_0)}=C_{S}\left(\frac{\mu(B_0)}{r_0^{Q(x_0)}}\right)^{\frac{1}{\gamma_{B_0}^-}}\left\|g\right\|_{L^{p(\cdot)}\left(\sigma B_0 \right)},
	\end{align*}
	\begin{align*}
		\fint_{B_0}\exp\left(C_{MT1}\frac{\left|u(x)-u_{B_0}\right|}{\left\|g\right\|_{L^{p(\cdot)}(\sigma B_0)}}\right) \mbox{d}\mu(x) &\leq \fint_{B_0} \exp\left(\beta_pC_{MT1}\frac{\left|u(x)-u_{B_0}\right|}{4A+2^{\frac{1}{p^-}}\left\|g\right\|_{L^{p(\cdot)}(\sigma B_0)}}\right)\mbox{d}\mu(x) \\&=\fint_{B_0} \exp\left(\beta_p C_{MT1} \left|\tilde{u}(x)-\tilde{u}_{B_0}\right|\right)\mbox{d}\mu(x)  \leq C_{MT2},
	\end{align*}
	and
	\begin{equation*}
		\left\|u - u_{B_0} \right\|_{L^{\infty}(B_0)} \leq \frac{C_H}{\beta_p}r_0^{\alpha(x_0)}\beta_p \left\|g\right\|_{L^{p(\cdot)}(\sigma B_0)}=C_H r_0^{\alpha(x_0)} \left\|g\right\|_{L^{p(\cdot)}\left(\sigma B_0\right)},
	\end{equation*}
	and our claim follows.
	
	Next, by Proposition~\ref{rel} for any $c\in \mathbb R$ we have
	\begin{equation*}
		\left\| u-c \right\|_{L^{\gamma(\cdot)}(B_0)}\leq 1+\left(\int_{B_0}\left| u(x)-c \right|^{\gamma(x)} \mbox{d}\mu(x)\right)^{\frac{1}{\gamma_{B_0}^-}}.
	\end{equation*}
	Therefore, to prove \eqref{one1} it is enough to prove the inequality
	\begin{equation}\label{redukcja1}
		\inf_{c\in \mathbb R}\int_{B_0}\left| u(x)-c \right|^{\gamma(x)}\mbox{d}\mu(x) \leq C\frac{\mu(B_0)}{r_0^{Q(x_0)}} 
	\end{equation}
	for some constant $C>0$.
	
	\indent If $E\subseteq \sigma B_0$ is a set of positive measure, we can always find $v_0\in\mathbb{R}$ such that
	\begin{equation*}
		\essinf_{x\in E} \left|u(x)-v_0\right|=0.
	\end{equation*} 
	We may replace $u$ by $u-v_0$ (since subtracting a constant from $u$ will not affect the inequalities \eqref{one1moser}, \eqref{twored1} and \eqref{redukcja1}) and assume that 
	\begin{equation*}
		\essinf_{x\in E} \left|u(x)\right|=0.
	\end{equation*} 
	With a correct choice of $E$ we will prove \eqref{one1moser}, \eqref{twored1} and \eqref{redukcja1}. Moreover, to prove \eqref{redukcja1} it suffices to prove that
	\begin{equation}\label{redukcja2}
		\int_{B_0} \left|u(x)\right|^{\gamma(x)} \mbox{d}\mu(x) \leq C\frac{\mu(B_0)}{r_0^{Q(x_0)}}.
	\end{equation}

	Let $N\subseteq \sigma B_0$ be the set of measure zero for which the pointwise estimate \eqref{pointwise} holds for $x,y\in \sigma B_0\setminus N.$ Consider the sets, for every $k\in\mathbb{Z},$
	\begin{equation*}
		E_k:=\left\{x\in \sigma B_0: g(x)^{p(x)}\leq 2^k\right\}\setminus N.
	\end{equation*}
	Clearly $E_k\subseteq E_{k+1}$ and the measure of the complement of $E_k$ can be estimated by Chebyshev's inequality as follows:
	\begin{equation}\label{Chebyschev1}
		\mu\left(\sigma B_0\setminus E_k\right)=\mu\left(\left\{x\in \sigma B_0: g(x)^{p(x)}>2^k\right\}\right)\leq 2^{-k}\int_{\sigma B_0}g(x)^{p(x)} \mbox{d}\mu(x).
	\end{equation}
	Note that
	\begin{equation}\label{g1}
		\frac{1}{2} \sum_{k\in\mathbb{Z}}2^k\mu(E_k\setminus E_{k-1}) \leq \int_{\sigma B_0}g(x)^{p(x)} \mbox{d}\mu(x) \leq \sum_{k\in\mathbb{Z}}2^k\mu(E_k\setminus E_{k-1}),
	\end{equation}
	and
	\begin{equation}\label{g1-X}
	\mu\left(\sigma B_0\setminus\bigcup_{k\in\mathbb{Z}}(E_k\setminus E_{k-1})\right)=0.
		\end{equation}
	Moreover, as $k\rightarrow\infty,$ $\mu(E_k)\rightarrow \mu(\sigma B_0)$ and \eqref{lowerbound} ensures that $E_k=\emptyset$ for all sufficiently small $k\in\mathbb Z$. Hence, there exists $\widetilde{k}_0\in\mathbb{Z}$ such that
	\begin{equation}\label{conv1}
		\mu\left(E_{\widetilde{k}_0-1}\right)<\frac{\mu\left(\sigma B_0\right)}{2}\leq\mu\left(E_{\widetilde{k}_0}\right).
	\end{equation}
	The first inequality of \eqref{conv1}, together with \eqref{Chebyschev1}, imply that
	\begin{equation}\label{second1}
		\frac{\mu\left(\sigma B_0\right)}{2}<\mu\left(\sigma B_0\setminus E_{\widetilde{k}_0-1}\right)\leq 2^{-\left(\widetilde{k}_0-1\right)}\int_{\sigma B_0}g(x)^{p(x)} \mbox{d}\mu(x).
	\end{equation}
	Since $\mu\left(E_{\widetilde{k}_0}\right) >0$, there exists $y\in E_{\widetilde{k}_0}$ and from \eqref{lowerbound} we obtain
	\begin{equation*}
		\frac{1}{2}\fint_{\sigma B_0}g(x)^{p(x)} \mbox{d}\mu(x) \leq g(y)^{p(y)}\leq 2^{\widetilde{k}_0}.
	\end{equation*}
	Combining this and \eqref{second1} gives
	\begin{equation}\label{tildekzero1}
		\frac{1}{2}\fint_{\sigma B_0}g(x)^{p(x)} \mbox{d}\mu(x) \leq 2^{\widetilde{k}_0}\leq 4\fint_{\sigma B_0}g(x)^{p(x)} \mbox{d}\mu(x).
	\end{equation}
	Choose the least integer $\ell\in\mathbb{Z}$ such that
	\begin{equation}\label{ell1}
		2^{\ell}>M(b,\sigma,Q,p,r_0)\frac{\mu(\sigma B_0)}{br_0^{Q_{\sigma B_0}^+}},
	\end{equation}
	where
	\begin{equation*}
		M(b,\sigma,Q,p,r_0):=\max\left\{ \frac{b^{\frac{Q_{\sigma B_0}^- - Q_{\sigma B_0}^+}{Q_{\sigma B_0}^-}}2^{2Q_{\sigma B_0}^+ +1}}{(\sigma -1)^{Q_{\sigma B_0}^+}\left(1-2^{-\frac{1}{Q_{\sigma B_0}^+}}\right)^{Q_{\sigma B_0}^+}},\, \frac{2^{2Q_{\sigma B_0}^- +1}r_0^{Q_{\sigma B_0}^+ - Q_{\sigma B_0}^-}\beta_p^{p_{\sigma B_0}^+ \left(1-\frac{Q_{\sigma B_0}^-}{Q_{\sigma B_0}^+}\right)}}{\left(\sigma -1\right)^{Q_{\sigma B_0}^-} \left(1-2^{-\frac{1}{Q_{\sigma B_0}^+}}\right)^{Q_{\sigma B_0}^-}},\, 1\right\},
	\end{equation*}
	and set $k_0:=\widetilde{k}_0+\ell.$ Note that $\ell>0$ by the lower bound of measure and hence $k_0>\widetilde{k}_0$, which means that $\mu(E_{k_0})>0,$ by \eqref{conv1}. With the above choice of $k_0,$ the inequalities in \eqref{tildekzero1} become
	\begin{equation}\label{kzero1}
		L(b,\sigma,Q,p,r_0) \frac{1}{br_0^{Q_{\sigma B_0}^+}}\int_{\sigma B_0}g(x)^{p(x)} \mbox{d}\mu(x) \leq 2^{k_0}\leq U(b,\sigma,Q,p,r_0)\frac{1}{br_0^{Q_{\sigma B_0}^+}}\int_{\sigma B_0}g(x)^{p(x)} \mbox{d}\mu(x)
	\end{equation}
	where
	\begin{align*}
		L(b,\sigma,Q,p,r_0)&:=\frac{1}{2}M(b,\sigma,Q,p,r_0),\,\, \\ U(b,\sigma,Q,p,r_0)&:=8M(b,\sigma,Q,p,r_0).
	\end{align*}
	Let us suppose that $\mu(B_0\setminus E_{k_0})>0$ (in the other case the proof is even simpler and it will be handled later). Let us also suppose that $k\in\mathbb Z$ satisfies $k\geq k_0+1$ and $\mu((E_k\setminus E_{k-1})\cap B_0)>0$ (if such a $k\in\mathbb Z$ did not exist, then we would have $\mu(B_0\setminus E_{k_0})=0,$ contradicting our assumption). For $k_0+1 \leq j\leq k$ and $y\in \sigma B_0$ set
	\begin{equation}\label{radii1}
		t_j(y):=2b^{-\frac{1}{Q(y)}}\mu\left(\sigma B_0\setminus E_{j-1}\right)^{\frac{1}{Q(y)}}.
	\end{equation}
	Then in particular $t_j(y)>0.$ Moreover, we claim that for every $z_k, z_{k-1},\dots,z_{k_0+1} \in \sigma B_0$
	\begin{equation}\label{suma}
		t_k(z_k)+t_{k-1}(z_{k-1})+\dots+t_{k_0+1}(z_{k_0+1}) \leq (\sigma -1)r_0.
	\end{equation}
	Indeed, for $j=k_0,\dots,k-1$ we have
	\begin{align}\label{estimate1}
		\sum_{n=j+1}^k t_n(z_n)& \leq 2 b^{-\frac{1}{Q_{\sigma B_0}^-}} \left(\int_{\sigma B_0} g(x)^{p(x)} \mbox{d}\mu(x)\right)^{\frac{1}{Q_{\sigma B_0}^+}} \sum_{n=j}^{k-1} 2^{-\frac{n}{Q(z_{n+1})}} \nonumber \\ & \leq 4b^{-\frac{1}{Q_{\sigma B_0}^-}} \left( \int_{\sigma B_0} g(x)^{p(x)} \mbox{d}\mu(x) \right)^{\frac{1}{Q_{\sigma B_0}^+}} \frac{1}{1-2^{-\frac{1}{Q_{\sigma B_0}^+}}} \max\left\{2^{-\frac{j}{Q_{\sigma B_0}^+}}, 2^{-\frac{j}{Q_{\sigma B_0}^-}}\right\}.
	\end{align}
	In particular, in \eqref{estimate1} we can take $j=k_0$. If $k_0\geq 0$, then
	\begin{align}\label{tyt-35}	
	\sum_{n=k_0+1}^k t_n(z_n)& \leq \frac{4}{1-2^{-\frac{1}{Q_{\sigma B_0}^+}}} b^{-\frac{1}{Q_{\sigma B_0}^-}} 2^{-\frac{k_0}{Q_{\sigma B_0}^+}} \left( \int_{\sigma B_0} g(x)^{p(x)} \mbox{d}\mu(x) \right)^{\frac{1}{Q_{\sigma B_0}^+}}.
	\end{align}
	Moreover,
	\begin{align*}
		2^{k_0}=2^{\widetilde{k}_0+\ell} & \geq 2^{\ell}\cdot \frac{1}{2}\fint_{\sigma B_0} g(x)^{p(x)} \mbox{d}\mu(x) \\ & \geq \frac{\mu(\sigma B_0)}{br_0^{Q_{\sigma B_0}^+}} \frac{b^{\frac{Q_{\sigma B_0}^- - Q_{\sigma B_0}^+}{Q_{\sigma B_0}^-}} 2^{2Q_{\sigma B_0}^+ +1} }{(\sigma -1)^{Q_{\sigma B_0}^+} \left(1-2^{-\frac{1}{Q_{\sigma B_0}^+}}\right)^{Q_{\sigma B_0}^+}}\cdot \frac{1}{2} \fint_{\sigma B_0} g(x)^{p(x)} \mbox{d}\mu(x) \\ & = \frac{4^{Q_{\sigma B_0}^+} b^{-\frac{Q_{\sigma B_0}^+}{Q_{\sigma B_0}^-}}}{(\sigma -1)^{Q_{\sigma B_0}^+}r_0^{Q_{\sigma B_0}^+}\left(1-2^{-\frac{1}{Q_{\sigma B_0}^+}}\right)^{Q_{\sigma B_0}^+}}\int_{\sigma B_0} g(x)^{p(x)} \mbox{d}\mu(x).
	\end{align*}
	Hence,
	\begin{equation*}
		2^{-\frac{k_0}{Q_{\sigma B_0}^+}} \leq \frac{(\sigma-1)r_0}{4b^{-\frac{1}{Q_{\sigma B_0}^-}}}\left(1-2^{-\frac{1}{Q_{\sigma B_0}^+}}\right) \left( \int_{\sigma B_0} g(x)^{p(x)} \mbox{d}\mu(x) \right)^{-\frac{1}{Q_{\sigma B_0}^+}},
	\end{equation*}
	which, together with \eqref{tyt-35}, yields \eqref{suma}.
	
	If $k_0 < 0$, then
	\begin{equation*}
			\sum_{n=k_0+1}^k t_n(z_n) \leq \frac{4}{1-2^{-\frac{1}{Q_{\sigma B_0}^+}}} b^{-\frac{1}{Q_{\sigma B_0}^-}} 2^{-\frac{k_0}{Q_{\sigma B_0}^-}} \left( \int_{\sigma B_0} g(x)^{p(x)} \mbox{d}\mu(x) \right)^{\frac{1}{Q_{\sigma B_0}^+}}.
	\end{equation*}
	Moreover,
	\begin{align*}
	2^{-\frac{k_0}{Q_{\sigma B_0}^-}} & = \frac{2^{-\frac{\widetilde{k}_0}{Q_{\sigma B_0}^-}}}{2^{\frac{\ell}{Q_{\sigma B_0}^-}}} \leq \frac{\mu(\sigma B_0)^{\frac{1}{Q_{\sigma B_0}^-}}}{2^{\frac{\ell}{Q_{\sigma B_0}^-}} }\left(\frac{1}{2}\int_{\sigma B_0} g(x)^{p(x)} \mbox{d}\mu(x) \right)^{-\frac{1}{Q_{\sigma B_0}^-}} \\ & \leq 2^{-\frac{1}{Q_{\sigma B_0}^-}}\left(1-2^{-\frac{1}{Q_{\sigma B_0}^+}}\right)  \frac{b^{\frac{1}{Q_{\sigma B_0}^-}}(\sigma -1)r_0}{4 \beta_p^{p_{\sigma B_0}^+\left(\frac{1}{Q_{\sigma B_0}^-} - \frac{1}{Q_{\sigma B_0}^+}\right) }} \left(\frac{1}{2} \int_{\sigma B_0} g(x)^{p(x)} \mbox{d}\mu(x) \right)^{-\frac{1}{Q_{\sigma B_0}^-}} \\ & \leq \frac{1-2^{-\frac{1}{Q_{\sigma B_0}^+}}}{4} b^{\frac{1}{Q_{\sigma B_0}^-}} \left( \int_{\sigma B_0} g(x)^{p(x)} \mbox{d}\mu(x) \right)^{-\frac{1}{Q_{\sigma B_0}^+}}(\sigma -1)r_0.
	\end{align*}
	Hence, we once again obtain \eqref{suma}.
	
	In particular, since $\sigma r_0\leq1$, for every $y\in \sigma B_0$, we have $t_j(y) \leq 1$. Hence, notice that for every $k_0+1\leq j \leq k$ and $y\in \sigma B_0$ such that $B(y,t_j(y))\subseteq \sigma B_0$ we have $B(y,t_j(y)) \cap E_{j-1} \neq \emptyset$. Indeed, assume that $B(y,t_j(y)) \cap E_{j-1}= \emptyset$. Then, by the lower Ahlfors-regularity of the measure and the definition of $t_j(y)$, we have
	\begin{equation*}
		\mu(\sigma B_0 \setminus E_{j-1})\geq \mu(B(y,t_j(y))\setminus E_{j-1})=\mu(B(y,t_j(y)))>\mu(\sigma B_0 \setminus E_{j-1}),
	\end{equation*}
	which is a clear contradiction. Now, suppose there exists a point $y_k \in E_k \cap B_0$. Then, in particular, we have $B(y_k,t_k(y_k)) \subseteq \sigma B_0$ since, for each $\xi \in B(y_k,t_k(y_k))$, we can use \eqref{suma} to estimate
	\begin{equation*}
		d(\xi,x_0)\leq d(\xi,y_k)+d(y_k,x_0)<t_k(y_k)+r_0\leq (\sigma-1)r_0+r_0=\sigma r_0.
	\end{equation*}
	Hence, $B(y_k,t_k(y_k))\cap E_{k-1}\neq\emptyset$ and thus we can choose a point $y_{k-1}\in E_{k-1}$ such that
	\begin{equation*}
		d(y_k,y_{k-1})<t_k(y_k).
	\end{equation*}
By induction, we obtain a sequence of points $\left\{y_j\right\}_{j={k_0}}^{k}$
	\begin{gather*}
		y_k \in E_k\cap B_0,\\
		y_{k-1} \in E_{k-1}\cap B(y_k,t_k(y_k)),\\
		y_{k-2} \in E_{k-2}\cap B(y_{k-1},t_{k-1}(y_{k-1})),\\
		\vdots  \\
		%x_{k-i} \in  E_{k-i}\cap B(x_{k-(i-1)},r_{k-(i-1)}),\\
		%\vdots  \\
		y_{k_0}\in  E_{k_0}\cap B(y_{k_0+1},t_{k_0+1}(y_{k_0+1})).
	\end{gather*}
	Recall that we are currently assuming $k\geq k_0+1$ is such that $\mu\left(\left(E_k\setminus E_{k-1}\right) \cap B_0\right)>0.$ Now, since  $\displaystyle \int_{\sigma B_0}g(x)^{p(x)} \mbox{d}\mu(x) \leq 1$, by \eqref{pointwise} and \eqref{Chebyschev1}, we have
	\begin{align*}
		\vert &u(y_k) \vert \leq \sum_{i=0}^{k-k_0-1}\left| u(y_{k-i})-u(y_{k-i-1})\right| + \left| u(y_{k_0})\right| \nonumber\\
		&\leq  \sum_{i=0}^{k-k_0-1}\left[d(y_{k-i}, y_{k-i-1})^{s(y_{k-i})}g(y_{k-i})+d(y_{k-i}, y_{k-i-1})^{s(y_{k-i-1})}g(y_{k-i-1})\right]+ \left| u(y_{k_0})\right| \nonumber\\
		& \leq				\sum_{i=0}^{k-k_0-1} \left[t_{k-i}(y_{k-i})^{s(y_{k-i})}2^{\frac{k-i}{p(y_{k-i})}} + t_{k-i}(y_{k-i})^{s(y_{k-i-1})}2^{\frac{k-i-1}{p(y_{k-i-1})}}\right] + \left|u(y_{k_0})\right| \nonumber\\
		& \leq				 C_1\left[\sum_{i=0}^{k-k_0}2^{(k-i)\left(\frac{1}{p(y_{k-i})}-\frac{s(y_{k-i})}{Q(y_{k-i})}\right)}\right.\nonumber\\
		&\left.\qquad\qquad\quad+\sum_{i=0}^{k-k_0-1}2^{(k-i-1)\left(\frac{1}{p(y_{k-i-1})}-\frac{s(y_{k-i-1})}{Q(y_{k-i-1})}\right)}\cdot 2^{-(k-i)\frac{s(y_{k-i-1})}{Q(y_{k-i})}+(k-i-1)\frac{s(y_{k-i-1})}{Q(y_{k-i-1})}}\right] + \left| u(y_{k_0})\right|,
	\end{align*}
	where $C_1=C_1(b,s^+,s^-,Q^-,Q^+)$ and we have used  the facts that $y_{k-i}\in E_{k-i}$ and $ y_{k-i-1}\in E_{k-i-1}$ for all $i=0,1,\ldots,k-k_0-1.$ 
	To estimate the above sum, we first find an upper bound for the term
	\begin{equation}\label{term}
		2^{-(k-i)\left( \frac{1}{Q(y_{k-i})}-\frac{1}{Q(y_{k-i-1})}\right)}.
	\end{equation}
	If $k-i\geq 0$, then the log-H\"older continuity of $1/Q$ implies that
	\begin{align*}
		\left|k-i\right|\left|\frac{1}{Q(y_{k-i})}-\frac{1}{Q(y_{k-i-1})}\right| &\leq (k-i)\frac{C_{\log}(1/Q)}{\log\left(e+1/t_{k-i}(y_{k-i})\right)}\\ &\leq (k-i) \frac{C_{\log}(1/Q)}{\log\left(e+2^{-1-\frac{1}{Q^-}}b^{\frac{1}{Q^-}}2^{\frac{k-i}{Q^+}}\right)}\leq C_2(b,Q^-,Q^+,C_{\log}(1/Q)),
	\end{align*}
	since the function
	\begin{equation*}
		[0,\infty) \ni x \longmapsto x\frac{C_{\log}(1/Q)}{\log\left(e+2^{-1-\frac{1}{Q^-}}b^{\frac{1}{Q^-}}2^{\frac{x}{Q^+}}\right)}
	\end{equation*}
	is bounded.
	
	If $k-i <0$, then
	\begin{align*}
		\left|k-i\right| \left|\frac{1}{Q(y_{k-i})} - \frac{1}{Q(y_{k-i-1})} \right| \leq \left|k_0+1\right| \frac{2}{Q^-}=\left(-\widetilde{k}_0-\ell-1\right)\frac{2}{Q^-}.
	\end{align*}
	Hence, by \eqref{1/6}, \eqref{tildekzero1} and \eqref{ell1} we get
	\begin{align*}
		2^{\left|k-i\right| \left|\frac{1}{Q(y_{k-i})} - \frac{1}{Q(y_{k-i-1})} \right|} \leq \left(2 \frac{\mu(\sigma B_0)}{\displaystyle\int_{\sigma B_0} g(x)^{p(x)} \mbox{d}\mu(x)} \frac{br_0^{Q_{\sigma B_0}^+}}{\mu(\sigma B_0)} \frac{1}{2} \right)^{\frac{2}{Q^-}} \leq b^{\frac{2}{Q^-}}\beta_p^{\frac{2p^+}{Q^-}}.
	\end{align*}
	Therefore, having in mind that $s$ is bounded, we know \eqref{term} is bounded and
	\begin{equation}\label{firstestimate1}
	\left|u(y_k)\right| \leq C_3\sum_{i=0}^{k-k_0} 2^{(k-i)\left(\frac{1}{p(y_{k-i})}-\frac{s(y_{k-i})}{Q(y_{k-i})}\right)}+\left|u(y_{k_0})\right|,
	\end{equation}
	where $C_3=C_3(b,s^+,s^-,Q^-,Q^+,C_{\log}(1/Q))$.
	
	Now we shall prove $(i)$. First of all, we assume for the moment that $\mu(B_0 \setminus E_{k_0})>0$ and $k\in\mathbb Z$ satisfies $k\geq k_0+1$ and $\mu((E_k\setminus E_{k-1})\cap B_0)>0$. Let us recall that it suffices to prove \eqref{redukcja2}. For every $i=1,\ldots,k-k_0,$ we have 
	\begin{equation*}
	y_{k-i}\in B(y_k, t_k(y_k)+t_{k-1}(y_{k-1})+\cdots+t_{k-i+1}(y_{k-i+1}))=:B(y_k,r_{k-i+1}).
\end{equation*} From \eqref{estimate1} for $j=k-i$ we have 
	\begin{equation*}
		r_{k-i+1}\leq C_4\max\left\{2^{-\frac{k-i}{Q_{\sigma B_0}^+}},2^{-\frac{k-i}{Q_{\sigma B_0}^-}}\right\},
	\end{equation*}
	where $C_4=C_4(b,Q^+,Q^-)$. Therefore, using Lemma~\ref{loglemma} $(i)$ with $R:=2 C_4\max\left\{2^{-\frac{k-i}{Q_{\sigma B_0}^+}},2^{-\frac{k-i}{Q_{\sigma B_0}^-}}\right\}$, $x:=y_{k-i}$, $z:=y_k$, and $r:=r_{k-i+1}$, we get
	\begin{equation*}
		2^{\frac{k-i}{\gamma(y_{k-i})}}\leq C_5 2^{\frac{k-i}{\gamma_{B(y_k, r_{k-i+1})}^+}},
	\end{equation*}
	where $C_5=C_5(b,Q^-,Q^+,\gamma^-,\gamma^+,C_{\log}\left(1/\gamma\right))$.
	Since for every $i,$ $\gamma(y_k)\leq \gamma_{B(y_k, r_{k-i+1})}^+\leq \gamma_{\sigma B_0}^+,$ we write
	\begin{equation*}
		2^{\frac{k-i}{\gamma_{B(y_k, r_{k-i+1})}^+}}\leq \max\left\{2^{\frac{k-i}{\gamma(y_k)}}, 2^{\frac{k-i}{\gamma_{\sigma B_0}^+}}\right\}.
	\end{equation*}
	Use these estimates in \eqref{firstestimate1}, to obtain
	\begin{align}\label{allterm1}
		\left| u(y_k)\right| &\leq  C_3C_5\sum_{j=k_0}^{k}\max\{2^{\frac{j}{\gamma(y_k)}}, 2^{\frac{j}{\gamma_{\sigma B_0}^+}}\}+\left| u(y_{k_0})\right|.
	\end{align}

	Note that if $k<0,$ then
	\begin{align*}
		\sum_{j=k_0}^{k}\max\{2^{\frac{j}{\gamma(y_k)}}, 2^{\frac{j}{\gamma_{\sigma B_0}^+}}\}\leq C_6 2^{\frac{k}{\gamma_{\sigma B_0}^+}}, 
	\end{align*}
	where $C_6=C_6(\gamma^-,\gamma^+)$, 	and if $k\geq 0,$ then
	\begin{align*}
		\sum_{j=k_0}^{k} \max\left\{2^{\frac{j}{\gamma(y_k)}}, 2^{\frac{j}{\gamma_{\sigma B_0}^+}}\right\} \leq  
		\sum_{j=-\infty}^{0} 2^{\frac{j}{\gamma_{\sigma B_0}^+}} + \sum_{j=1}^{k} 2^{\frac{j}{\gamma(y_k)}} \leq
		\sum_{j=-\infty}^{0} 2^{\frac{j}{\gamma_{\sigma B_0}^+}} + \sum_{j=-\infty}^{k} 2^{\frac{j}{\gamma(y_k)}} \leq 2C_6 2^{\frac{k}{\gamma(y_k)}}
	\end{align*}
	Therefore from \eqref{allterm1} 
	\begin{equation}\label{allterms1}
		\left| u(y_k)\right|\leq  C_7\max\left\{2^{\frac{k}{\gamma(y_k)}}, 2^{\frac{k}{\gamma_{\sigma B_0}^+}}\right\}+ \left| u(y_{k_0})\right|, 
	\end{equation}
	where $C_7:=2C_3 C_5 C_6.$
	
	To proceed, for each $k\in \mathbb{Z},$ we define
	\begin{equation*}
		a_k:= \left\{ \begin{array}{ll} \displaystyle\sup_{x\in B_0\cap E_k}\left| u(x)\right|,& \textnormal{ if } B_0 \cap E_k\neq \emptyset,\\ 0, & \textnormal{ if } B_0 \cap E_k = \emptyset. \end{array}\right.    
	\end{equation*}
	We will now estimate
	\begin{equation*}
	c_{k_0}:=\sup_{x\in E_{k_0}}\left| u(x)\right|.    
	\end{equation*}	
	We can assume that $\displaystyle\essinf_{E_{k_0}}\vert u\vert=0,$ by the discussion in the beginning of the proof and the fact that $\mu(E_{k_0})>0.$ This means that there exists a sequence $\left\{z_i\right\}_{i=1}^\infty\subseteq E_{k_0}$ such that $u(z_i)\rightarrow 0$ as $i\rightarrow\infty.$ Therefore, for all $x\in E_{k_0},$ we have
	\begin{align}\label{lastterm1}
		\left| u(x)\right|&=\lim_{i\rightarrow\infty}\left| u(x)-u(z_i)\right|\leq \limsup_{i\rightarrow\infty}\left[d(x,z_i)^{s(x)}g(x)+d(x,z_i)^{s(z_i)}g(z_i)\right]\nonumber \\ &\leq\limsup_{i\rightarrow\infty} \left[(2\sigma r_0)^{s(x)}+(2\sigma r_0)^{s(z_i)}\right]\cdot\left[2^{\frac{k_0}{p(x)}}+2^{\frac{k_0}{p(z_i)}}\right].
	\end{align} 
	On the one hand, for $x\in \sigma B_0$, by Lemma~\ref{loglemma} $(i)$ with $R:=2\sigma r_0$, $z:=x_0$, $r:=\sigma r_0$, and $t:=p$, we get
	\begin{equation}\label{nierown1}
		\left(\frac{1}{2\sigma r_0}\right)^{\frac{1}{p(x)}}\leq e^{C_{\log}(1/p)}\left(\frac{1}{2\sigma r_0}\right)^{\frac{1}{p_{\sigma B_0}^+}}.
	\end{equation} 
	Thus, from \eqref{kzero1}, \eqref{nierown1}, the assumption that $\displaystyle \int_{\sigma B_0}g(x)^{p(x)}\mbox{d}\mu(x)\leq 1$, and  the definition of the constants $L(b,\sigma,Q,p,r_0)$ and $U(b,\sigma,Q,p,r_0)$, we have 
	\begin{equation}\label{Celeven1}
		2^{\frac{k_0}{p(x)}}\leq C_8 2^{\frac{k_0}{p_{\sigma B_0}^+}},
	\end{equation} 
	where $C_8=C_8(b,\sigma,C_{\log}(1/p),p^+,p^-,Q^+,Q^-).$
	Moreover, by the virtue of Lemma~\ref{loglemma} $(i)$ with $R:=2\sigma r_0$, $z:=x_0$, $r:=\sigma r_0$, and $t:=1/s$, we get
	\begin{equation}\label{nierown2}
		\left(2\sigma r_0\right)^{s(y)}\leq e^{C_{\log}(s)} \left(2\sigma r_0\right)^{s_{\sigma B_0}^+} 
	\end{equation}
	for any $y\in \sigma B_0$.
	Therefore, from \eqref{lastterm1}, \eqref{Celeven1}, and \eqref{nierown2} we obtain 
	\begin{equation}\label{lastterm3}
		a_{k_0}\leq c_{k_0}\leq  C_9 r_0^{s_{\sigma B_0}^+}2^{\frac{k_0}{p_{\sigma B_0}^+}},
	\end{equation}
	where $C_9=C_9(b,\sigma,C_{\log}(1/p),p^+,p^-,Q^+,Q^-,C_{\log}(s),s^+)$. 
	
	We now establish estimates for $|u(x)|$ and $|u(x)|^{\gamma(x)}$ when $x\in (E_k\setminus E_{k-1})\cap B_0$.
	
	\noindent\textbf{Case I: $c_{k_0}\geq 1.$} 
	
	\textbf{Subcase Ia}: $k\geq 0$, $k> k_0$ and $\mu((E_k\setminus E_{k-1})\cap B_0)>0$.\\
	In this sub-case, from \eqref{allterms1} for every $y_k \in E_k \cap B_0$ we have
	\begin{equation*}
		\left| u(y_k)\right| \leq C_7 2^{\frac{k}{\gamma(y_k)}}+ c_{k_0}.
	\end{equation*}
	Raise both sides of the above inequality to the power $\gamma(y_k)$ and use the fact that $c_{k_0}\geq 1$ to obtain
	\begin{align*}
		\left| u(y_k)\right|^{\gamma(y_k)} \leq (2C_7)^{\gamma(y_k)}2^k+ 2^{\gamma(y_k)}c_{k_0}^{\gamma(y_k)}\leq C_{10}\left(2^k+c_{k_0}^{\gamma_{\sigma B_0}^+}\right), 
	\end{align*}
	where $C_{10}=C_{10}(\gamma^+,C_7)$. 
	We define
	\begin{equation*}
		b_k:=\sup_{x\in B_0\cap E_k}\left| u(x)\right|^{\gamma(x)}.    
	\end{equation*}
	After taking supremum over $y_k\in E_{k}\cap B_0,$ the above inequality becomes
	\begin{equation*}
		b_{k}\leq C_{10}\left(2^k+ c_{k_0}^{\gamma_{\sigma B_0}^+}\right),
	\end{equation*} 
	and using \eqref{lastterm3} it takes the form
	\begin{equation}\label{positive1}
		b_{k}\leq C_{10} 2^k+ C_{10}C_9^{\gamma_{\sigma B_0}^+} r_0^{s_{\sigma B_0}^+\gamma_{\sigma B_0}^+}2^{\frac{k_0\gamma_{\sigma B_0}^+}{p_{\sigma B_0}^+}} \leq C_{11}\left(2^k+r_0^{s_{\sigma B_0}^+ \gamma_{\sigma B_0}^+} 2^{\frac{k_0 \gamma_{\sigma B_0}^+}{p_{\sigma B_0}^+}}\right), 
	\end{equation}
	where $C_{11}=C_{11}(C_9,C_{10},\gamma^+).$
	
	\textbf{Subcase Ib}: $k_0< k<0$ and $\mu((E_k\setminus E_{k-1})\cap B_0)>0$.\\
	In this subcase from \eqref{allterms1} for every $y_k \in E_k \cap B_0$ we get
	\begin{equation}
		\left| u(y_k)\right| \leq C_7 2^{\frac{k}{\gamma_{\sigma B_0}^+}}+ c_{k_0}
	\end{equation}
	which gives, upon taking supremum over $y_k\in E_k\cap B_0$ and using \eqref{lastterm3} 
	\begin{equation*}
		a_k\leq C_7 2^{\frac{k}{\gamma_{\sigma B_0}^+}}+ c_{k_0}\leq C_7 2^{\frac{k}{\gamma_{\sigma B_0}^+}}+C_9 r_0^{s_{\sigma B_0}^+}2^{\frac{k_0}{p_{\sigma B_0}^+}}.
	\end{equation*}
	Since $c_{k_0}\geq 1,$ we obtain, for any $x\in B_0,$
	\begin{align}\label{negative1}
		a_k^{\gamma(x)}&\leq\left(C_7 2^{\frac{k}{\gamma_{\sigma B_0}^+}}+ c_{k_0}\right)^{\gamma(x)}
		\leq\left(C_7 2^{\frac{k}{\gamma_{\sigma B_0}^+}}+ c_{k_0}\right)^{\gamma_{\sigma B_0}^+}\nonumber\\
		&\leq C_{12}\left(2^k+r_0^{s_{\sigma B_0}^+\gamma_{\sigma B_0}^+}2^{\frac{k_0\gamma_{\sigma B_0}^+}{p_{\sigma B_0}^+}}\right),
	\end{align}
	where $C_{12}=C_{12}(C_7,C_9,\gamma^+).$
	
	\textbf{Subcase Ic}: $k\leq k_0$ and $\mu\left(\left(E_k \setminus E_{k-1}\right) \cap B_0 \right)>0.$\\
	In this sub-case,  we use \eqref{lastterm3} and the facts that $a_k\leq a_{k_0}\leq c_{k_0}$ and $c_{k_0}\geq 1$, to conclude that, for any $x\in B_0,$ 
	\begin{align}\label{row1}
		a_k^{\gamma(x)}\leq c_{k_0}^{\gamma(x)}\leq c_{k_0}^{\gamma_{\sigma B_0}^+}\leq C_9^{\gamma_{\sigma B_0}^+}r_0^{s_{\sigma B_0}^+\gamma_{\sigma B_0}^+}2^{\frac{k_0\gamma_{\sigma B_0}^+}{p_{\sigma B_0}^+}}\leq C_{13} r_0^{s_{\sigma B_0}^+\gamma_{\sigma B_0}^+}2^{\frac{k_0\gamma_{\sigma B_0}^+}{p_{\sigma B_0}^+}}, 
	\end{align}
	where  $C_{13}=C_{13}(C_9,\gamma^+).$
	
	Now, using \eqref{g1}, \eqref{g1-X}, \eqref{kzero1},  \eqref{positive1}, \eqref{negative1}, \eqref{row1}, $\displaystyle \int_{\sigma B_0}g(x)^{p(x)}\mbox{d}\mu(x)\leq 1$, and the lower  Ahlfors-regularity of the measure, we finally obtain
	\begin{align*}
		\int_{B_0}&\left| u(x)\right| ^{\gamma(x)} \mbox{d}\mu(x) = \sum_{k\in\mathbb{Z}}\int_{B_0\cap (E_k\setminus E_{k-1})}\left| u(x)\right|^{\gamma(x)} \mbox{d}\mu(x)\\
		&= \left[\sum_{\substack{k>k_0, k\geq 0,{}\\  \mu(B_0\cap (E_k\setminus E_{k-1}))>0}}+\sum_{\substack{k_0< k<0,{}\\  \mu(B_0\cap (E_k\setminus E_{k-1}))>0}} +\sum_{\substack{k\leq k_0,{} \\ \mu(B_0 \cap (E_k \setminus E_{k-1}))>0}}\right]\int_{B_0\cap (E_k\setminus E_{k-1})}\left| u(x)\right| ^{\gamma(x)} \mbox{d}\mu(x)\\
		&\leq \sum_{\substack{k>k_0, k\geq 0,{}\\  \mu(B_0\cap (E_k\setminus E_{k-1}))>0}}b_k\mu(B_0\cap(E_k\setminus E_{k-1}))+\sum_{\substack{k_0< k<0  \textnormal{ or } k \, \leq k_0,{}\\  \mu(B_0\cap (E_k\setminus E_{k-1}))>0}}\int_{B_0\cap (E_k\setminus E_{k-1})}a_k^{\gamma(x)} \mbox{d}\mu(x)\\
		&\leq  \max\left\{C_{11}, C_{12}\right\}\sum_{k \in \mathbb{Z}}2^k\mu(E_k\setminus E_{k-1})+\max\left\{C_{11}, C_{12},C_{13}\right\} r_0^{s_{\sigma B_0}^+\gamma_{\sigma B_0}^+}2^{\frac{k_0\gamma_{\sigma B_0}^+}{p_{\sigma B_0}^+}}\mu(B_0)\\
		&\leq 2\max\left\{C_{11}, C_{12}\right\}\frac{\mu(B_0)}{br_0^{Q(x_0)}}+\max\left\{C_{11}, C_{12},C_{13}\right\}r_0^{s_{\sigma B_0}^+\gamma_{\sigma B_0}^+}\left(\frac{U(b,\sigma,Q,p,r_0)}{br_0^{Q_{\sigma B_0}^+}}\right)^{\frac{\gamma_{\sigma B_0}^+}{p_{\sigma B_0}^+}}\mu(B_0).
	\end{align*}
	Since $r_0 \leq 1$ we have
	\begin{equation}\label{koniec1}
		r_0^{s_{\sigma B_0}^+ \gamma_{\sigma B_0}^+}  \leq r_0^{s(x_0)\gamma(x_0)} \hspace{5mm} \textnormal{and} \hspace{5mm}
		\left(\frac{1}{r_0}\right)^{\frac{1}{p_{\sigma B_0}^+}} \leq \left(\frac{1}{r_0} \right)^{\frac{1}{p(x_0)}}.
	\end{equation}
	Moreover, by Lemma~\ref{loglemma} $(i)$ applied for $R:=2\sigma r_0$, $r:=\sigma r_0$, $x:=x_0$ and respectively $t:=1/Q$ and $t:=1/\gamma$, we have
	\begin{align}
	\left(\frac{1}{r_0}\right)^{Q_{\sigma B_0}^+}& \leq e^{C_{\log}(Q)} \left(\frac{1}{2\sigma} \right)^{Q(x_0)-Q_{\sigma B_0}^+} \left(\frac{1}{r_0}\right)^{Q(x_0)},\label{koniec2}\\
		\left(\frac{1}{r_0}\right)^{\gamma_{\sigma B_0}^+}& \leq e^{C_{\log}(\gamma)} \left(\frac{1}{2\sigma} \right)^{\gamma(x_0)-\gamma_{\sigma B_0}^+} \left(\frac{1}{r_0}\right)^{\gamma(x_0)}\label{koniec3}
	\end{align}
	Then, by \eqref{koniec1}, \eqref{koniec2}, \eqref{koniec3}, the choice $\displaystyle \gamma:=\frac{Qp}{Q-sp}$, and the definition of $U(b,\sigma,Q,p,r_0)$, we get
	\begin{equation*}
		\int_{B_0} \left|u(x)\right|^{\gamma(x)}\mbox{d}\mu(x)\leq C_{14}\frac{\mu(B_0)}{r_0^{Q(x_0)}}, 
	\end{equation*}
	where $C_{14}=C_{14}(C_{11},C_{12},C_{13},b,\sigma,Q^-,Q^+,C_{\log}(Q),\gamma^-,\gamma^+,p^+,p^-,C_{\log}(\gamma)).$
	
	\textbf{Case II: $c_{k_0}<1.$} 
	
	By arguing as in Case I, for $k\geq 0$, $k> k_0$, and  $\mu(B_0\cap (E_k\setminus E_{k-1}))>0$, we have 
	\begin{equation*}
		b_k\leq C_{10}\left(2^k+1\right).
	\end{equation*}
	Moreover, when $k_0< k<0$ and  $\mu(B_0\cap (E_k\setminus E_{k-1}))>0$ or when $k\leq k_0$ and  $\mu(B_0\cap (E_k\setminus E_{k-1}))>0$, we have
	\begin{equation*}
		a_k\leq C_7+1.
	\end{equation*}
	Hence, proceeding as in Case I (keeping in mind $r_0\leq 1$), we obtain
	\begin{align*}
		\int_{B_0}&\left| u(x)\right| ^{\gamma(x)}\mbox{d}\mu(x) = \sum_{k\in\mathbb{Z}}\int_{B_0\cap (E_k\setminus E_{k-1})}\left| u(x)\right|^{\gamma(x)} \mbox{d}\mu(x)\\
		&= \left[\sum_{\substack{k>k_0, k\geq 0,{}\\  \mu(B_0\cap (E_k\setminus E_{k-1}))>0}}+\sum_{\substack{k_0< k<0,{}\\  \mu(B_0\cap (E_k\setminus E_{k-1}))>0}} +\sum_{\substack{k\leq k_0,{} \\ \mu(B_0 \cap (E_k \setminus E_{k-1}))>0}}\right]\int_{B_0\cap (E_k\setminus E_{k-1})}\left| u(x)\right| ^{\gamma(x)} \mbox{d}\mu(x)\\ & \leq \sum_{\substack{k>k_0, k\geq 0,{}\\  \mu(B_0\cap (E_k\setminus E_{k-1}))>0}} b_k\mu(B_0\cap(E_k\setminus E_{k-1}))+\sum_{\substack{k_0< k<0 \textnormal{ or } k \leq k_0,{}\\  \mu(B_0\cap (E_k\setminus E_{k-1}))>0}}\int_{B_0\cap (E_k\setminus E_{k-1})}a_k^{\gamma(x)} \mbox{d}\mu(x)\\
		&\leq  2C_{10}+\left[C_{10}+\left(C_7+1\right)^{\gamma^+}\right]\mu(B_0) \\ & \leq C_{15}\frac{\mu(B_0)}{r_0^{Q(x_0)}},
	\end{align*}
	where
	\begin{equation*}
		C_{15}:=2\frac{C_{10}}{b}+\left[C_{10}+\left(C_7+1\right)^{\gamma^+}\right].
	\end{equation*}
	This proves part  $(i)$ in the case when $\mu(B_0\setminus E_{k_0})>0.$
	
	Suppose now that $\mu(B_0\setminus E_{k_0})=0.$ Then we have 
	\begin{equation}\label{xxi-48}
		\int_{B_0}\left|u(x)\right|^{\gamma(x)} \mbox{d}\mu(x)=\int_{E_{k_0}\cap B_0}\left|u(x)\right|^{\gamma(x)} \mbox{d}\mu(x) \leq \int_{E_{k_0}\cap B_0}a_{k_0}^{\gamma(x)} \mbox{d}\mu(x).
	\end{equation}
	Therefore, if $a_{k_0} >1$, then by \eqref{xxi-48} and \eqref{lastterm3}, we have 
	\begin{align*}
		\int_{B_0}\left|u(x)\right|^{\gamma(x)} \mbox{d}\mu(x) &\leq a_{k_0}^{\gamma_{\sigma B_0}^+} \mu(B_0) \leq C_9^{\gamma_{\sigma B_0}^+} r_0^{s_{\sigma B_0}^+ \gamma_{\sigma B_0}^+} 2^{\frac{k_0 \gamma_{\sigma B_0}^+}{p_{\sigma B_0}^+}}\mu(B_0) \\ & \leq C_9^{\gamma_{\sigma B_0}^+} \left(\frac{U(b,\sigma,Q,p,r_0)}{b}\right)^{\frac{\gamma_{\sigma B_0}^+}{p_{\sigma B_0}^+}}r_0^{s_{\sigma B_0}^+ \gamma_{\sigma B_0}^+} r_0^{-\frac{Q_{\sigma B_0}^+ \gamma_{\sigma B_0}^+}{p_{\sigma B_0}^+}} \mu(B_0).
		\end{align*}
		Hence, using \eqref{xxi-48}, \eqref{koniec1}, \eqref{koniec2}, \eqref{koniec3}, and the choice $\displaystyle \gamma:=\frac{Qp}{Q-sp}$,  we obtain
		\begin{equation*}
			\int_{B_0} \left|u(x)\right|^{\gamma(x)} \mbox{d}\mu(x) \leq C_{16} \frac{\mu(B_0)}{r_0^{Q(x_0)}},
		\end{equation*}
		where $C_{16}=C_{16}(C_9,b,\sigma,Q^-,Q^+,C_{\log}(Q),\gamma^-,\gamma^+,C_{\log}(\gamma),p^+,p^-).$
		
	If $a_{k_0} \leq 1$, then  since $r_0\leq1$, we have
	\begin{equation*}
		\int_{B_0}|u(x)|^{\gamma(x)}\mbox{d}\mu(x) \leq \mu(B_0)\leq \frac{\mu(B_0)}{r_0^{Q(x_0)}}.
	\end{equation*}
	Therefore,
	\begin{equation*}
		\int_{B_0}\left|u(x)\right|^{\gamma(x)} \mbox{d}\mu(x) \leq C_{17}\frac{\mu(B_0)}{r_0^{Q(x_0)}},
	\end{equation*}
	where $C_{17}:= \max\left\{1, C_{16}\right\}$.
	Finally, we proved \eqref{redukcja1} with $C:=\max\left\{C_{14}, C_{15},C_{17}\right\}$, and we get \eqref{one1} with
	\begin{equation*}
		C_{S}:=\beta_p\left(\frac{1}{b^{\frac{1}{\gamma^-}}}+C^{\frac{1}{\gamma^-}}\right).
	\end{equation*}
	
	Before we start proving $(ii)$ and $(iii)$, let us notice that inequality $sp\geq Q$ implies that quantity $u_{B_0}$ is well defined. Indeed, for $q:=Q/(Q+s)$, we have $\displaystyle Q-sq=\frac{Q^2}{Q+s}\geq \frac{\left(Q^-\right)^2}{Q^+ + s^+}>0$ and 
	\begin{equation*}
	\frac{Q\cdot \frac{Q}{Q+s}}{Q-s\frac{Q}{Q+s}}=\frac{Q^2}{Q^2}=1.
	\end{equation*}
	Moreover, $q\ll p$. Therefore, by Lemma~\ref{wlozenielp}, and the recently proved statement $(i)$ (with $q$ in place of $p$), we have
	\begin{equation*}
	\dot{M}^{s(\cdot),p(\cdot)}(\sigma B_0) \subseteq \dot{M}^{s(\cdot),q(\cdot)}(\sigma B_0)
	\subseteq L^{\frac{Q(\cdot)q(\cdot)}{Q(\cdot)-s(\cdot)q(\cdot)}}( B_0)=L^{1}( B_0).
	\end{equation*}
	Hence, every function $u\in \dot{M}^{s(\cdot),p(\cdot)}(\sigma B_0)$ is integrable over $B_0$ and therefore, $u_{B_0}$ is well defined.
	
	Now, we pass to the proof of $(ii)$. For any $a>0$ we have
	\begin{align*}
		\fint_{B_0} \exp\left(a\left|u(x)-u_{B_0}\right|\right) \mbox{d} \mu(x) &\leq \fint_{B_0} \exp\left(\fint_{B_0 }a\left|u(x)-u(y)\right|\mbox{d}\mu(y)\right)\mbox{d}\mu(x) \\ & \leq \fint_{B_0} \fint_{B_0} \exp\left(a\left|u(x)-u(y)\right|\right)\mbox{d}\mu(y)\mbox{d}\mu(x) \\ &\leq \left(\fint_{B_0} \exp\left(a\left|u(x)\right|\right) \mbox{d}\mu(x)\right)^2.
	\end{align*}
	where we applied Jensen inequality. Therefore it suffices to prove that
	\begin{equation}\label{pomocmoser}
		\fint_{B_0} \exp\left(\beta_pC_{MT1}\left|u(x)\right|\right) \mbox{d} \mu(x) \leq \sqrt{C_{MT2}}.
	\end{equation}
	Let $C_{MT1}\in(0,\infty)$ be such that $\exp\left(\beta_pC_{MT1}C_3\right)=2.$ We have
	\begin{align*}
		\fint_{B_0} \exp\left(\beta_pC_{MT1}\left|u(x)\right|\right) \mbox{d} \mu(x) \leq I_1+I_2,
	\end{align*}
	where
	\begin{align*}
	I_1&:=\frac{1}{\mu(B_0)}\int_{B_0 \cap E_{k_0}} \exp\left(\beta_p C_{MT1}\left|u(x)\right|\right) \mbox{d} \mu(x),\\
	I_2&:=\frac{1}{\mu(B_0)}\int_{B_0 \setminus E_{k_0}} \exp\left(\beta_p C_{MT1}\left|u(x)\right|\right) \mbox{d} \mu(x).
	\end{align*}
	Using \eqref{lastterm3}, \eqref{kzero1}, the definition of $U(b,\sigma,Q,p,r_0)$, and the facts that $\displaystyle \int_{\sigma B_0}g(x)^{p(x)}\mbox{d}\mu(x)\leq 1$ and $r_0 \leq 1$, we obtain
	\begin{equation*}
		c_{k_0}\leq C_9 r_0^{s_{\sigma B_0}^+}2^{\frac{k_0}{p_{\sigma B_0}^+}}\leq C_{18}r_0^{s_{\sigma B_0}^+}\left(\frac{1}{r_0}\right)^{\frac{Q_{\sigma B_0}^+}{p_{\sigma B_0}^+}} = C_{18} r_0^{s_{\sigma B_0}^+} \left(\frac{1}{r_0}\right)^{\frac{\left(sp\right)_{\sigma B_0}^+}{p_{\sigma B_0}^+}}\leq C_{18}r_0^{s_{\sigma B_0}^+}\left(\frac{1}{r_0}\right)^{s_{\sigma B_0}^+}=C_{18},
	\end{equation*}
	where
	\begin{equation*}
		C_{18}=C_{18}(C_9,b,\sigma,p^-,p^+,Q^+,Q^-).
	\end{equation*}
	Therefore
		\begin{equation*}
			I_1\leq \frac{\mu\left(B_0 \cap E_{k_0}\right)}{\mu(B_0)}\exp\left(\beta_pC_{MT1}C_{18}\right)\leq \exp\left(\beta_pC_{MT1}C_{18}\right).
		\end{equation*}
		
To estimate $I_2$, suppose that $\mu(B_0 \setminus E_{k_0})>0.$ The case $\mu(B_0 \setminus E_{k_0})=0$ will be considered later. From \eqref{firstestimate1}, if $k\in\mathbb Z$ satisfies $k\geq k_0+1$ and $\mu((E_k\setminus E_{k-1})\cap B_0)>0$, then we have
	\begin{equation}\label{szacowanie1}
		a_k \leq C_3(k-k_0+1)+c_{k_0}\leq C_3(k-k_0)+C_3+C_{18}.
	\end{equation}
	Using this, the definition of $C_{MT1}$, the fact that $\displaystyle \left(\frac{1}{\beta_p}\right)^{p_{\sigma B_0}^+} \leq \int_{\sigma B_0} g(x)^{p(x)}\mbox{d}\mu(x) \leq 1$, \eqref{g1}, the definition of the constant $L(b,\sigma,Q,p,r_0)$, and \eqref{kzero1}, we obtain
	\begin{align*}
		I_2& \leq \frac{\exp\left(\beta_pC_{MT1}\left(C_{18}+C_3\right)\right)}{\mu(B_0)} \sum_{k=k_0+1}^{\infty} \exp\left(\beta_p C_{MT1}C_3(k-k_0)\right)\mu\left(B_0 \cap \left(E_k \setminus E_{k-1}\right) \right) \\ &\leq \exp\left(\beta_p C_{MT1}\left(C_{18}+C_3\right)\right)\frac{2^{-k_0}}{\mu(B_0)}\sum_{k=-\infty}^{\infty} 2^k \mu\left(E_k \setminus E_{k-1}\right) \\ &\leq \exp\left(\beta_p C_{MT1}\left(C_{18}+C_3\right)\right)\cdot 2\beta_p^{p^+}\frac{br_0^{Q_{\sigma B_0}^+}}{\mu(B_0)} \\ &\leq 2\beta_p^{p^+}\exp\left(\beta_pC_{MT1}\left(C_{18}+C_3\right)\right),
	\end{align*}
	where in the last inequality we used the lower Ahlfors-regularity of the measure.
	
	Suppose that $\mu(B_0 \setminus E_{k_0})=0.$ Then,
	\begin{align*}
		\fint_{B_0}\exp\left(\beta_p C_{MT1} \left|u(x)\right|\right) \mbox{d}\mu(x) & = \fint_{B_0 \cap E_{k_0}} \exp\left( \beta_p C_{MT1} \left|u(x)\right|\right) \mbox{d}\mu(x) \\ & \leq \fint_{E_{k_0}} \exp\left(\beta_p C_{MT1} c_{k_0} \right) \mbox{d}\mu(x) \\ & \leq \exp\left(\beta_p C_{MT1}C_{18}\right) \leq C_{19},
	\end{align*}
	where
	\begin{equation*}
		C_{19}:=\left(1+2\beta_p^{p^+}\right) \exp\left(\beta_p C_{MT1} \left(C_{18}+C_3\right)\right).
	\end{equation*}
	Therefore, \eqref{pomocmoser} was proven with
	\begin{equation*}
	C_{MT2}:= C_{19}^2,
	\end{equation*}
	and the proof of $(ii)$ is finished.
	
	Now we shall prove $(iii).$ Suppose that $\mu(B_0 \setminus E_{k_0})>0$, and let $k\in\mathbb Z$ be such that $k \geq k_0 +1$ and $\mu(B_0 \cap (E_k \setminus E_{k-1}))>0$. Then, inequality \eqref{firstestimate1}, the choice $\displaystyle \alpha:=s-\frac{Q}{p}$, and the fact that $sp \gg Q$, yield 
	\begin{align*}
		|u(y_k)| &\leq C_3 \sum_{i=0}^{k-k_0}\max\left\{2^{-(k-i)\alpha_{\sigma B_0}^+/Q_{\sigma B_0}^-},2^{-(k-i)\alpha_{\sigma B_0}^- / Q_{\sigma B_0}^+} \right\}+|u(y_{k_0})| \nonumber\\
		& \leq  C_3 \sum_{j=k_0}^{k}\max\left\{2^{-j\alpha_{\sigma B_0}^+/Q_{\sigma B_0}^-},2^{-j\alpha_{\sigma B_0}^-/Q_{\sigma B_0}^+}\right\}+|u(y_{k_0})|.
	\end{align*}
	If $k_0 +1 \leq 0$, then 
	\begin{align*}
		\sum_{j=k_0}^{k}\max\left\{2^{-j\alpha_{\sigma B_0}^+/Q_{\sigma B_0}^-},2^{-j\alpha_{\sigma B_0}^-/Q_{\sigma B_0}^+}\right\} &\leq \sum_{j=k_0}^{\infty}2^{-j\alpha_{\sigma B_0}^+/Q_{\sigma B_0}^-} + \sum_{j=1}^{\infty}2^{-j\alpha_{\sigma B_0}^-/Q_{\sigma B_0}^+}\\
		&\leq \frac{2^{-k_0\alpha_{\sigma B_0}^+/Q_{\sigma B_0}^-}}{1- 2^{-\alpha_{\sigma B_0}^+/Q_{\sigma B_0}^-}} + \frac{2^{-\alpha_{\sigma B_0}^-/Q_{\sigma B_0}^+}}{1- 2^{-\alpha_{\sigma B_0}^-/Q_{\sigma B_0}^+}}2^{-k_0\alpha_{\sigma B_0}^-/Q_{\sigma B_0}^+}.
	\end{align*}
	If $k_0 +1 > 0$, then 
	\begin{align*}
		\sum_{j=k_0}^{k}\max\left\{2^{-j\alpha_{\sigma B_0}^+/Q_{\sigma B_0}^-},2^{-j\alpha_{\sigma B_0}^-/Q_{\sigma B_0}^+}\right\} &\leq \sum_{j=k_0}^{\infty}2^{-j\alpha_{\sigma B_0}^-/Q_{\sigma B_0}^+}= \frac{2^{-k_0\alpha_{\sigma B_0}^-/Q_{\sigma B_0}^+}}{1- 2^{-\alpha_{\sigma B_0}^-/Q_{\sigma B_0}^+}}.
	\end{align*}
	Hence,
	\begin{align}\label{firstestimate for p>Q 1}
		|u(y_k)| &\leq C_{20}\max\left\{2^{-k_0\alpha_{\sigma B_0}^+/Q_{\sigma B_0}^-},2^{-k_0\alpha_{\sigma B_0}^-/Q_{\sigma B_0}^+}\right\}  +|u(y_{k_0})|,
	\end{align}
	where 
	\begin{equation*}
		C_{20}=C_{20}(C_3,\alpha^-,Q^+).
	\end{equation*}
	Upon taking supremum over $y_k\in{E_k\cap B_0}$, \eqref{firstestimate for p>Q 1} and \eqref{lastterm3} give,
	\begin{equation}\label{allterms for p>Q 1}
		a_k\leq C_{20}\max\left\{2^{-k_0\alpha_{\sigma B_0}^+/Q_{\sigma B_0}^-},2^{-k_0\alpha_{\sigma B_0}^-/Q_{\sigma B_0}^+}\right\} +c_{k_0},
	\end{equation}
	where, by \eqref{lastterm3},
	\begin{equation*}
 c_{k_0}  \leq C_9 r_0^{s_{\sigma B_0}^+} 2^{\frac{k_0}{p_{\sigma B_0}^+}}.
	\end{equation*}
	Since $\displaystyle \left(\frac{1}{\beta_p}\right)^{p_{\sigma B_0}^+} \leq \int_{\sigma B_0}g(x)^{p(x)}\mbox{d}\mu(x)\leq 1$, from \eqref{kzero1} we get
	\begin{equation*}
		\frac{L(b,\sigma,Q,p,r_0)}{\beta_p^{p_{\sigma B_0}^+}br_0^{Q_{\sigma B_0}^+}}\leq 2^{k_0}\leq \frac{U(b,\sigma, Q, p,r_0)}{b{r_0}^{Q_{\sigma B_0}^+}}.
	\end{equation*}
	Using this in \eqref{allterms for p>Q 1}, we obtain, for all $k\geq k_0+1$ satisfying $\mu(B_0 \cap (E_k \setminus E_{k-1}))>0$,
	\begin{align*}
		a_k &\leq C_{21}\max\left\{r_0^{\alpha_{\sigma B_0}^+},r_0^{\alpha_{\sigma B_0}^-}\right\}+ C_{22}r_0^{s_{\sigma B_0}^+-\frac{Q_{\sigma B_0}^+}{p_{\sigma B_0}^+}}\nonumber\\
		&\leq C_{21} \max\left\{r_0^{\alpha(x_0)},r_0^{\alpha_{\sigma B_0}^-}\right\}+C_{22} r_0^{s(x_0)-\frac{Q_{\sigma B_0}^+}{p(x_0)}},
	\end{align*}
	where we used the fact that $r_0 \leq 1$, and where
	\begin{equation*}
		C_{21}=C_{21}(C_{20},p^+,p^-,\alpha^+,Q^-)\quad\text{and}\quad C_{22}=C_{22}(C_9,b,p^-,p^+,Q^-,Q^+,\sigma).
	\end{equation*}
	Moreover, by Lemma~\ref{loglemma} $(i)$ for $r:=\sigma r_0$, $R:=2\sigma r_0$, $z:=x_0$,  $x:=x_0$, and respectively $t:=1/\alpha$ and $t:=1/Q$, we get
	\begin{align*}
		r_0^{\alpha_{\sigma B_0}^-}&\leq e^{C_{\log}(\alpha)}(2\sigma)^{\alpha^+} r_0^{\alpha(x_0)}, \\
		\left(\frac{1}{r_0}\right)^{Q_{\sigma B_0}^+}& \leq e^{C_{\log}(Q)} \left(2\sigma\right)^{Q^+} \left(\frac{1}{r_0}\right)^{Q(x_0)}.
	\end{align*} 
	Therefore
	\begin{equation}\label{gwiazdka}
		a_k \leq C_{23} r_0^{\alpha(x_0)},
	\end{equation}
	where $C_{23}:=C_{23}(C_{21},C_{22},C_{\log}(\alpha),C_{\log}(Q),\sigma,\alpha^+,Q^+,p^-)$.

	If $k \leq k_0$, then by using \eqref{lastterm3} and \eqref{kzero1}, together with an argument similar to the one used in the proof of \eqref{gwiazdka}, we easily see that
	\begin{equation*}
		a_k \leq a_{k_0}\leq C_9 r_0^{s_{\sigma B_0}^+} 2^{\frac{k_0}{p_{\sigma B_0}^+}}\leq C_{23} r_0^{\alpha(x_0)}.
	\end{equation*}
	Therefore, 
	\begin{equation}\label{supremum norm 1}
		\left\|u-u_{B_0}\right\|_{L^{\infty}(B_0)}\leq 2\left\|u\right\|_{L^{\infty}(B_0)}\leq 2C_{23}r_0^{\alpha(x_0)}.
	\end{equation}
	Now, in the case when $\mu(B_0\setminus E_{k_0})=0,$ we have
	\begin{equation*}
		\left\|u-u_{B_0}\right\|_{L^{\infty}(B_0)}\leq 2\left\|u\right\|_{L^{\infty}(B_0)}=2\left\|u\right\|_{L^{\infty}(B_0\cap E_{k_0})} \leq 2a_{k_0} \leq 2C_{23} r_0^{\alpha(x_0)}.
	\end{equation*}
	Therefore we get
	\begin{equation}\label{hold}
		\left\|u-u_{B_0}\right\|_{L^{\infty}(B_0)}\leq C_H  r_0^{\alpha(x_0)}\left\|g\right\|_{L^{p(\cdot)}(\sigma B_0)},
	\end{equation}
	where $C_H:= 2C_{23}\beta_p$ and \eqref{morrey} is proven.
	
	Now, we shall prove the last part of $(iii)$. Let $x \in B_0$ and $\tilde{B}:=B(x,\tilde{r})$ with $\tilde{r} \leq r_0$ such that $\sigma \tilde{B} \subseteq \sigma B_0$. Then by \eqref{morrey}, applied with $\tilde{B}$ in place of $B_0$, we have
	\begin{equation}\label{holderpomoc}
		\left\|u-u_{\tilde{B}}\right\|_{L^{\infty}(\tilde{B})}\leq C_H \tilde{r}^{\alpha(x)} \left\|g\right\|_{L^{p(\cdot)}(\sigma \tilde{B})}\leq C_H  \tilde{r}^{\alpha(x)} \left\|g\right\|_{L^{p(\cdot)}(\sigma B_0)}.
	\end{equation}			
	By Lemma~\ref{zero} there exists a null set $N\subseteq B_0$ such that, for all $x, y \in B_0 \setminus N$,
	\begin{equation}\label{propozycja}
		|u(x)-u(y)|\leq 2\left\|\hat{u}-\hat{u}_{B(x,2d(x,y))}\right\|_{L^{\infty}(B(x,2d(x,y)))},
	\end{equation}
	where $\hat{u}=u$ on $\sigma B_0$ and is equal to zero outside of $\sigma B_0$. Moreover, we can assume, for all $x, y \in B_0 \setminus N$,
	\begin{equation}\label{holderloc}
		|u(x)-u(y)|\leq 2\left\|u-u_{B_0}\right\|_{L^{\infty}(B_0)}.
	\end{equation}
	Let $x,y\in B_0\setminus N$. Suppose first that $d(x,y)\leq \hat{r}_0 := \frac{1}{2}\left(1-\frac{1}{\sigma}\right)r_0$ and define $\tilde{B}:= B(x,2 d(x,y))$. Therefore, since $\sigma \tilde{B}\subseteq \sigma B_0$ and $\hat{r}_0\leq r_0$, by \eqref{propozycja} and \eqref{holderpomoc}, we have 
	\begin{align}\label{holder continuity 2}
		|u(x)-u(y)|&\leq  2C_H  \left[2d(x,y)\right]^{\alpha(x)}\left\|g\right\|_{L^{p(\cdot)}(\sigma B_0)} \nonumber \\ & \leq  2^{\alpha^++1}C_H  d(x,y)^{\alpha(x)}\left\|g\right\|_{L^{p(\cdot)}(\sigma B_0)}.
	\end{align}
	On the other hand, if $d(x,y)> \hat{r}_0$, then by \eqref{holderloc} and \eqref{morrey},  we get
	\begin{align}\label{r1 1}
		|u(x)-u(y)| &\leq 2\|u-u_{B_0}\|_{L^{\infty}(B_0)}\nonumber \\
		&\leq 2C_H   \frac{r_0^{\alpha(x_0)}}{\hat{r_0}^{\alpha(x)}}d(x,y)^{\alpha(x)} \left\|g\right\|_{L^{p(\cdot)}(\sigma B_0)} \nonumber \\ & \leq 2C_H   \frac{1}{\hat{r_0}^{\alpha^+}}d(x,y)^{\alpha(x)} \left\|g\right\|_{L^{p(\cdot)}(\sigma B_0)}
	\end{align}
	Therefore, for every $x,y\in B_0\setminus N$, we have
	\begin{align}\label{fin 1}
		|u(x)-u(y)| \leq& D_H(r_0)d(x,y)^{\alpha(x)} \left\|g\right\|_{L^{p(\cdot)}(\sigma B_0)},
	\end{align}
	where
	\begin{equation*}
		D_H(r_0):=\max \left\{2^{\alpha^++1}C_H, 2C_H  \left(\frac{2\sigma}{(\sigma -1)r_0}\right)^{\alpha^+} \right\}=2^{\alpha^++1}C_H \left(\frac{\sigma}{(\sigma -1)r_0}\right)^{\alpha^+}.
	\end{equation*}
	Since $B_0 \setminus N$ is dense in $B_0$, for every $x\in B_0$, there exists a sequence $\{x_k\}_{k=1}^\infty\subseteq B_0 \setminus N$ such that $d(x_k,x)\to0$ as $k\to\infty$. Moreover, since $\alpha^->0$ from \eqref{fin 1}, we get that $\{u(x_k)\}_{k=1}^\infty$ is a Cauchy sequence in $\mathbb{R}$. For each  $x\in B_0$, let
	\begin{equation*}
	\tilde{u}(x):=\lim_{k \rightarrow \infty}u(x_k).
	\end{equation*}
	Then, from \eqref{fin 1} we get that $\tilde{u}(x)$ is well defined and, 
	for every $x, y \in B_0$,
	\begin{equation*}\label{finfin 1}
		|\tilde{u}(x)-\tilde{u}(y)| \leq D_H(r_0)d(x,y)^{\alpha(x)}\left\|g\right\|_{L^{p(\cdot)}(\sigma B_0)}.
	\end{equation*}
	Therefore, we have
	\begin{equation*}
		\sup_{\substack{x, y\in B_0\\ x \neq y}} \frac{|\tilde{u}(x)-\tilde{u}(y)|}{d(x,y)^{\alpha(x)}} \leq D_H(r_0) \left\|g\right\|_{L^{p(\cdot)}(\sigma B_0)}.
	\end{equation*}
	This completes the proof of $(iii)$ and hence, the proof of Theorem~\ref{sobolevpoincare}.
\end{proof}

As a corollary, we can obtain following local Sobolev embedding.

\begin{prop}\label{localemb}
	Let $(X,d,\mu)$ be a metric measure space and assume that there exist $Q\in \mathcal{P}^{\log}_b(X)$, $\delta\in (0,\infty)$, and $b\in (0,1]$ such that for every $x\in X$ and $r \in (0,\delta]$, 
	\begin{equation*}
		\mu\left(B(x,r)\right)\geq br^{Q(x)}.
	\end{equation*}
	Fix $\sigma \in (1,\infty)$ and let $p,s\in \mathcal{P}^{\log}_b(X)$ satisfy $sp \ll Q$. Then, for every ball $B_0=B(x_0,r_0)\subseteq X$ with $r_0\leq \delta/\sigma$, and every pair of functions $u\in M^{s(\cdot),p(\cdot)}(\sigma B_0)$ and $g\in \mathcal{D}^{s(\cdot)}(u)\cap L^{p(\cdot)}(\sigma B_0)$, there holds 
	\begin{equation}\label{localembedding}
		\left\|u\right\|_{L^{\gamma(\cdot)}(B_0)} \leq \left(\kappa_{\gamma(\cdot)}^2+\Lambda(B_0)\right)C_S \left(\frac{\mu(B_0)}{r_0^{Q(x_0)}}\right)^{\frac{1}{\gamma_{B_0}^-}}\left\|g \right\|_{L^{p(\cdot)}(\sigma B_0)}+\Lambda(B_0)\left\|u\right\|_{L^{p(\cdot)}(B_0)},
	\end{equation}
	where $\displaystyle \gamma:=\frac{Qp}{Q-sp}$ is the Sobolev conjugate exponent, $C_S$ is a constant from Theorem~\ref{sobolevpoincare}, and
	\begin{equation*}
		\Lambda(B_0):=\kappa_{\gamma(\cdot)}^2\max\left\{2,\left(\frac{2}{\mu(B_0)}\right)^{\frac{1}{\gamma_{B_0}^-}},\left(\frac{2}{\mu(B_0)}\right)^{\frac{1}{p_{B_0}^-}}\right\}\left\|1\right\|_{L^{\gamma(\cdot)}(B_0)}.
	\end{equation*}
\end{prop}

\begin{proof}
	Let $u\in M^{s(\cdot),p(\cdot)}(\sigma B_0)$ and let $g\in \mathcal{D}^{s(\cdot)}(u) \cap L^{p(\cdot)}(\sigma B_0)$. Then, by Theorem~\ref{sobolevpoincare} $(i)$ we know that $u\in L^{\gamma(\cdot)}(B_0)$ and 
	\begin{equation*}
		\inf_{c\in \mathbb R} \left\|u - c\right\|_{L^{\gamma(\cdot)}(B_0)} \leq C_S\left(\frac{\mu(B_0)}{r_0^{Q(x_0)}}\right)^{\frac{1}{\gamma_{B_0}^-}}\left\| g \right\|_{L^{p(\cdot)}(\sigma B_0)},
	\end{equation*}
	where $C_S$ is a constant from Theorem~\ref{sobolevpoincare}.
Now, by Proposition~\ref{medianlemma}, for every $c\in \mathbb R$, we have
	\begin{align*}
		\left\|u\right\|_{L^{\gamma(\cdot)}(B_0)} &\leq \kappa_{\gamma(\cdot)}^2\left(\left\|u-c\right\|_{L^{\gamma(\cdot)}(B_0)}+\left\|c-m_u(B_0)\right\|_{L^{\gamma(\cdot)}(B_0)}+\left\|m_u(B_0)\right\|_{L^{\gamma(\cdot)}(B_0)}\right) \\ & = \kappa_{\gamma(\cdot)}^2\left(\left\|u - c \right\|_{L^{\gamma(\cdot)}(B_0)}+\left|c-m_u(B_0)\right| \left\|1\right\|_{L^{\gamma(\cdot)}(B_0)}+\left|m_u(B_0)\right|\left\|1\right\|_{L^{\gamma(\cdot)}(B_0)}\right) \\ & \leq \kappa_{\gamma(\cdot)}^2\left\|u-c \right\|_{L^{\gamma(\cdot)}(B_0)}+ \Lambda(B_0) \left( \left\| u-c \right\|_{L^{\gamma(\cdot)}(B_0)} +\left\|u\right\|_{L^{p(\cdot)}(B_0)}\right).
	\end{align*}
	Hence
	\begin{align*}
		\left\|u \right\|_{L^{\gamma(\cdot)}(B_0)} &\leq \left(\kappa_{\gamma(\cdot)}^2+\Lambda(B_0)\right) \inf_{c\in \mathbb R} \left\|u- c\right\|_{L^{\gamma(\cdot)}(B_0)}+\Lambda(B_0)\left\|u\right\|_{L^{p(\cdot)}(B_0)} \\ & \leq \left(\kappa_{\gamma(\cdot)}^2+\Lambda(B_0)\right)C_S \left(\frac{\mu(B_0)}{r_0^{Q(x_0)}}\right)^{\frac{1}{\gamma_{B_0}^-}}\left\|g \right\|_{L^{p(\cdot)}(\sigma B_0)}+\Lambda(B_0)\left\|u\right\|_{L^{p(\cdot)}(B_0)},
	\end{align*}
	which completes the proof.
\end{proof}

\subsection{Global Sobolev inequalities}
%Let us observe that if $\mu(B(x,r)) \geq br^{Q(x)}$ holds for every $x\in X$, $r\in (0,1]$, then for every $\delta \in [1,\infty)$ and $r\in (0,\delta]$, $x\in X$ we have $\mu(B(x,r)) \geq b\delta^{-Q^+} r^{Q(x)}.$ Hence, in the case of bounded metric space $(X,d)$ by taking in Theorem \ref{sobolevpoincare} and Proposition \ref{localemb} $\delta:=4\diam X$, $\sigma:=2$ and $r_0:=2\diam X$ we obtain the following result.

Invoking Remark~\ref{measurethreshold} with $\delta':=4\diam X$ and specializing Theorem~\ref{sobolevpoincare} and Proposition~\ref{localemb} to the setting of a bounded metric space $(X,d)$ with $\sigma:=2$ and $r_0:=\delta'/\sigma=2\diam X$, we arrive at the following result.

\begin{tw}\label{boundedembedding}
Let $(X,d,\mu)$ be a bounded metric measure space and assume that there exist $Q\in \mathcal{P}_b^{\log}(X)$, $\delta\in(0,\infty)$, and $b\in (0,1]$ such that for every $x\in X$ and $r\in (0,\delta]$,
\begin{equation*}
	\mu\left(B(x,r)\right)\geq br^{Q(x)}.
\end{equation*}
Let $p,s\in \mathcal{P}_b^{\log}(X)$. Then, the following statements are valid.
\begin{enumerate}
	\item[(i)] If $sp \ll Q$, then, there exists a positive constant $\hat{C}_{S}$, depending only on $b$, $\delta$, $\diam X$, $\mu (X)$, $p$, $s$, $Q$, such that, for every pair of functions $u\in \dot{M}^{s(\cdot),p(\cdot)}(X,d,\mu)$ and $g\in \mathcal{D}^{s(\cdot)}(u)\cap L^{p(\cdot)}(X,\mu)$, there holds
	\begin{equation*}
		\inf_{c\in \mathbb R}\left\|u-c\right\|_{L^{\gamma(\cdot)}(X,\mu)} \leq \hat{C}_{S} \left\|g\right\|_{L^{p(\cdot)}(X,\mu)},
	\end{equation*}
	where $\displaystyle \gamma:=\frac{Qp}{Q-sp}$. Moreover, for every $u\in M^{s(\cdot),p(\cdot)}(X,d,\mu)$ and $g\in \mathcal{D}^{s(\cdot)}(u) \cap L^{p(\cdot)}(X,\mu)$
	\begin{equation*}
		\left\| u\right\|_{L^{\gamma(\cdot)}(X,\mu)} \leq \hat{C}_S\left( \left\|u \right\|_{L^{p(\cdot)}(X,\mu)} + \left\| g\right\|_{L^{p(\cdot)}(X,\mu)}\right).
	\end{equation*}
	\item[(ii)] If $sp=Q$, then there exist positive constants $\hat{C}_{MT1}, \hat{C}_{MT2}$, depending only on $b$, $\delta$, $\diam X$, $p$, $s$, $Q$, such that for every pair of functions $u\in \dot{M}^{s(\cdot),p(\cdot)}(X,d,\mu)$ and $g\in \mathcal{D}^{s(\cdot)}(u) \cap L^{p(\cdot)}(X,\mu)$ with $ \left\|g\right\|_{L^{p(\cdot)}(X,\mu)}> 0$, there holds
	\begin{equation*}
		\fint_{X} \exp\left(\hat{C}_{MT1}\frac{\left|u(x)-u_{X}\right|}{\left\|g\right\|_{L^{p(\cdot)}(X,\mu)}}\right)\mbox{d}\mu(x) \leq \hat{C}_{MT2}.
	\end{equation*}
	\item[(iii)] If $sp \gg Q$, then there exists a positive constant $\hat{C}_H$, depending only on $b$, $\delta$, $\diam X$, $p$, $s$, $Q$, such that for every pair of functions $u\in \dot{M}^{s(\cdot),p(\cdot)}(X,d,\mu)$ and $g\in \mathcal{D}^{s(\cdot)}(u)\cap L^{p(\cdot)}(X,\mu)$, there holds
	\begin{equation*}
		\left\|u-u_{X}\right\|_{L^{\infty}(X,\mu)}\leq \hat{C}_H \left\|g\right\|_{L^{p(\cdot)}(X,\mu)}.
	\end{equation*}
	Moreover, the function $u$ has a continuous representative $\tilde{u}$ on $X$ such that, for all $x,y\in X$,
	\begin{equation*}
		\left|\tilde{u}(x)-\tilde{u}(y)\right| \leq \hat{D}_H \|g\|_{L^{p(\cdot)}(X,\mu)}d(x,y)^{\alpha(x)},
	\end{equation*}
	where $\hat{D}_H$ is a positive constant and $\displaystyle \alpha:=s-\frac{Q}{p}$.
\end{enumerate}
\end{tw}

In general case, we have the following theorem.

\begin{tw}\label{doublingembedding}
	Let $(X,d,\mu)$ be a metric measure space and assume that there exist $Q\in \mathcal{P}_b^{\log}(X)$, $\delta\in(0,\infty)$, and $b\in (0,1]$ such that for every $x\in X$ and $r\in (0,\delta]$, 
	\begin{equation*}
		\mu(B(x,r))\geq br^{Q(x)}.
	\end{equation*}
	\begin{enumerate}
		\item[(i)] If $(X,d)$ is geometrically doubling and $\displaystyle \sup\left\{\mu(B(x,\delta)): x \in X\right\}<\infty,$ then for every $p,s\in \mathcal{P}^{\log}_b(X)$ satisfying $sp \ll Q$, the following embedding holds,
		\begin{equation*}
			M^{s(\cdot),p(\cdot)}(X,d,\mu) \hookrightarrow L^{\gamma(\cdot)}(X,\mu),
		\end{equation*}
		where $\displaystyle \gamma:=\frac{Qp}{Q-sp}$.
		\item[(ii)] If $(X,d)$ is geometrically doubling and $\displaystyle \sup\left\{\mu(B(x,\delta)): x \in X\right\}<\infty,$ then for every $p,s\in \mathcal{P}^{\log}_b(X)$ satisfying $sp=Q$, the following embedding holds,
		\begin{equation*}
			M^{s(\cdot),p(\cdot)}(X,d,\mu) \hookrightarrow L^{\beta(\cdot)}(X,\mu),
		\end{equation*}
		for every $\beta \in \mathcal{P}_b(X)$ such that $\beta \geq p.$
		\item[(iii)] For every $p,s \in \mathcal{P}^{\log}_b(X)$ satisfying $sp \gg Q$, the following embedding holds,
		\begin{equation*}
			M^{s(\cdot),p(\cdot)}(X,d,\mu) \hookrightarrow C^{0,\alpha(\cdot)}(X,d),
		\end{equation*}
		where $\displaystyle \alpha:=s-\frac{Q}{p}$. 
		
	\end{enumerate}
\end{tw}

\begin{proof}
We shall start proof of Theorem~\ref{doublingembedding} with the following proposition.

\begin{prop}\label{techniczne}
Let $(X,d,\mu)$ be a metric measure space. Assume that there exist $r\in (0,\infty)$, $M\in (0,\infty)$, set $I \subseteq \mathbb N$, and a sequence $\left\{x_i\right\}_{i\in I} \subseteq X$ satisfying
\begin{equation*}
	X=\bigcup_{i\in I} B(x_i,r)
\end{equation*}
and for every $x\in X$,
\begin{equation*}
	\# \left\{i\in I: x\in B(x_i,2r)\right\} \leq M.
\end{equation*}
Moreover, let $p,s,q \in \mathcal{P}_b(X)$ satisfy $q_{B(x_i,r)}^- \geq p_{B(x_i,2r)}^+$ for all $i \in I.$ Assume that there exists a constant $C(r)>0$ such that for every $i\in I$, $u\in M^{s(\cdot),p(\cdot)}(X,d,\mu)$, and $g\in \mathcal{D}^{s(\cdot)}(u) \cap L^{p(\cdot)}(X,\mu)$, one has $u\in L^{q(\cdot)}(B(x_i,r))$ and
\begin{equation}\label{local1}
\left\| u \right\|_{L^{q(\cdot)}(B(x_i,r))} \leq C(r)\left( \left\| u \right\|_{L^{p(\cdot)}(B(x_i,2r))} + \left\|g \right\|_{L^{p(\cdot)}(B(x_i,2r))}\right).
\end{equation}
Then, the following global continuous embedding holds
\begin{equation*}
	M^{s(\cdot),p(\cdot)}(X,d,\mu) \hookrightarrow L^{q(\cdot)}(X,\mu).
\end{equation*}
\end{prop}

\begin{proof}
We shall denote $B_i:=B(x_i,r)$. It suffices to prove that the set
	\begin{equation*}
		\mathcal{A}:=\left\{u\in M^{s(\cdot),p(\cdot)}(X,d,\mu): \left\|u\right\|_{M^{s(\cdot),p(\cdot)}(X)}< \min\left\{1,\frac{1}{2C(r)}\right\}\right\}
	\end{equation*}
	is bounded in $L^{q(\cdot)}(X,\mu)$. Let $u\in \mathcal{A}$ and let $g\in \mathcal{D}^{s(\cdot)}(u)\cap L^{p(\cdot)}(X,\mu)$ be such that
	\begin{equation}\label{normbound1}
		\left\|u\right\|_{L^{p(\cdot)}(X,\mu)}+\left\|g\right\|_{L^{p(\cdot)}(X,\mu)}<\min\left\{1,\frac{1}{2C(r)}\right\}.
	\end{equation}
	From \eqref{local1} we get that $\left\|u\right\|_{L^{q(\cdot)}\left(B_i\right)}<1$ for every $i\in I$. Now, using Proposition~\ref{rel} and the above facts, we have can estimate
	\begin{align*}
		\int_X \left|u(x)\right|^{q(x)}\mbox{d}\mu(x)&\leq \sum_{i\in I} \int_{B_i} \left|u(x)\right|^{q(x)}\mbox{d}\mu(x)\leq \sum_{i\in I} \left\|u\right\|_{L^{q(\cdot)}(B_i)}^{q_{B_i}^-}\\ &\leq \max\left\{1,C(r)\right\}^{q^+}\max\left\{1,2^{q^+-1}\right\}\sum_{i\in I} \left(\left\|u\right\|_{L^{p(\cdot)}\left(2B_i\right)}^{q_{B_i}^-}+\left\|g\right\|_{L^{p(\cdot)}\left(2B_i\right)}^{q_{B_i}^-}\right) \\ & \leq \tilde{C}(r)\sum_{i\in I}\left[\left(\int_{2B_i} \left|u(x)\right|^{p(x)}\mbox{d}\mu(x)\right)^{\frac{q_{B_i}^-}{p_{2B_i}^+}}+\left(\int_{2B_i} g(x)^{p(x)}\mbox{d}\mu(x)\right)^{\frac{q_{B_i}^-}{p_{2B_i}^+}}\right],
	\end{align*}
	where $\tilde{C}(r):=\max\left\{1,C(r)\right\}^{q^+}\max\left\{1,2^{q^+ - 1}\right\}.$
	
	Now, since $\left\|u\right\|_{L^{p(\cdot)}\left(2B_i\right)}+\left\|g\right\|_{L^{p(\cdot)}\left(2B_i\right)}<1$ and $q_{B_i}^- \geq p_{2B_i}^+$, we have that
	\begin{align*}
		\int_X \left|u(x)\right|^{q(x)}\mbox{d}\mu(x)&\leq \tilde{C}(r)\sum_{i\in I}\left[\left(\int_{2B_i} \left|u(x)\right|^{p(x)}\mbox{d}\mu(x)\right)^{\frac{q_{B_i}^-}{p_{2B_i}^+}}+\left(\int_{2B_i} g(x)^{p(x)}\mbox{d}\mu(x)\right)^{\frac{q_{B_i}^-}{p_{2B_i}^+}}\right] \\ & \leq \tilde{C}(r)\sum_{i\in I} \left(\int_{2B_i} \left|u(x)\right|^{p(x)}\mbox{d}\mu(x)+\int_{2B_i} g(x)^{p(x)}\mbox{d}\mu(x)\right)\\ & \leq \tilde{C}(r)M\left(\int_X \left|u(x)\right|^{p(x)}\mbox{d}\mu(x)+\int_X g(x)^{p(x)}\mbox{d}\mu(x)\right),
	\end{align*}
	where the last term is bounded due to inequality \eqref{normbound1}. Therefore $\mathcal{A}$ is bounded in $L^{q(\cdot)}(X,\mu)$ and the continuous embedding is proven.
\end{proof}
Let us return to the proof of Theorem~\ref{doublingembedding}. 
We shall start with proving $(i)$. Firstly, let us notice that local log-H\"older continuity of $p$ implies that $p$ is uniformly continuous. Moreover, since $p \ll \gamma$, uniform continuity of $p$ implies that there exists $\eta\in (0,\infty)$, such that for every ball $B:=B(x,r)\subseteq X$ with radius satisfying $r\in (0,\eta)$ we have inequality $p_B^+\leq \gamma_B^-$.
	
	Let us fix $r\in \left(0,\min\left\{\delta/2,\eta/2,1/2\right\}\right).$ Since the metric space $(X,d)$ is geometrically doubling, by Lemma~\ref{covering}, applied  with $R:=2r$, we can find constants $A,B>0$, set $I\subseteq \mathbb N$, and a sequence $\left\{x_i\right\}_{i\in I} \subseteq X$ such that the family of balls $\left\{B\left(x_i,r\right)\right\}_{i\in I}$ cover $X$ and have the property that, for every $x\in X$,
	\begin{equation*}
		\# \left\{i\in I: x\in B\left(x_i,2r\right)\right\}\leq A\cdot 2^B.
	\end{equation*}
	Now, from Proposition~\ref{localemb} with $\sigma:=2$, it follows that for every $i\in I$, $u\in M^{s(\cdot),p(\cdot)}(X,d,\mu)$, and $g\in \mathcal{D}^{s(\cdot)}(u)\cap L^{p(\cdot)}(X,\mu)$, there holds 
	\begin{equation*}
		\left\|u\right\|_{L^{\gamma(\cdot)}\left(B_i\right)}\leq \left(\kappa_{\gamma(\cdot)}^2+\Lambda(B_i)\right)C_S \left(\frac{\mu(B_i)}{r^{Q(x_i)}}\right)^{\frac{1}{\gamma_{B_i}^-}}\left\|g \right\|_{L^{p(\cdot)}(2 B_i)}+\Lambda(B_i)\left\|u\right\|_{L^{p(\cdot)}(B_i)},
	\end{equation*}
	where $B_i:=B\left(x_i,r\right)$ and $\Lambda(B_i)$ and $C_S$ are constants from Proposition~\ref{localemb}. Let
	\begin{equation*}
		C:= \max\left\{1, \sup_{x\in X} \mu(B(x,\delta))\right\}.
	\end{equation*}
	Then, we know that $C\in [1,\infty)$ and by the lower Ahlfors-regularity of the measure, we have
	\begin{align*}
		\Lambda(B_i)\leq \kappa_{\gamma(\cdot)}^2\max\left\{2,\left(\frac{2}{br^{Q^+}}\right)^{\frac{1}{\gamma^-}},\left(\frac{2}{br^{Q^+}}\right)^{\frac{1}{p^-}}\right\}C^{\frac{1}{\gamma^-}}
	\end{align*}
	and
	\begin{equation*}
		\left( \frac{\mu(B_i)}{r^{Q(x_i)}} \right)^{\frac{1}{\gamma_{B_i}^-}} \leq \left(\frac{C}{r^{Q^+}}\right)^{\frac{1}{\gamma^-}}.
	\end{equation*}
	Hence,
	\begin{equation}\label{local}
		\left\|u\right\|_{L^{\gamma(\cdot)}\left(B_i\right)}\leq C(r)\left(\left\|u\right\|_{L^{p(\cdot)}\left(2B_i\right)}+\left\|g\right\|_{L^{p(\cdot)}\left(2B_i\right)}\right),
	\end{equation}
	where $C(r)$ is a constant which does not depend on $x_i$. Moreover, because $r\in (0,\eta/2)$, we know from the first part of the proof that $p_{2B_i}^+ \leq \gamma_{2B_i}^-$ for all $i \in \mathbb N$. As a consequence, $p_{2B_i}^+\leq \gamma_{B_i}^-.$ Hence, Proposition~\ref{techniczne} applied for $q:=\gamma$, $M:=A\cdot 2^B$ and balls $\left\{B_i\right\}_{i=1}^{\infty}$ yields
	\begin{equation*}
		M^{s(\cdot),p(\cdot)}(X,d,\mu) \hookrightarrow L^{\gamma(\cdot)}(X,\mu)
	\end{equation*}
	and the proof of $(i)$ is complete.
	
	Now, we pass to the proof of $(ii)$. Let us fix number $k\in (\beta^+,\infty)$ and define an exponent $t$ by setting
	\begin{equation*}
		t:=\frac{Q\left(k-p\right)}{pk}.
	\end{equation*}
	Then, clearly $t\in \mathcal{P}_b^{\log}(X)$ and $t\ll s$. Hence, $tp \ll sp =Q$. Moreover, $k=Qp/(Q-tp)$. Hence, applying respectively Proposition~\ref{embeddingsbetween} and recently proven $(i)$ we get
	\begin{equation*}
		M^{s(\cdot),p(\cdot)}(X,d,\mu) \hookrightarrow M^{t(\cdot),p(\cdot)}(X,d,\mu) \hookrightarrow L^{k}(X,\mu).
	\end{equation*}
	On the other hand, obviously
	\begin{equation*}
		M^{s(\cdot),p(\cdot)}(X,d,\mu) \hookrightarrow L^{p(\cdot)}(X,\mu).
	\end{equation*}
	Hence, by Proposition~\ref{interpolacyjny} applied for $q_0:=p$, $q:=\beta$, $q_1=k$ we get
	\begin{equation*}
		M^{s(\cdot),p(\cdot)}(X,d,\mu) \hookrightarrow L^{\beta(\cdot)}(X,\mu)
	\end{equation*}
	and $(ii)$ is proven.

	There remains to prove $(iii)$. Fix $u\in M^{s(\cdot),p(\cdot)}(X,d,\mu)$ and let $N_u\subseteq X$ be the null set such that $\left|u(x)\right| \leq \left\|u\right\|_{L^{\infty}(X,\mu)}$ for every $x\in X\setminus N_u$. Next, using Lemma~\ref{pokrycieprodukt} we can find\footnote{Recall that every metric measure space is necessarily separable.} a sequence $\left\{z_i\right\}_{i=1}^\infty\subseteq X$ such that
	\begin{equation}\label{zawieranie}
		\left\{ (x,y)\in X \times X: 0<d(x,y)<\delta/4 \right\} \subseteq \bigcup_{i=1}^{\infty} \left\{ (x,y)\in X \times X: x,y\in B\left(z_i,\delta/2\right) \textnormal{ and } x \neq y\right\}.
	\end{equation}
	From Theorem~\ref{sobolevpoincare} $(iii)$ with $\sigma:=2$ and $r_0:=\delta/2$, we have  that for every $i\in \mathbb N$, there is a continuous function $\tilde{u}_i: B(z_i,\delta/2)\to \mathbb{R}$ and the null set $A_i\subseteq B(z_i,\delta/2)$ such that $u=\tilde{u}_i$ on $B(z_i,\delta/2)\setminus A_i$. Now, define
	\begin{equation*}
		E:=N_u \cup \bigcup_{i=1}^{\infty} A_i.
	\end{equation*}
	Obviously $\mu(E)=0$. Moreover, we notice that
	\begin{align}\label{pierwsza}
		\sup_{\substack{x, y\in X\setminus E\\ x \neq y}} \frac{\left|u(x)-u(y)\right|}{d(x,y)^{\alpha(x)}} & \leq \sup_{\substack{x, y\in X\setminus E\\ 0<d(x,y)<\delta/4}} \frac{\left|u(x)-u(y)\right|}{d(x,y)^{\alpha(x)}} + \sup_{\substack{x, y\in X\setminus E\\ d(x,y)\geq \delta/4}} \frac{\left|u(x)-u(y)\right|}{d(x,y)^{\alpha(x)}} \nonumber \\ & \leq \sup_{\substack{x, y\in X\setminus E\\ 0<d(x,y)<\delta/4}} \frac{\left|u(x)-u(y)\right|}{d(x,y)^{\alpha(x)}} + G(\delta,\alpha) \left\|u\right\|_{L^{\infty}(X,\mu)}, 
	\end{align}
	where 
	\begin{equation*}
	G(\delta,\alpha):=2 \max\left\{\left(\frac{4}{\delta}\right)^{\alpha^+}, \left(\frac{4}{\delta}\right)^{\alpha^-}\right\}.
	\end{equation*}
	Let $g\in \mathcal{D}^{s(\cdot)}(u) \cap L^{p(\cdot)}(X,\mu)$. By the virtue of Theorem~\ref{sobolevpoincare} $(iii)$, we conclude that for every $i\in \mathbb N$, we have $\displaystyle u\in L^{\infty}\left(B\left(z_i,\delta/2\right)\right)$ and
	\begin{equation}\label{ineq}
		\left\|u-u_{B\left(z_i,\delta/2\right)}\right\|_{L^{\infty}\left(B\left(z_i,\delta/2\right)\right)}\leq C_H\left(\frac{\delta}{2}\right)^{\alpha(z_i)}\left\|g\right\|_{L^{p(\cdot)}\left(B\left(z_i,\delta\right)\right)}\leq \tilde{C}_H \left\|g\right\|_{L^{p(\cdot)}(X,\mu)},
	\end{equation}
	where
	\begin{equation*}
		\tilde{C}_H:=\max\left\{\left(\frac{\delta}{2}\right)^{\alpha^+}, \left(\frac{\delta}{2}\right)^{\alpha^-} \right\}.
	\end{equation*}
	Thus, by Proposition~\ref{med1}, Proposition~\ref{medianlemma}, inequality \eqref{ineq}, and the very definition of the median, we get
	\begin{align}\label{druga}
		&\left\|u\right\|_{L^{\infty}(X,\mu)} = \sup_{i\in \mathbb N} \left\|u\right\|_{L^{\infty}\left(B\left(z_i,\delta/2\right)\right)} \leq \sup_{i\in \mathbb N} \left\|u-m_u(B\left(z_i,\delta/2\right))\right\|_{L^{\infty}(B\left(z_i,\delta/2\right))}+\sup_{i\in \mathbb N} m_{\left|u\right|}(B\left(z_i,\delta/2\right))\nonumber \\ & \leq \sup_{i\in\mathbb N} \left\|u-u_{B\left(z_i,\delta/2\right)}\right\|_{L^{\infty}(B\left(z_i,\delta/2\right))} + \sup_{i \in \mathbb N} \left|u_{B\left(z_i,\delta/2\right)}-m_u(B\left(z_i,\delta/2\right))\right|+\sup_{i\in \mathbb N} m_{\left|u\right|}(B\left(z_i,\delta/2\right))\nonumber \\ &\leq  \sup_{i\in\mathbb N} \left\|u-u_{B\left(z_i,\delta/2\right)}\right\|_{L^{\infty}(B\left(z_i,\delta/2\right))}+\sup_{i\in \mathbb N} m_{\left|u-u_{B\left(z_i,\delta/2\right)}\right|}(B\left(z_i,\delta/2\right))+\sup_{i\in \mathbb N} m_{\left|u\right|}\left(B\left(z_i,\delta/2\right)\right) \nonumber \\ & \leq 2\sup_{i\in\mathbb N} \left\|u-u_{B\left(z_i,\delta/2\right)}\right\|_{L^{\infty}(B\left(z_i,\delta/2\right))}+\sup_{i\in \mathbb N}\max\left\{2,\left(\frac{2}{\mu(B\left(z_i,\delta/2\right))}\right)^{\frac{1}{p^-_{B\left(z_i,\delta/2\right)}}}\right\}\left\|u\right\|_{L^{p(\cdot)}(B\left(z_i,\delta/2\right))} \nonumber \\ & \leq 2\sup_{i\in\mathbb N} \left\|u-u_{B\left(z_i,\delta/2\right)}\right\|_{L^{\infty}(B\left(z_i,\delta/2\right))}+\sup_{i\in \mathbb N}\max\left\{2,\left(\frac{2\cdot 2^{Q(z_i)}}{b\cdot \delta^{Q(z_i)}}\right)^{\frac{1}{p^-_{B\left(z_i,\delta/2\right)}}}\right\}\left\|u\right\|_{L^{p(\cdot)}(B\left(z_i,\delta/2\right))} \nonumber \\ & \leq 2\tilde{C}_H \left\|g\right\|_{L^{p(\cdot)}(X,\mu)}+ F(b,\delta,Q,p) \left\|u\right\|_{L^{p(\cdot)}(X,\mu)}, 
	\end{align}
	where $F(b,\delta,Q,p)$ is a constant depending on $b$, $\delta$, $Q$ and $p$. On the other hand, using again Theorem~\ref{sobolevpoincare} $(iii)$ and the inclusion \eqref{zawieranie} we obtain
	\begin{align}\label{trzecia}
		\sup_{\substack{x, y\in X\setminus E\\ 0<d(x,y)<\delta/4}} \frac{\left|u(x)-u(y)\right|}{d(x,y)^{\alpha(x)}} &\leq \sup_{i\in \mathbb N} \sup_{\substack{x,y \in B(z_i,\delta/2) \setminus E\\ x\neq y}} \frac{\left|\tilde{u}_i(x)-\tilde{u}_i(y)\right|}{d(x,y)^{\alpha(x)}} \leq D_H\left(\delta/2\right) \sup_{i\in \mathbb N}\left\|g\right\|_{L^{p(\cdot)}(B(z_i,\delta))} \nonumber \\ & \leq D_H\left(\delta/2\right) \left\|g\right\|_{L^{p(\cdot)}(X,\mu)}.
	\end{align}
	Combining this with \eqref{pierwsza} and \eqref{druga}, yields
	\begin{equation*}
		\sup_{\substack{x,y\in X\setminus E \\ x \neq y}} \frac{\left|u(x)-u(y)\right|}{d(x,y)^{\alpha(x)}} \leq \left(2\tilde{C}_HG(\delta,\alpha)+D_H\left(\delta/2\right)\right)\left\|g\right\|_{L^{p(\cdot)}(X,\mu)}+ G(\delta,\alpha)F(b,\delta,Q,p)\left\|u\right\|_{L^{p(\cdot)}(X,\mu)},
	\end{equation*}
	and hence $u$ has an $\alpha(\cdot)$-H\"older continuous representative on $X$ (denoted by $u^*$).
	Finally, using \eqref{pierwsza}, \eqref{druga}, and \eqref{trzecia}, we obtain
	\begin{align*}
		\left\|u^*\right\|_{C^{0,\alpha(\cdot)}(X,d)} &=  \left\|u^*\right\|_{C(X)}+\left[u^*\right]_{\alpha(\cdot),X} =  \left\|u^*\right\|_{L^{\infty}(X,\mu)}+\left[u^*\right]_{\alpha(\cdot),X} \\ & \leq \left[D_H(\delta/2)+2\tilde{C}_H(G(\delta,\alpha)+1)\right]\left\|g\right\|_{L^{p(\cdot)}(X,\mu)}+\left(G(\delta,\alpha)+1\right)F(b,\delta,Q,p)\left\|u\right\|_{L^{p(\cdot)}(X,\mu)},
	\end{align*}
	and the proof is complete.
\end{proof}	

The following example illustrates that the global embeddings in Theorem~\ref{doublingembedding} may fail if $Q\not\in \mathcal{P}_b^{\log}(X)$.
\begin{example}
Let $X:=B(0,1)\subseteq \mathbb R^n$ for some $n\in\mathbb N.$ Moreover, let $\left| \cdot -\cdot\right|$ be Euclidean distance on $X$ and let $\mu:=\lambda_n+\delta_0$ be defined on Borel subsets of $X$, where $\lambda_n$ is the $n$-dimensional Lebesgue measure and $\delta_0$ is Dirac delta at $0$. For a fixed $\beta \in (0,n)$, define a function $Q: X \to \mathbb R$ by setting
\begin{equation*}
Q(x):=\left\{\begin{array}{ll}\beta ,& \textnormal{ for } x=0,\\ n,& \textnormal{ for } x\in X \setminus \left\{0\right\}.\end{array} \right.
\end{equation*}
Then, measure $\mu$ is lower Ahlfors $Q(\cdot)$-regular, but for every number $p\in (Q^+,\infty)$,
\begin{equation*}
M^{1,p}(X,\left|\cdot-\cdot\right|,\mu) \nsubseteq C^{0,\alpha(\cdot)}(X,\left|\cdot -\cdot\right|),
\end{equation*}
where $\displaystyle \alpha:=1-\frac{Q}{p}.$
\end{example}

\begin{proof}
First, note that the Lebesgue measure, $\lambda_n$, satisfies the measure density condition on $B(0,1)$, that is, there exists a constant $A>0$ such that, for every $x\in B(0,1)$ and $r\in (0,1]$, there holds
\begin{equation*}
	\lambda_n\left(B(0,1) \cap B(x,r)\right) \geq Ar^n.
\end{equation*}
Hence, for every $x\in X \setminus \left\{0\right\}$,
\begin{equation*}
	\mu(B(x,r))\geq \lambda_n(B(0,1) \cap B(x,r))\geq A r^n.
\end{equation*}
Moreover,
\begin{equation*}
	\mu(B(0,r)) = 1+\lambda_n(B(0,r))\geq r^{\beta}+Ar^n \geq r^{\beta}.
\end{equation*}
Thus, $\mu$ is lower Ahlfors $Q(\cdot)$-regular. 

Fix, $\theta\in \left(1-n/p,1-\beta/p\right)$ and consider the function $u: X \to \mathbb R$ defined by $u(x):=\left|x\right|^{\theta}$ for $x\in X.$
We show that $u\in M^{1,p}(X,\left|\cdot - \cdot\right|,\mu).$ Using polar coordinates, we have
\begin{equation*}
	\int_{X} \left|u(x)\right|^{p} \mbox{d}\mu(x)=n \omega_n\int_0^1 r^{\theta p+n-1}\mbox{d}r,
\end{equation*}
where $\omega_n$ is the Lebesgue measure of unit ball in $\mathbb R^n$. Since $\theta p +n -1 >-1$ the last integral is finite. Hence, $u\in L^p(X,\mu).$ Now, using the mean value theorem, for every $x,y\in X\setminus \left\{0\right\}$, we get
\begin{equation*}
	\left| \left|x\right|^{\theta} - \left|y\right|^{\theta}\right| \leq \left|x-y\right| \left(\theta \left|x\right|^{\theta -1} +\theta \left|y\right|^{\theta -1}\right) \leq \left|x-y\right| \left(\left|x\right|^{\theta -1} +\left|y\right|^{\theta -1}\right).
\end{equation*}
As a consequence, the function $g: X \to [0,\infty)$, defined by
\begin{equation*}
g(x):=\left\{\begin{array}{ll} \left|x\right|^{\theta -1},& \textnormal{ for } x\in X\setminus \left\{0 \right\}, \\ 0,& \textnormal{ for } x=0,\end{array}\right.		\end{equation*}
is a scalar 1-gradient of $u.$ Moreover, using polar coordinates once again gives
\begin{equation*}
	\int_X g(x)^{p} \mbox{d}\mu(x)=n \omega_n \int_0^1 r^{\theta p-p +n-1}\mbox{d}r ,
\end{equation*}
where the last integral is finite, since $\theta p-p+n-1>-1.$ Therefore, $g\in L^p(X,\mu)$ and hence we conclude that $u\in M^{1,p}(X,\left|\cdot - \cdot \right|,\mu).$

On the other hand, for all $x\in X \setminus \left\{0\right\}$ we have
\begin{equation*}
	\frac{\left|u(x)-u(0)\right|}{\left|x-0\right|^{\alpha(0)}}=\left|x\right|^{\theta-1+\beta/p} \stackrel{x\to 0}{\longrightarrow} \infty,
\end{equation*}
since $\theta-1+\beta/p<0.$ Then, $u\notin C^{0,\alpha(\cdot)}(X,\left|\cdot - \cdot  \right|).$
\end{proof}

\subsection{Remarks for Besov and Triebel-Lizorkin spaces}
		
\begin{tw}\label{localtriebellizorkin}
	Let $(X,d,\mu)$ be a metric measure space and assume that there exist $Q\in \mathcal{P}_b^{\log}(X)$, $\delta\in (0,\infty)$, and $b\in (0,1]$ such that for every $x\in X$ and $r\in (0,\delta]$,
	\begin{equation*}
		\mu\left(B(x,r)\right)\geq br^{Q(x)}.
	\end{equation*}
	Let $p,s\in \mathcal{P}_b^{\log}(X)$ and $\sigma \in (1,\infty)$. Then, the following statements are valid.
	\begin{enumerate}
		\item[(i)] If $sp \ll Q$, then, there exists a  positive constant $C_{S}$, depending only on $b$, $\delta$, $\sigma$, $p$, $s$, and $Q$, such that for every ball $B_0:=B(x_0,r_0)\subseteq X$ with $r_0\leq \delta/\sigma$, $q\in \mathcal{P}(\sigma B_0)$, $u\in \dot{M}^{s(\cdot)}_{p(\cdot),q(\cdot)}(\sigma B_0)$,  and $g\in \mathbb{D}^{s(\cdot)}(u)\cap L^{p(\cdot)}(\ell^{q(\cdot)}(\sigma B_0))$, there holds 
		\begin{equation*}
			\inf_{c\in \mathbb R}\left\|u-c\right\|_{L^{\gamma(\cdot)}(B_0)} \leq C_{S}\left(\frac{\mu(B_0)}{r_0^{Q(x_0)}}\right)^{\frac{1}{\gamma^-_{B_0}}} \left\|g\right\|_{L^{p(\cdot)}(\ell^{q(\cdot)}(\sigma B_0))},
		\end{equation*}
		where $\displaystyle \gamma:=\frac{Qp}{Q-sp}$.
		\item[(ii)] If $sp=Q$, then there exist positive constants $C_{MT1}, C_{MT2}$,  depending only on $b$, $\delta$, $\sigma$, $p$, $s$, and $Q$, such that for every ball $B_0:=B(x_0,r_0)\subseteq X$ with $r_0\leq \delta/\sigma$, $q\in \mathcal{P}(\sigma B_0)$, $u\in \dot{M}^{s(\cdot)}_{p(\cdot),q(\cdot)}(\sigma B_0)$, and $g\in \mathbb{D}^{s(\cdot)}(u)\cap L^{p(\cdot)}(\ell^{q(\cdot)}(\sigma B_0))$ with $ \left\|g\right\|_{L^{p(\cdot)}(\ell^{q(\cdot)}(\sigma B_0))}> 0$, there holds
		\begin{equation*}
			\fint_{B_0} \exp\left(C_{MT1}\frac{\left|u(x)-u_{B_0}\right|}{\left\|g\right\|_{L^{p(\cdot)}(\ell^{q(\cdot)}(\sigma B_0))}}\right)\mbox{d}\mu(x) \leq C_{MT2}.
		\end{equation*}
		\item[(iii)] If $sp \gg Q$, then there exists a positive constant $C_H$, depending only on $b$, $\delta$, $\sigma$, $p$, $s$, and $Q$, such that for every ball $B_0:=B(x_0,r_0)\subseteq X$ with $r_0 \leq \delta/\sigma$, $q\in \mathcal{P}(\sigma B_0)$, $u\in \dot{M}^{s(\cdot)}_{p(\cdot),q(\cdot)}(\sigma B_0)$, and $g\in \mathbb{D}^{s(\cdot)}(u)\cap L^{p(\cdot)}(\ell^{q(\cdot)}(\sigma B_0))$, there holds
		\begin{equation*}
			\left\|u-u_{B_0}\right\|_{L^{\infty}(B_0)}\leq C_Hr_0^{\alpha(x_0)}\left\|g\right\|_{L^{p(\cdot)}(\ell^{q(\cdot)}(\sigma B_0))},
		\end{equation*}
		where $\displaystyle \alpha:=s-\frac{Q}{p}$.
		Moreover, function $u$ has a continuous representative $\tilde{u}$ on $B_0$ satisfying that, for all $x,y\in B_0$,
		\begin{equation*}
			\left|\tilde{u}(x)-\tilde{u}(y)\right| \leq D_H(r_0) \left\|g\right\|_{L^{p(\cdot)}(\ell^{q(\cdot)}(\sigma B_0))}d(x,y)^{\alpha(x)},
		\end{equation*}
		where
		\begin{equation*}
			D_H(r_0):=2^{\alpha^++1}C_H \left(\frac{\sigma \delta}{(\sigma -1)r_0}\right)^{\alpha^+}.
		\end{equation*}
	\end{enumerate}
In addition, if $q\leq p$, then all of the statements above continue to be remain valid after replacing $\dot{M}^{s(\cdot)}_{p(\cdot),q(\cdot)}$ by $\dot{N}^{s(\cdot)}_{p(\cdot),q(\cdot)}$ and $L^{p(\cdot)}(\ell^{q(\cdot)})$ by $\ell^{q(\cdot)}(L^{p(\cdot)}).$
\end{tw}

\begin{proof}
The proofs of $(i)$ and $(iii)$ follow immediately from Proposition~\ref{embeddingsbetween} $(v)$, together with Theorem~\ref{sobolevpoincare}  $(i)$ and $(iii)$, respectively. The statement $(ii)$ follows from Proposition~\ref{embeddingsbetween} $(v)$, Lemma~\ref{gradientzero}, and Theorem~\ref{sobolevpoincare} $(ii)$. Moreover, by Proposition~\ref{embeddingsbetween} $(vi)$ the statements $(i)$, $(ii)$, $(iii)$ continue to be true if we replace $\dot{M}^{s(\cdot)}_{p(\cdot),q(\cdot)}$ by $\dot{N}^{s(\cdot)}_{p(\cdot),q(\cdot)}$ and $L^{p(\cdot)}(\ell^{q(\cdot)})$ by $\ell^{q(\cdot)}(L^{p(\cdot)}).$
\end{proof}

The embeddings for Haj\l{}asz--Besov spaces $\dot{N}^{s(\cdot)}_{p(\cdot),q(\cdot)}$ in Theorem~\ref{localtriebellizorkin} are restricted to the case when $q\leq p$; however, an upper bound on the exponent $q$ is to be expected (see, e.g, \cite[Remark~4.17]{AYY24}). In the following theorem, we prove that one can relax the restriction on $q$ and still obtain Sobolev-type embeddings for  $\dot{N}^{s(\cdot)}_{p(\cdot),q(\cdot)}$ with the critical exponent $Q/s$ replaced by $Q/\varepsilon$, where $\varepsilon \in \mathcal{P}_b^{\log}(X)$ is any function satisfying $\varepsilon \ll s$.

\begin{tw}\label{localbesov}
	Let $(X,d,\mu)$ be a metric measure space and assume that there exist $Q\in \mathcal{P}_b^{\log}(X)$, $\delta \in (0,\infty)$, and $b\in (0,1]$ such that for every $x\in X$ and $r\in (0,\delta]$,
	\begin{equation*}
		\mu\left(B(x,r)\right)\geq br^{Q(x)}.
	\end{equation*}
	Let $p\in \mathcal{P}_b^{\log}(X)$, $s\in \mathcal{P}_b(X)$, and $\sigma \in (1,\infty)$. Moreover, let $\varepsilon \in \mathcal{P}_b^{\log}(X)$ be such that $\varepsilon \ll s$. Then, the following statements are valid.
	\begin{enumerate}
		\item[(i)] If $\varepsilon p \ll Q$, then, there exists a positive constant $C_{S}'$, depending only on $b$, $\delta$, $\sigma$, $p$, $\varepsilon$, $s$, and $Q$, such that for every ball $B_0:=B(x_0,r_0)\subseteq X$ with $r_0\leq \delta/\sigma$, $q\in \mathcal{P}(\sigma B_0)$,  $u\in \dot{N}^{s(\cdot)}_{p(\cdot),q(\cdot)}(\sigma B_0)$, and $g\in \mathbb{D}^{s(\cdot)}(u)\cap \ell^{q(\cdot)}(L^{p(\cdot)}(\sigma B_0))$, there holds 
		\begin{equation*}
			\inf_{c\in \mathbb R}\left\|u-c\right\|_{L^{\gamma_{\varepsilon}(\cdot)}(B_0)} \leq C_{S}' \left(\frac{\mu(B_0)}{r_0^{Q(x_0)}}\right)^{\frac{1}{{\gamma_{\varepsilon}}^-_{B_0}}} \left\|g\right\|_{\ell^{q(\cdot)}(L^{p(\cdot)}(\sigma B_0))},
		\end{equation*}
		where $\displaystyle \gamma_{\varepsilon}:=\frac{Qp}{Q-\varepsilon p}$.
		\item[(ii)] If $\varepsilon p=Q$, then there exist positive constants $C_{MT1}', C_{MT2}'$, depending only on $b$, $\delta$, $\sigma$, $p$, $s$, $\varepsilon$, and $Q$, such that for every ball $B_0:=B(x_0,r_0)\subseteq X$ with $r_0\leq \delta/\sigma$, $q\in \mathcal{P}(\sigma B_0)$, $u\in \dot{N}^{s(\cdot)}_{p(\cdot),q(\cdot)}(\sigma B_0)$, and $g\in \mathbb{D}^{s(\cdot)}(u)\cap \ell^{q(\cdot)}(L^{p(\cdot)}(\sigma B_0))$ with $ \left\|g\right\|_{\ell^{q(\cdot)}(L^{p(\cdot)}(\sigma B_0))}> 0$, there holds
		\begin{equation*}
			\fint_{B_0} \exp\left(C_{MT1}' \frac{\left|u(x)-u_{B_0}\right|}{\left\|g\right\|_{\ell^{q(\cdot)}(L^{p(\cdot)}(\sigma B_0))}}\right)\mbox{d}\mu(x) \leq C_{MT2}'.
		\end{equation*}
		\item[(iii)] If $\varepsilon p \gg Q$, then there exists a positive constant $C_H'$,  depending only on $b$, $\delta$, $\sigma$, $p$, $s$, $\varepsilon$, and $Q$, such that for every ball $B_0:=B(x_0,r_0)\subseteq X$ with $r_0 \leq \delta/\sigma$, $q\in \mathcal{P}(\sigma B_0)$, $u\in \dot{N}^{s(\cdot)}_{p(\cdot),q(\cdot)}(\sigma B_0)$,  and  $g\in \mathbb{D}^{s(\cdot)}(u)\cap \ell^{q(\cdot)}(L^{p(\cdot)}(\sigma B_0))$, there holds 
		\begin{equation*}
			\left\|u-u_{B_0}\right\|_{L^{\infty}(B_0)}\leq C_H' r_0^{\alpha_{\varepsilon}(x_0)} \left\|g\right\|_{\ell^{q(\cdot)}(L^{p(\cdot)}(\sigma B_0))},
		\end{equation*}
		where $\displaystyle \alpha_{\varepsilon}:=\varepsilon -\frac{Q}{p}$.
		Moreover, function $u$ has a continuous representative $\tilde{u}$ on $B_0$ such that for all $x,y\in B_0$,
		\begin{equation*}
		\left|\tilde{u}(x)-\tilde{u}(y)\right| \leq D_H'(r_0) \left\|g\right\|_{\ell^{q(\cdot)}(L^{p(\cdot)}(\sigma B_0))}d(x,y)^{\alpha_{\varepsilon}(x)},
		\end{equation*}
		where $D_H'(r_0)=D_H(r_0)\cdot \zeta(p,s,\varepsilon,\delta)$. Here, $D_H(r_0)$ is a constant from Theorem~\ref{sobolevpoincare} and $\zeta(p,s,\varepsilon,\delta)$ is a constant from Proposition~\ref{embeddingsbetween} $(ix)$.
	\end{enumerate}			
\end{tw}

\begin{proof}
The proof of $(i)$ and $(iii)$ follows immediately from Proposition~\ref{embeddingsbetween} $(ix)$, together with Theorem~\ref{sobolevpoincare}  $(i)$ and $(iii)$, respectively. Statement $(ii)$ follows from Proposition~\ref{embeddingsbetween} $(ix)$, Lemma~\ref{gradientzero}, and Theorem~\ref{sobolevpoincare} $(ii)$.
\end{proof}

Using Remark~\ref{measurethreshold} with $\delta':=4\diam X$, Theorems~\ref{localtriebellizorkin} and \ref{localbesov} with $\sigma:=2$ and $r_0:=\delta'/\sigma=2\diam X$, and Propositions~\ref{embeddingsbetween}, \ref{localemb}, we obtain the following embedding theorems in the setting of bounded metric space.

\begin{tw}\label{boundedtriebel}
	Let $(X,d,\mu)$ be a bounded metric measure space and assume that there exist $Q\in \mathcal{P}_b^{\log}(X)$, $\delta \in (0,\infty)$, and $b\in (0,1]$ such that for every $x\in X$ and $r\in (0,\delta]$,
	\begin{equation*}
		\mu(B(x,r)) \geq br^{Q(x)}.
	\end{equation*}
	Let $p,s\in \mathcal{P}_b^{\log}(X)$ and $q\in \mathcal{P}(X)$. Then, the following statements are valid.
	\begin{enumerate}
		\item[(i)] If $s p \ll Q$, then, there exists a positive constant $\tilde{C}_S$, depending only on $b$, $\delta$, $\diam X$, $\mu(X)$, $p$, $s$, $Q$, such that, for every pair of functions $u\in \dot{M}^{s(\cdot)}_{p(\cdot),q(\cdot)}(X,d,\mu)$ and $g\in \mathbb{D}^{s(\cdot)}(u)\cap L^{p(\cdot)}(\ell^{q(\cdot)}(X,\mu))$, there holds
		\begin{equation*}
			\inf_{c\in \mathbb R} \left\| u-c\right\|_{L^{\gamma(\cdot)}(X,\mu)} \leq \tilde{C}_S \left\| g \right\|_{L^{p(\cdot)}(\ell^{q(\cdot)}(X,\mu))},
		\end{equation*}
		where $\displaystyle\gamma:=\frac{Qp}{Q-s p}$. 
		
		\noindent Moreover, for every $u\in M^{s(\cdot)}_{p(\cdot),q(\cdot)}(X,d,\mu)$ and $g\in \mathbb{D}^{s(\cdot)}(u) \cap L^{p(\cdot)}(\ell^{q(\cdot)}(X,\mu))$
		\begin{equation*}
			\left\| u \right\|_{L^{\gamma(\cdot)}(X,\mu)} \leq \tilde{C}_S\left( \left\| u \right\|_{L^{p(\cdot)}(X,\mu)} + \left\|g \right\|_{L^{p(\cdot)}(\ell^{q(\cdot)}(X,\mu))}\right).
		\end{equation*}
		\item[(ii)] If $s p=Q$, then there exist positive constants $\tilde{C}_{MT1}$, $\tilde{C}_{MT2}$, depending only on $b$, $\delta$, $\diam X$, $p$, $s$, $Q$, such that for every pair of functions $u\in \dot{M}^{s(\cdot)}_{p(\cdot),q(\cdot)}(X,d,\mu)$ and $g\in \mathbb{D}^{s(\cdot)}(u) \cap L^{p(\cdot)}(\ell^{q(\cdot)}(X,\mu))$ with $\left\| g\right\|_{L^{p(\cdot)}(\ell^{q(\cdot)}(X,\mu))}>0$, there holds
		\begin{equation*}
			\fint_X \exp\left(\tilde{C}_{MT1} \frac{\left|u(x)-u_X\right|}{\left\| g \right\|_{L^{p(\cdot)}(\ell^{q(\cdot)}(X,\mu))}}\right) \mbox{d}\mu(x) \leq \tilde{C}_{MT2}.
		\end{equation*}
		\item[(iii)] If $s p\gg Q$, then there exists a positive constant $\tilde{C}_H$, depending only on $b$, $\delta$, $\diam X$, $p$, $s$, $Q$, such that for every pair of functions $u\in \dot{M}^{s(\cdot)}_{p(\cdot),q(\cdot)}(X,d,\mu)$ and $g\in \mathbb{D}^{s(\cdot)}(u)\cap L^{p(\cdot)}(\ell^{q(\cdot)}(X,\mu))$, there holds
		\begin{equation*}
			\left\| u-u_X \right\|_{L^{\infty}(X,\mu)} \leq \tilde{C}_H \left\|g\right\|_{L^{p(\cdot)}(\ell^{q(\cdot)}(X,\mu))}.
		\end{equation*}
		Moreover, the function $u$ has a continuous representative $\tilde{u}$ on $X$ such that, for all $x,y\in X$,
		\begin{equation*}
			\left| \tilde{u}(x)-\tilde{u}(y)\right| \leq \tilde{D}_H \left\| g\right\|_{L^{p(\cdot)}(\ell^{q(\cdot)}(X,\mu))} d(x,y)^{\alpha(x)},
		\end{equation*}
		there $\tilde{D}_H$ is a positive constant and $\displaystyle \alpha:=s-\frac{Q}{p}$.
	\end{enumerate}
	In addition, if $q\leq p$, then all of the statements above continue to be remain valid after replacing $\dot{M}^{s(\cdot)}_{p(\cdot),q(\cdot)}$ by $\dot{N}^{s(\cdot)}_{p(\cdot),q(\cdot)}$ and $L^{p(\cdot)}(\ell^{q(\cdot)})$ by $\ell^{q(\cdot)}(L^{p(\cdot)})$.
\end{tw}

\begin{tw}\label{boundedbesov}
Let $(X,d,\mu)$ be a bounded metric measure space and assume that there exist $Q\in \mathcal{P}_b^{\log}(X)$, $\delta \in (0,\infty)$,  and $b\in (0,1]$ such that for every $x\in X$ and $r\in (0,\delta]$,
\begin{equation*}
	\mu(B(x,r)) \geq br^{Q(x)}.
\end{equation*}
Let $p\in \mathcal{P}_b^{\log}(X)$, $s\in \mathcal{P}_b(X)$ and $q\in \mathcal{P}(X)$. Moreover, let $\varepsilon \in \mathcal{P}_b^{\log}(X)$ be such that $\varepsilon \ll s$. Then, the following statements are valid.
\begin{enumerate}
	\item[(i)] If $\varepsilon p \ll Q$, then, there exists a positive constant $\check{C}_S$, depending only on $b$, $\delta$,  $\diam X$, $\mu(X)$, $p$, $s$, $Q$, $\varepsilon$, such that, for every pair of functions $u\in \dot{N}^{s(\cdot)}_{p(\cdot),q(\cdot)}(X,d,\mu)$ and $g\in \mathbb{D}^{s(\cdot)}(u)\cap \ell^{q(\cdot)}(L^{p(\cdot)}(X,\mu))$, there holds
	\begin{equation*}
		\inf_{c\in \mathbb R} \left\| u-c\right\|_{L^{\gamma_{\varepsilon}(\cdot)}(X,\mu)} \leq \check{C}_S \left\| g \right\|_{\ell^{q(\cdot)}(L^{p(\cdot)}(X,\mu))},
	\end{equation*}
	where $\displaystyle\gamma_{\varepsilon}:=\frac{Qp}{Q-\varepsilon p}$.
	
	\noindent Moreover, for every $u\in N^{s(\cdot)}_{p(\cdot),q(\cdot)}(X,d,\mu)$ and $g\in \mathbb{D}^{s(\cdot)}(u) \cap \ell^{q(\cdot)}(L^{p(\cdot)}(X,\mu))$
	\begin{equation*}
		\left\| u \right\|_{L^{\gamma_{\varepsilon}(\cdot)}(X,\mu)} \leq \check{C}_S\left( \left\| u \right\|_{L^{p(\cdot)}(X,\mu)} + \left\|g \right\|_{\ell^{q(\cdot)}(L^{p(\cdot)}(X,\mu))}\right).
	\end{equation*}
	\item[(ii)] If $\varepsilon p=Q$, then there exist positive constants $\check{C}_{MT1}$, $\check{C}_{MT2}$, depending only on $b$, $\delta$,  $\diam X$, $p$, $s$, $Q$, $\varepsilon$, such that for every pair of functions $u\in \dot{N}^{s(\cdot)}_{p(\cdot),q(\cdot)}(X,d,\mu)$ and $g\in \mathbb{D}^{s(\cdot)}(u) \cap \ell^{q(\cdot)}(L^{p(\cdot)}(X,\mu))$ with $\left\| g\right\|_{\ell^{q(\cdot)}(L^{p(\cdot)}(X,\mu))}>0$, there holds
	\begin{equation*}
		\fint_X \exp\left(\check{C}_{MT1} \frac{\left|u(x)-u_X\right|}{\left\| g \right\|_{\ell^{q(\cdot)}(L^{p(\cdot)}(X,\mu))}}\right) \mbox{d}\mu(x) \leq \check{C}_{MT2}.
	\end{equation*}
	\item[(iii)] If $\varepsilon p\gg Q$, then there exists a positive constant $\check{C}_H$, depending only on $b$, $\delta$, $\diam X$, $p$, $s$, $Q$, $\varepsilon$, such that for every pair of functions $u\in \dot{N}^{s(\cdot)}_{p(\cdot),q(\cdot)}(X,d,\mu)$ and $g\in \mathbb{D}^{s(\cdot)}(u)\cap \ell^{q(\cdot)}(L^{p(\cdot)}(X,\mu))$, there holds
	\begin{equation*}
		\left\| u-u_X \right\|_{L^{\infty}(X,\mu)} \leq \check{C}_H \left\|g\right\|_{\ell^{q(\cdot)}(L^{p(\cdot)}(X,\mu))}.
	\end{equation*}
	Moreover, the function $u$ has a continuous representative $\tilde{u}$ on $X$ such that, for all $x,y\in X$,
	\begin{equation*}
	\left| \tilde{u}(x)-\tilde{u}(y)\right| \leq \check{D}_H \left\| g\right\|_{\ell^{q(\cdot)}(L^{p(\cdot)}(X,\mu))} d(x,y)^{\alpha_{\varepsilon}(x)},
	\end{equation*}
	there $\check{D}_H$ is a positive constant and $\displaystyle \alpha_{\varepsilon}:=\varepsilon-\frac{Q}{p}$.
\end{enumerate}
\end{tw}

We now record global embeddings for the spaces $M^{s(\cdot)}_{p(\cdot),q(\cdot)}$ and
$N^{s(\cdot)}_{p(\cdot),q(\cdot)}$ in the cases $sp \ll Q$, $sp=Q$, and $sp \gg Q$; see Theorems~\ref{gsob}, \ref{gmoser}, and \ref{ghold}. Since all of the statements follow immediately from Theorem~\ref{doublingembedding} combined with Proposition~\ref{embeddingsbetween} $(v)$, $(vi)$, $(vii)$, respectively, we omit the details.

\begin{center}
	\textbf{A global embedding theorem in the case $sp \ll Q$}
\end{center}

\begin{tw}\label{gsob}
	Let $(X,d,\mu)$ be a geometrically doubling metric measure space such that there exists $\delta \in (0,\infty)$ satisfying
	\begin{equation*}
	\sup_{x\in X} \mu(B(x,\delta)) <\infty.
	\end{equation*} 
	Suppose that there exist $Q\in \mathcal{P}^{\log}_b(X)$ such that $\mu$ is lower Ahlfors $Q(\cdot)$-regular.
	Moreover, let $p,s\in \mathcal{P}^{\log}_b(X)$ be such that $sp \ll Q.$ Then,
	\begin{enumerate}
		\item[(i)] for every $q\in \mathcal{P}(X)$
		\begin{equation*}
			M^{s(\cdot)}_{p(\cdot),q(\cdot)}(X,d,\mu) \hookrightarrow L^{\gamma(\cdot)}(X,\mu),
		\end{equation*}
		where $\displaystyle \gamma:=\frac{Qp}{Q-sp}$,
		\item[(ii)] for every $q\in \mathcal{P}(X)$ such that $q \leq p$
		\begin{equation*}
			N^{s(\cdot)}_{p(\cdot),q(\cdot)}(X,d,\mu) \hookrightarrow L^{\gamma(\cdot)}(X,\mu),
		\end{equation*}
		where $\displaystyle \gamma:=\frac{Qp}{Q-sp}$,
		\item[(iii)] for every $q\in \mathcal{P}(X)$ and $t\in \mathcal{P}^{\log}_b(X)$ such that $t \ll s$
		\begin{equation*}
			N^{s(\cdot)}_{p(\cdot),q(\cdot)}(X,d,\mu) \hookrightarrow L^{\sigma(\cdot)}(X,\mu),
		\end{equation*}
		where $\displaystyle \sigma:=\frac{Qp}{Q-tp}$.
	\end{enumerate}
\end{tw}

\begin{center}
	\textbf{A global embedding theorem in the case $sp=Q$}
\end{center}

\begin{tw}\label{gmoser}
	Let $(X,d,\mu)$ be a geometrically doubling metric measure space such that there exists $\delta \in (0,\infty)$ satisfying
\begin{equation*}
	\sup_{x\in X} \mu(B(x,\delta)) <\infty.
\end{equation*} 
Suppose that there exist $Q\in \mathcal{P}^{\log}_b(X)$ such that $\mu$ is lower Ahlfors $Q(\cdot)$-regular.
Moreover, let $p,s\in \mathcal{P}^{\log}_b(X)$ be such that $sp = Q.$ Then,
\begin{enumerate}
	\item[(i)] for every $q\in \mathcal{P}(X)$
	\begin{equation*}
		M^{s(\cdot)}_{p(\cdot),q(\cdot)}(X,d,\mu) \hookrightarrow L^{\beta(\cdot)}(X,\mu),
	\end{equation*}
	for every $\beta \in \mathcal{P}_b(X)$ such that $\beta \geq p$,
	\item[(ii)] for every $q\in \mathcal{P}(X)$ such that $q \leq p$
	\begin{equation*}
		N^{s(\cdot)}_{p(\cdot),q(\cdot)}(X,d,\mu) \hookrightarrow L^{\beta(\cdot)}(X,\mu),
	\end{equation*}
		for every $\beta \in \mathcal{P}_b(X)$ such that $\beta \geq p$,
	\item[(iii)] for every $q\in \mathcal{P}(X)$ and $t\in \mathcal{P}^{\log}_b(X)$ such that $t \ll s$
	\begin{equation*}
		N^{s(\cdot)}_{p(\cdot),q(\cdot)}(X,d,\mu) \hookrightarrow L^{\sigma(\cdot)}(X,\mu),
	\end{equation*}
	where $\displaystyle \sigma:=\frac{Qp}{Q-tp}$.
\end{enumerate}
\end{tw}

\begin{center}
\textbf{A global embedding theorem in the case $sp \gg Q$}
\end{center}

\begin{tw}\label{ghold}
	Let $(X,d,\mu)$ be a metric measure space. Suppose that there exist $Q\in \mathcal{P}^{\log}_b(X)$ such that $\mu$ is lower Ahlfors $Q(\cdot)$-regular.
	Moreover, let $p,s\in \mathcal{P}^{\log}_b(X)$ be such that $sp \gg Q$. Then,
	\begin{enumerate}
		\item[(i)] for every $q\in \mathcal{P}(X)$
		\begin{equation*}
			M^{s(\cdot)}_{p(\cdot),q(\cdot)}(X,d,\mu) \hookrightarrow C^{0,\alpha(\cdot)}(X,d),
		\end{equation*}
		where $\displaystyle \alpha:=s-\frac{Q}{p}$,
		\item[(ii)] for every $q\in \mathcal{P}(X)$ such that $q \leq p$
		\begin{equation*}
			N^{s(\cdot)}_{p(\cdot),q(\cdot)}(X,d,\mu) \hookrightarrow C^{0,\alpha(\cdot)}(X,d),
		\end{equation*}
		where $\displaystyle \alpha:=s-\frac{Q}{p}$,
		\item[(iii)] for every $q\in \mathcal{P}(X)$ and $t\in \mathcal{P}^{\log}_b(X)$ such that $Q \ll tp \ll sp$
		\begin{equation*}
			N^{s(\cdot)}_{p(\cdot),q(\cdot)}(X,d,\mu) \hookrightarrow C^{0,\lambda(\cdot)}(X,d),
		\end{equation*}
		where $\displaystyle \lambda:=t-\frac{Q}{p}$.
	\end{enumerate}
\end{tw}

\section{Necessity of Geometric Conditions}\label{sect:necessity}

This final section examines the geometric conditions that are not only sufficient but necessary for the validity of the embeddings in Section~\ref{sect:mainresults}. We first introduce the concept of uniform perfectness for metric spaces and establish technical norm estimates for Lipschitz cut-off functions, which are vital in proving our necessity results. Finally, we demonstrate that the existence of Sobolev-type embeddings in this work forces a quantitative lower Ahlfors-regularity growth condition on the measure of order $r^{Q(x)}$, identifying the fundamental link between these functional inequalities and the geometry of the space.

\subsection{Uniform perfectness}
\begin{defi}\label{uniformly perfect}
	A metric space $(X,d)$ is said to be \texttt{uniformly perfect} if there exists a constant $\lambda\in (0,1)$ with the property that for each $x\in X$ one has
	\begin{equation}\label{uniformly}
	B(x,r)\setminus B(x,\lambda r)\neq\emptyset\quad\text{whenever}\quad X\setminus B(x, r)\neq\emptyset.
	\end{equation}
\end{defi}
It follows immediately from \eqref{uniformly} that, if \eqref{uniformly} holds for some $\lambda\in(0,1)$ then it also holds for any other $\lambda'\in(0,\lambda]$. As such, without loss of generality, we can always assume $\lambda\in(0,1/5)$.

\begin{prop}\label{fi}\cite{AGH20}
	Let $\left(X,d,\mu\right)$ be a metric measure space. For $x\in X$ and $r\geq 0$ we define
	\begin{equation*}
		\varphi_x(r):=\sup\left\{s\in[0,r]: \mu\left(B(x,s)\right)\leq \frac{1}{2}\mu\left(B(x,r)\right) \right\}.
	\end{equation*}
	Then, the following statements hold.
	\begin{enumerate}
		\item[(a)] The function $\varphi_x(\cdot)$ is non-decreasing, i.e., $\varphi_x(s)\leq \varphi_x(t)$ for $0\leq s\leq t<\infty$.
		\item[(b)] One has that
		\begin{equation*}
			\mu\left(B(x,\varphi_x(r))\right)\leq \frac{1}{2}\mu\left(B(x,r)\right)\leq \mu\left(\overline{B}(x,\varphi_x(r))\right).
		\end{equation*}
		\item[(c)] It holds that $\varphi_x(r)\in [0,r]$, where $\varphi_x(r)=r$ if and only if $r=0$.
		\item[(d)] If $\mu\left(\left\{x\right\}\right)=0$ and $r>0$, then $\varphi_x^j(r)>0$ for every $j\in \mathbb N_0$, where
		\begin{equation*}
			\varphi_x^0(r):=r \textnormal{ and } \varphi_x^j(r):=\varphi_x\left(\varphi_x^{j-1}(r)\right), j \in \mathbb N.
		\end{equation*}
		Moreover, the sequence $\left\{\varphi_x^{j}(r)\right\}_{j=0}^\infty$ is strictly decreasing, i.e.,
		\begin{equation*}
			r>\varphi_x(r)>\varphi_x^2(r)>\dots>\varphi_x^j(r)>\varphi_x^{j+1}(r)>\dots>0
		\end{equation*}
		and $\mu\left(B(x,\varphi_x^j(r))\right)\leq 2^{-j}\mu\left(B(x,r)\right)$. Consequently, $\displaystyle \lim_{j\to \infty}\mu\left(B(x,\varphi_x^j(r))\right)=0$.
	\end{enumerate}
\end{prop}

The following Lemma was proved in \cite{AGH20} the setting of constant dimension. Below, we shall sketch a proof for variable dimension.
\begin{lemma}\label{pomocniczy}
	Let $\left(X,d,\mu\right)$ be uniformly perfect metric measure space and let $\lambda\in \left(0,1/5\right)$ be as in \eqref{uniformly}. Let $Q\in \mathcal{P}_b^{\log}(X)$ and assume that there is a constant $C\in (0,\infty)$ such that $\mu\left(B(x,r)\right)\geq Cr^{Q(x)}$ whenever $x\in X$ and $r\in (0,\min\left\{1,\diam X\right\}]$ satisfy $\displaystyle r\leq 3\varphi_x(r)/\lambda^2$. Then $$\mu\left(B(x,r)\right)\geq C'r^{Q(x)}$$ for all $x \in X$ and every $r\in (0,1]$, where $C':=\min\{\mu\left(X\right), C\lambda^{Q^+} e^{-C_{\log}(Q)} 2^{- Q^+}\}$.
\end{lemma}

\begin{proof}
Let $x \in X$ and $r\in (0,1]$. 
Observe that if $\diam X<r\leq 1$ then  
$$
\infty>\mu\left(B(x,r)\right)=\mu\left(X\right)\geq \mu\left(X\right)r^{Q(x)}.
$$ 
For the case $r\leq \diam X$, we shall use the following lemma from \cite{AGH20}.

\begin{lemma}\label{pomocniczydrugi}
	Let $(X,d,\mu)$ be uniformly perfect metric measure space and let $\lambda \in (0,1/5)$ be as in \eqref{uniformly}. If $x\in X$ and $r\in \left(0,\diam X\right]$ is finite and $r>3\varphi_x(r)/\lambda^2,$ then there is a ball $B(\tilde{x},\tilde{r}) \subseteq B(x,r)$ such that $\lambda r < \tilde{r} \leq  \min\left\{r,3\varphi_{\tilde{x}}(\tilde{r})/\lambda^2\right\}.$
\end{lemma}

We fix a point $x\in X$ and radius $r\in \left(0,\min\left\{1,\diam X\right\}\right].$ If $r\leq 3\varphi_x(r)/\lambda^2$, then $\mu(B(x,r))\geq Cr^{Q(x)}\geq C'r^{Q(x)}$ by assumption. Thus, we will assume that $r>3\varphi_x(r)/\lambda^2.$ By the virtue of Lemma~\ref{pomocniczydrugi}, there is a ball $B(\tilde{x},\tilde{r}) \subseteq B(x,r)$ such that $\lambda r< \tilde{r}\leq \min\left\{r,3\varphi_{\tilde{x}}(\tilde{r})/\lambda^2\right\}.$ Then, since $\tilde{r}\leq r\leq 1$ we have
	\begin{equation*}
		\mu(B(x,r)) \geq \mu(B(\tilde{x},\tilde{r})) \geq C\tilde{r}^{Q(\tilde{x})}\geq C\lambda^{Q(\tilde{x})} r^{Q(\tilde{x})}\geq C\lambda^{Q^+} r^{Q_{B(x,r)}^+}.
	\end{equation*}
	Applying Lemma~\ref{loglemma} $(i)$ with $t:=1/Q$, $B:=B(x,r)$, $z:=x$, and $R:=2r$, we have
	\begin{equation*}
		e^{-C_{\log}(Q)} \left(\frac{1}{2r}\right)^{Q_{B(x,r)}^+} \leq \left(\frac{1}{2r}\right)^{Q(x)}.
	\end{equation*}
	Hence,
	\begin{equation*}
		\mu(B(x,r))\geq C\lambda^{Q^+} e^{-C_{\log}(Q)} 2^{- Q^+} r^{Q(x)}\geq C'r^{Q(x)}.
	\end{equation*}
\end{proof}

\subsection{Norm estimates for Lipschitz cut-off functions}
\begin{remark}
Given a metric measure space $(X,d,\mu)$ and numbers $\alpha\in (0,\infty)$ and $t\in (0,\infty]$, we shall write $\alpha \preccurlyeq_{t} 1$ if $\alpha \leq 1$ and that the value $\alpha=1$ is permissible if and only if $t=\infty$.
\end{remark}

\begin{tw}\label{lipschitz}
	Let $\left(X,d,\mu \right)$ be a metric measure space and $s, p\in \mathcal{P}_b(X)$, $q\in \mathcal{P}(X)$ with $s^+ \preccurlyeq_{q^-} 1$. Let $B\subseteq X$ be measurable, bounded set with positive measure and suppose that $u: X \to \mathbb [0,1]$ is an $L$-Lipschitz function which is zero outside of $B$, where $L\in(0,\infty)$. Then, the sequence $\left\{g_k\right\}_{k\in\mathbb Z}$ of functions, which are defined by
	\begin{equation*}
		g_k(x):=\left\{\begin{array}{lll} L2^{k\left(s(x)-1\right)}\chi_{B}(x),& \textnormal{ for } k\geq k_L,\\ 2^{\left(k+1\right)s(x)+1}\chi_{B}(x),&\textnormal{ for } k< k_L, \end{array} \right. 
	\end{equation*}
	where $k_L\in \mathbb Z$ is such that $2^{k_L-1}\leq L <2^{k_L}$, is a vector $s(\cdot)$-gradient of the function $u$. Moreover, there exists a constant $C_{\rm lip}>0$, depending only on $q^-$, $s^-$, $s^+$, such that
	\begin{align*}
		\left\| \left\{g_k\right\} \right\|_{L^{p(\cdot)}\left(\ell^{q(\cdot)} \left(X,\mu\right) \right)} &\leq C_{\rm lip} \max\left\{L^{s_B^-},L^{s_B^+}\right\}\left\|\chi_{B}\right\|_{L^{p(\cdot)}\left(X,\mu\right)},\\ \left\| \left\{g_k\right\}\right\|_{\ell^{q(\cdot)}\left(L^{p(\cdot)}\left(X,\mu\right)\right)} &\leq C_{\rm lip} \max\left\{L^{s_B^-},L^{s_B^+}\right\}\left\|\chi_{B}\right\|_{L^{p(\cdot)}\left(X,\mu\right)}.
	\end{align*}
	In particular, $u$ belongs to $\dot{M}^{s(\cdot)}_{p(\cdot),q(\cdot)}(X,d,\mu)$ and $\dot{N}^{s(\cdot)}_{p(\cdot),q(\cdot)}(X,d,\mu)$.
\end{tw}

\begin{proof}
We start by showing that $\left\{g_k\right\}_{k\in\mathbb Z}$ is a vector $s(\cdot)$-gradient of $u$. Let us fix $k\in \mathbb Z$ and $x,y\in X$ such that $2^{-k-1}\leq d(x,y)<2^{-k}$. If $x,y\notin B$, then there is nothing to prove. Without loss of generality we may assume that $x\in B$. Firstly, consider the case when $k\geq k_L$.  We have that
	\begin{align*}
		\left|u(x)-u(y)\right|&\leq Ld(x,y)=d(x,y)^{s(x)}Ld(x,y)^{1-s(x)}\\
				&\leq d(x,y)^{s(x)}L2^{k\left(s(x)-1\right)}\\
		& \leq d(x,y)^{s(x)}L2^{k\left(s(x)-1\right)}\chi_B(x)+d(x,y)^{s(y)}L2^{k\left(s(y)-1\right)}\chi_{B}(y).
	\end{align*}
	Now, assume that $k < k_L$. Then
	\begin{align*}
		\left|u(x)-u(y)\right|&\leq 2=2d(x,y)^{s(x)}d(x,y)^{-s(x)}\leq d(x,y)^{s(x)}2^{\left(k+1\right)s(x)+1} \\ 
		& \le d(x,y)^{s(x)}2^{\left(k+1\right)s(x)+1}\chi_{B}(x)+d(x,y)^{s(y)}2^{\left(k+1\right)s(y)+1}\chi_B(y).
	\end{align*}
	Therefore, we have proved that $\left\{g_k\right\}_{k\in\mathbb Z}$ is vector $s(\cdot)$-gradient of $u$. 
	
	Now we shall prove that
	\begin{equation*}\left\| \left\{g_k\right\}_{k\in\mathbb Z} \right\|_{L^{p(\cdot)}\left(\ell^{q(\cdot)} \left(X,\mu\right) \right)} \leq C_{\lip} \max\left\{L^{s_B^-},L^{s_B^+}\right\}\left\|\chi_{B}\right\|_{L^{p(\cdot)}\left(X,\mu\right)}.
	\end{equation*}
	By the virtue of Proposition~\ref{monotonicity} $(i)$, it suffices to prove that 
	\begin{equation*}\left\| \left\{g_k\right\}_{k\in\mathbb Z} \right\|_{L^{p(\cdot)}\left(\ell^{q^-} \left(X,\mu\right) \right)} \leq C_{\lip} \max\left\{L^{s_B^-},L^{s_B^+}\right\}\left\|\chi_{B}\right\|_{L^{p(\cdot)}\left(X,\mu\right)}.
	\end{equation*}
	If $q^-<\infty$, then for all $x\in B$, we have following string of estimates
	\begin{align*}
		\left\|\left\{g_k(x)\right\}_{k\in\mathbb Z}\right\|_{\ell^{q^-}}^{q^-}&=\sum_{k=-\infty}^{k_L-1} 2^{\left(k+1\right)q^-s(x)+q^-}+L^{q^-}\sum_{k=k_L}^{\infty}2^{kq^-\left(s(x)-1\right)}\leq  2^{q^-} \frac{2^{k_L q^- s(x)}}{1-2^{-q^-s^-}}+ L^{q^-}\frac{2^{k_Lq^-\left(s(x)-1\right)}}{1-2^{-q^-\left(1-s^+\right)}}\\
		& \leq 4^{q^-} \frac{L^{q^- s(x)}}{1-2^{-q^- s^-}} + L^{q^-}\frac{L^{q^-\left(s(x)-1\right)}}{1-2^{-q^-\left(1-s^+\right)}} = L^{q^-s(x)} A_1^{q^-}\\ & \leq \max\left\{L^{q^- s_B^+}, L^{q^- s_B^-}\right\} A_1^{q^-},
	\end{align*}
	where 
	\begin{equation*}
		A_1:=\left(\frac{4^{q^-}}{1-2^{-q^-s^-}}+\frac{1}{1-2^{-q^-\left(1-s^+\right)}}\right)^{\frac{1}{q^-}}.
	\end{equation*}
	Finally, for $x\in X$
	\begin{equation}\label{skon1}
		\left\|\left\{g_k(x)\right\}_{k\in\mathbb Z}\right\|_{\ell^{q^-}}\leq A_1 \max\left\{L^{s_B^+},L^{s_B^-}\right\} \chi_B(x).
	\end{equation}
	Now, let consider the case $q^-=\infty$. For all $k \geq k_L$ and $x\in B$, we have 
	\begin{equation*}
		L2^{k\left(s(x)-1\right)} \leq L2^{k_L\left(s(x)-1\right)}\leq L\cdot L^{s(x)-1}=L^{s(x)} \leq \max\left\{L^{s_B^+}, L^{s_B^-}\right\}.
	\end{equation*}
	Additionally, for $k <k_L $ and $x\in B$, we have 
	\begin{align*}
	2^{(k+1)s(x)+1} \leq 2^{k_Ls(x)+1} \leq 4L^{s(x)}\leq 4 \max\left\{L^{s_B^+},L^{s_B^-}\right\}.
	\end{align*}
	Therefore, for all $x\in X$,
	\begin{equation}\label{niesk1}
		\left\|\left\{g_k(x)\right\}_{k\in\mathbb Z}\right\|_{\ell^{\infty}}\leq 4 \max\left\{L^{s_B^+},L^{s_B^-}\right\}\chi_B(x).
	\end{equation}
	Taking $L^{p(\cdot)}$-quasi-norm in \eqref{skon1}, \eqref{niesk1} we obtain the first part of the theorem.
	
	Now we shall prove that the second inequality holds and $u\in \dot{N}^{s(\cdot)}_{p(\cdot),q(\cdot)}(X,d,\mu)$. By the virtue of Proposition~\ref{monotonicity} suffices to prove that
	\begin{equation*}
		\left\|g\right\|_{\ell^{q^-}\left(L^{p(\cdot)}(X,\mu)\right)}\leq C_{\lip}\max\left\{L^{s_B^+},L^{s_B^-}\right\}\left\|\chi_{B}\right\|_{L^{p(\cdot)}(X,\mu)}.
	\end{equation*}
	Let us recall that from Lemma~\ref{mieszane} $(i)$ we have
	\begin{equation*}
		\left\|g\right\|_{\ell^{q^-}\left(L^{p(\cdot)}(X,\mu)\right)}=\left\{\begin{array}{ll} \displaystyle\left(\sum_{k\in \mathbb Z} \left\|g_k\right\|_{L^{p(\cdot)}(X,\mu)}^{q^-}\right)^{\frac{1}{q^-}}& \textnormal{ if } q^- < \infty,\\ \displaystyle\sup_{k\in \mathbb Z} \left\|g_k\right\|_{L^{p(\cdot)}(X,\mu)}& \textnormal{ if } q^- = \infty.
		\end{array}\right.
	\end{equation*}
	Hence, if $q^-<\infty$ we obtain
	\begin{align}\label{n3}
		\left\|g\right\|_{\ell^{q^-}\left(L^{p(\cdot)}(X,\mu)\right)} &\leq 2^{\frac{1}{q^-}} \left[\left(\sum_{k=-\infty}^{k_L-1} \left\|2^{(k+1)s +1}\chi_{B}\right\|_{L^{p(\cdot)}(X,\mu)}^{q^-}\right)^{\frac{1}{q^-}}\right.\nonumber\\
		&\qquad\qquad\qquad\left.+\left(\sum_{k=k_L}^{\infty} \left\|L\cdot 2^{k\left(s-1\right)}\chi_{B}\right\|_{L^{p(\cdot)}(X,\mu)}^{q^-}\right)^{\frac{1}{q^-}}\right].
	\end{align} 
	Observe that,
	\begin{align*}
		\sum_{k=-\infty}^{k_L-1} \left\|2^{(k+1)s +1}\chi_{B}\right\|_{L^{p(\cdot)}(X,\mu)}^{q^-}&\leq 2^{q^-}\left\|\chi_B\right\|_{L^{p(\cdot)}(X,\mu)}^{q^-}\sum_{k=-\infty}^{k_L-1} \max\left\{2^{(k+1)s_B^-},2^{(k+1)s_B^+}\right\}^{q^-} \\ & \leq 2^{q^-}\left\|\chi_B\right\|_{L^{p(\cdot)}(X,\mu)}^{q^-}\left(\sum_{k=-\infty}^{k_L-1} 2^{(k+1)s_B^-q^-} +\sum_{k=-\infty}^{k_L-1} 2^{(k+1)s_B^+q^-}\right) \\ & = 2^{q^-} \left\|\chi_B \right\|_{L^{p(\cdot)}(X,\mu)}^{q^-}\left(\frac{2^{k_Ls_B^-q^-}}{1-2^{-s_B^-q^-}} +\frac{ 2^{k_Ls_B^+q^-}}{1-2^{-s_B^+q^-}}\right)\\ & \leq \frac{2^{q^- +1}}{1-2^{-s^-q^-}}\left\|\chi_B \right\|_{L^{p(\cdot)}(X,\mu)}^{q^-}\max\left\{\left(2L\right)^{q^- s_B^+}, \left(2L\right)^{q^- s_B^-}\right\} \\ & \leq \frac{2^{2q^- +1}}{1-2^{-s^-q^-}} \left\|\chi_B \right\|_{L^{p(\cdot)}(X,\mu)}^{q^-}\max\left\{L^{s_B^+},L^{s_B^-}\right\}^{q^-}
	\end{align*}
	and
	\begin{align*}
		\sum_{k=k_L}^{\infty} \left\| L 2^{k(s-1)}\chi_B \right\|^{q^-}_{L^{p(\cdot)}(X,\mu)}&\leq \left\| \chi_B\right\|_{L^{p(\cdot)}(X,\mu)}^{q^-} L^{q^-} \sum_{k=k_L}^{\infty} \max\left\{2^{k\left(s_B^+ -1\right)}, 2^{k\left(s_B^- - 1\right)}\right\}^{q^-} \\ & \leq \left\| \chi_B\right\|_{L^{p(\cdot)}(X,\mu)}^{q^-} L^{q^-}\left(\sum_{k=k_L}^{\infty} 2^{k\left(s_B^+ -1 \right)q^-} +\sum_{k=k_L}^{\infty} 2^{k\left(s_B^- - 1\right)q^-} \right) \\ & = \left\| \chi_B\right\|_{L^{p(\cdot)}(X,\mu)}^{q^-} L^{q^-}\left(\frac{2^{k_L\left(s_B^+ - 1\right)q^-}}{1-2^{\left(s_B^+ -1\right)q^-}} + \frac{2^{k_L\left(s_B^- - 1\right)q^-}}{1-2^{\left(s_B^- - 1\right)q^-}}\right) \\ & \leq \frac{1}{1-2^{\left(s^+ - 1\right)q^-}} \left\| \chi_B\right\|_{L^{p(\cdot)}(X,\mu)}^{q^-} L^{q^-}\left(L^{\left(s_B^+ - 1\right)q^-} + L^{\left(s_B^- - 1\right)q^-}\right) \\ & \leq \frac{2}{1-2^{\left(s^+ -1\right)q^-}} \left\| \chi_B\right\|_{L^{p(\cdot)}(X,\mu)}^{q^-} \max\left\{L^{s_B^+},L^{s_B^-}\right\}^{q^-}.
	\end{align*}
	Combining these estimates with \eqref{n3} gives
	\begin{align*}
		\left\| g\right\|_{\ell^{q^-}\left(L^{p(\cdot)}\left(X,\mu\right)\right)} &\leq A_2\max\left\{L^{s_{B}^-},L^{s_{B}^+}\right\} \left\|\chi_B \right\|_{L^{p(\cdot)}(X,\mu)},
	\end{align*}
	where
	\begin{equation*}
		A_2:=2^{\frac{1}{q^-}} \left[\left(\frac{2^{2q^- +1}}{1-2^{-s^-q^-}}\right)^{\frac{1}{q^-}} + \left(\frac{2}{1-2^{-q^-\left(1-s^+\right)}}\right)^{\frac{1}{q^-}}\right].
	\end{equation*}
	Finally,
	\begin{equation*}
\left\|g\right\|_{\ell^{q(\cdot)}\left(L^{p(\cdot)}(X,\mu)\right)} \leq \left\|g\right\|_{\ell^{q^-}\left(L^{p(\cdot)}(X)\right)}\leq A_2\max\left\{L^{s^-_{B}},L^{s^+_{B}}\right\}\left\|\chi_{B}\right\|_{L^{p(\cdot)}(X,\mu)},
	\end{equation*}
	and the proof in this case is complete.
	
	If $q^-=\infty$, then
	\begin{align*}
	\sup_{k\in \mathbb Z} \left\| g_k \right\|_{L^{p(\cdot)}(X,\mu)}&\leq \sup_{\substack{k \geq k_L \\ k \in \mathbb Z}} \left\|g_k \right\|_{L^{p(\cdot)}(X,\mu)} + \sup_{\substack{k < k_L \\ k \in \mathbb Z}} \left\|g_k \right\|_{L^{p(\cdot)}(X,\mu)} \\&=\left\| L2^{k_L\left(s-1\right)}\chi_B\right\|_{L^{p(\cdot)}(X,\mu)}+\left\| 2^{k_Ls+1}\chi_B \right\|_{L^{p(\cdot)}(X,\mu)} \\ & \leq \left\|L^{s} \chi_B \right\|_{L^{p(\cdot)}(X,\mu)} +4\left\|L^{s}\chi_B\right\|_{L^{p(\cdot)}(X,\mu)} = 5\left\| L^{s} \chi_B \right\|_{L^{p(\cdot)}(X,\mu)} \\ & \leq 5\max\left\{L^{s_B^+}, L^{s_B^-}\right\} \left\| \chi_B \right\|_{L^{p(\cdot)}(X,\mu)}
	\end{align*}
	and the proof of Theorem~\ref{lipschitz} is finished.
\end{proof}

\subsection{Necessity of the lower measure bound assumption}
We now prove the following theorem.
	
\begin{tw}\label{necessaryembedding}
Let $(X,d,\mu)$ be metric measure space, $\sigma \in [1,\infty)$, and suppose that $p,s\in \mathcal{P}_b^{\log}(X)$ and $q\in \mathcal{P}(X)$ satisfy $s^+ \preceq_{q^-} 1.$ Moreover, fix $A\in \left\{M,N\right\}$. Then, the following statements are valid.
\begin{enumerate}
\item[(i)] Suppose that there exists $\gamma \in \mathcal{P}^{\log}_b(X)$ such that $\gamma \gg p$ and
\begin{equation*}
	A^{s(\cdot)}_{p(\cdot),q(\cdot)}(X,d,\mu) \hookrightarrow L^{\gamma(\cdot)}(X,\mu).
\end{equation*}
Then, the measure $\mu$ is lower Ahlfors $Q(\cdot)$-regular, where $\displaystyle Q:=\frac{\gamma sp}{\gamma-p}$.
\end{enumerate}
If the space $(X,d)$ is uniformly perfect\footnote{see Definition~\ref{uniformly perfect}.}, then the following statements also hold true:
\begin{enumerate}
\item[(ii)] Suppose that there exist $\gamma \in \mathcal{P}_b^{\log}(X)$ with $\gamma \gg p$ and positive constants $C_S$, and $\omega$ such that, for every ball $B_0:=B(x_0,r_0)\subseteq X$ with $r_0 \leq 1/\sigma$ and every $u \in \dot{A}^{s(\cdot)}_{p(\cdot),q(\cdot)}(\sigma B_0)$, there holds 
\begin{equation}\label{sobnecessary}
\inf_{c\in \mathbb R} \left\| u -c \right\|_{L^{\gamma(\cdot)}(B_0)} \leq C_S \left( \frac{\mu(B_0)}{r_0^{Q(x_0)}}\right)^{\omega} \left\|u \right\|_{\dot{A}^{s(\cdot)}_{p(\cdot),q(\cdot)}(\sigma B_0)},
\end{equation}
where $\displaystyle Q:=\frac{sp\gamma}{\gamma-p}.$ Then, the measure $\mu$ is lower Ahlfors $Q(\cdot)$-regular.
\item[(iii)] Suppose that there exist positive constants $C_{MT1}$, $C_{MT2}$, and $\omega$ such that, for every ball $B_0:=B(x_0,r_0)\subseteq X$ with $r_0 \leq 1/\sigma$ and every $u\in \dot{A}^{s(\cdot)}_{p(\cdot),q(\cdot)}(\sigma B_0)$ with $\left\|u \right\|_{\dot{A}^{s(\cdot)}_{p(\cdot),q(\cdot)}(\sigma B_0)}>0$, there holds
\begin{equation}\label{mosnecessary}
	\inf_{c\in \mathbb R} \fint_{B_0} \exp\left[\left(C_{MT1} \frac{\left|u(x)-c\right|}{\left\|u \right\|_{\dot{A}^{s(\cdot)}_{p(\cdot),q(\cdot)}(\sigma B_0)}}\right)^{\omega}\right] \mbox{d}\mu(x) \leq C_{MT2}.
\end{equation}
Then, the measure $\mu$ is lower Ahlfors $Q(\cdot)$-regular, where $Q:=sp.$
\item[(iv)] Suppose that there exists $\alpha\in \mathcal{P}^{\log}_b(X)$ such that $s \geq \alpha$ and
\begin{equation*}
	A^{s(\cdot)}_{p(\cdot),q(\cdot)}(X,d,\mu) \hookrightarrow \dot{C}^{0,\alpha(\cdot)}(X,d).
\end{equation*}
Then, the measure $\mu$ is lower Ahlfors $Q(\cdot)$-regular, where $Q:=p\left(s-\alpha\right)$.  Moreover, if there exists a point $x_0 \in X$ such that $s(x_0)=\alpha(x_0)$, then $\mu\left(\left\{x_0\right\}\right)>0$.	
\end{enumerate}
\end{tw}
\begin{proof}
We begin with the proof of the statement $(i).$ 	Fix $x\in X$ and $r\in (0,1]$. We shall denote $B(x,r)$ by $B_r$. It is enough to consider the case when $\mu\left(B_r\right)\leq 1$ since, otherwise, we easily have $\mu\left(B_r\right)>1\geq r^{Q(x)}$. Moreover, we can assume that $r\in (0,r')$ for some $r'<\frac{1}{4}$, which will be fixed later (and it will depend on case).

For each fixed $j\in \mathbb N$, set $r_j:=\left(2^{-j-1}+2^{-1}\right)r$ and $B_j:=B(x,r_j)$. Note that, for all $j\in \mathbb N$,
\begin{equation}\label{nier}
	\frac{1}{2}r<r_{j+1}<r_j\leq \frac{3}{4}r.
\end{equation}
Now, for each $j\in \mathbb N$, we define $u_j:X \to \mathbb [0,1]$ by setting for each $y\in X$,
\begin{equation*}
	u_j(y):=\left\{\begin{array}{ll} 1,& \textnormal{ if } y\in B_{j+1},\\ \displaystyle \frac{r_j-d(x,y)}{r_j-r_{j+1}},& \textnormal{ if } y\in B_j \setminus B_{j+1},\\ 0,& \textnormal{ if } y\in X \setminus B_j.\end{array} \right.
\end{equation*}
It is easy to see that $u_j$ is a $\left(r_j-r_{j+1}\right)^{-1}$-Lipschitz function on $X$. Hence, by Theorem~\ref{lipschitz}, applied with the function $u_j$, we have that
\begin{align*}
	\left\|u_j\right\|_{A^{s(\cdot)}_{p(\cdot),q(\cdot)}(X)}&\leq C_{\lip}\left(r_j-r_{j+1}\right)^{-s_{B_j}^+}\left\|\chi_{B_j}\right\|_{L^{p(\cdot)}(X)}+\left\|\chi_{B_j}\right\|_{L^{p(\cdot)}(X)}\leq \left(C_{\lip}+1\right)\frac{2^{s_{B_r}^+(j+2)}}{r^{s_{B_r}^+}}\left\|\chi_{B_j}\right\|_{L^{p(\cdot)}(B_r)}\\& \leq \left(C_{\lip}+1\right) \frac{ 2^{s_{B_r}^+(j+2)}}{r^{s_{B_r}^+}}\left[\mu\left( B_j \right)\right]^{ \frac{1}{p_{B_r}^+} }, 
\end{align*}
where we also have used Proposition~\ref{rel} and the fact that $\left(r_j-r_{j+1}\right)^{-1}=2^{j+2}r^{-1}$.

On the other hand, by Proposition~\ref{rel}, we have
\begin{equation*}
	\left\|u_j\right\|_{L^{\gamma(\cdot)}(X)}\geq \left\|\chi_{B_{j+1}}\right\|_{L^{\gamma(\cdot)}(B_r)}\geq \left[\mu(B_{j+1})\right]^{\frac{1}{\gamma_{B_r}^-}}.
\end{equation*}
The assumption that $A^{s(\cdot)}_{p(\cdot),q(\cdot)}(X,d,\mu) \hookrightarrow L^{\gamma(\cdot)}(X,\mu)$ means that there exists a constant $C_{S,G}>0$ such that, for every $A^{s(\cdot)}_{p(\cdot),q(\cdot)}(X,d,\mu)$,
\begin{equation*}
\left\|u\right\|_{L^{\gamma(\cdot)}(X)}\leq C_{S,G} \left\|u\right\|_{A^{s(\cdot)}_{p(\cdot),q(\cdot)}(X)}.
\end{equation*}
Using the above estimates for $u=u_j$ we obtain
\begin{align*}
	\left[\mu\left(B_{j+1}\right)\right]^{\frac{1}{\gamma^-_{B_r}}} 
	&\leq \left(C_{\lip}+1\right)C_{S,G}\frac{2^{s_{B_r}^+(j+2)}}{r^{s_{B_r}^+}} \left[\mu\left(B_j\right)\right]^{\frac{1}{p_{B_r}^+}}\\
	&\leq \left(C_{\lip}+1\right)C_{S,G} \frac{2^{s_{B_r}^+(j+2)}}{r^{s_{B_r}^+}} \left[\mu(B_j)\right]^{\frac{1}{\gamma_{B_r}^+}+\frac{s_{B_r}^-}{Q_{B_r}^+}},
\end{align*}
where we have  also used the fact that $\displaystyle \frac{1}{\gamma^+_{B_r}}+\frac{s^-_{B_r}}{Q_{B_r}^+}\leq \frac{1}{p^+_{B_r}}$ and $\mu(B_r) \leq 1$.

Let
\begin{equation}\label{kwadracik}
	\displaystyle r':=\frac{1}{2}\min\left\{\frac{1}{4},\frac{1}{2}\exp\left(-\frac{C_{\textnormal{log}}(1/\gamma)Q^+}{s^-}\right)\right\} \hspace{5mm} \textnormal{and} \hspace{5mm} \eta:= \frac{s^-}{Q^+}-\frac{C_{\log}(1/\gamma)}{\log\left(\frac{1}{2r'}\right)}.
\end{equation}
By the definition of $r'$ we get that $\eta\in(0,\infty)$. We define
\begin{equation}\label{kwadracik1}
	\eta_r:=\frac{s_{B_r}^-}{Q_{B_r}^+}+\frac{1}{\gamma_{B_r}^+}-\frac{1}{\gamma_{B_r}^-}.
\end{equation}
Then, since $B_r \subseteq B_{r'}$, we obtain
\begin{equation*}
	\eta_r \geq \frac{s^-}{Q^+}+\frac{1}{\gamma_{B_{r'}}^+} -\frac{1}{\gamma_{B_{r'}}^-}.
\end{equation*}
Moreover, using Lemma~\ref{loglemma} $(ii)$ with $r:=r'$ and $t:=\gamma$, we get
\begin{equation*}
	\frac{1}{\gamma_{B_{r'}}^+}-\frac{1}{\gamma_{B_{r'}}^-} \geq -\frac{C_{\log}(1/\gamma)}{\log\left(\frac{1}{2r'}\right)}.
\end{equation*}
Therefore, $\eta_r \geq \eta>0$.

Since we have $\displaystyle \frac{1}{\gamma^-_{B_r}}<\frac{s^-_{B_r}}{Q_{B_r}^+}+\frac{1}{\gamma^+_{B_r}}$ and also $\mu\left(B(x,r/2)\right)\leq \mu\left(B_j\right)\leq \mu\left(B(x,r)\right)$, we can apply Lemma~\ref{iterative lemma} with
\begin{equation*}
	a_j:=\mu\left(B_j\right),\hspace{3mm} \frac{1}{p}:=\frac{s^-_{B_r}}{Q_{B_r}^+}+\frac{1}{\gamma^+_{B_r}},\hspace{2mm} \frac{1}{q}:=\frac{1}{\gamma^-_{B_r}},\hspace{2mm} \rho:=C_{S,G}\left(C_{\lip}+1\right)\frac{4^{s_{B_r}^+}}{r^{s_{B_r}^+}}, \hspace{2mm} \textnormal{and} \hspace{2mm}\tau:=2^{s_{B_r}^+}
\end{equation*}
to conclude that 
\begin{equation*}
	\left[\mu\left(B_1\right)\right]^{1-\frac{1}{t_r}}\rho^{\frac{\gamma^-_{B_r}}{t_r}}\tau^{\frac{1}{\eta_r}}\geq 1, 
\end{equation*}
where we denote $t_r:=\gamma^-_{B_r}\left(\frac{s^-_{B_r}}{Q_{B_r}^+}+\frac{1}{\gamma^+_{B_r}}\right)=1+\eta_r\gamma^-_{B_r}$. Hence,

\begin{equation}\label{rown2}
	\mu\left(B_r\right)\left[C_{S,G}\left(C_{\lip}+1\right)\frac{4^{s^+}}{r^{s_{B_r}^+}}\right]^{\frac{\gamma^-_{B_r}}{t_r-1}}2^{s^+\frac{t_r\gamma^-_{B_r}}{\left(t_r-1\right)^2}}\geq 1.
\end{equation}
Now, by using $\eta_r\geq \eta$ and $\gamma^-_{B_r}\geq \gamma^-$ we get
\begin{equation}\label{fds-45}
	\frac{\gamma^-_{B_r}}{t_r-1}=\frac{1}{\eta_r}\leq \frac{1}{\eta} \hspace{2mm} \textnormal{ and } \hspace{2mm} \frac{t_r\gamma^-_{B_r}}{\left(t_r-1\right)^2}= \frac{1}{\gamma^-_{B_r}\eta_r^2}+\frac{1}{\eta_r}\leq \frac{1}{\gamma^- \eta^2}+\frac{1}{\eta}.
\end{equation}
Therefore, inequality \eqref{rown2} yields
\begin{equation*}
	r^{\frac{s_{B_r}^+}{\eta_r}} \leq \max\left\{1,C_{S,G}\left(C_{\lip}+1\right)4^{s^+}\right\}^{\frac{1}{\eta}}2^{\frac{s^+}{\gamma^-\eta^2}+\frac{s^+}{\eta}}\mu\left(B(x,r)\right).
\end{equation*}
Set $c_1^{-1}:=\max\left\{1,C_{S,G}\left(C_{\lip}+1\right)4^{s^+}\right\}^{\frac{1}{\eta}}2^{\frac{s^+}{\gamma^- \eta^2}+\frac{s^+}{\eta}}$ and $\displaystyle\beta_r:=s_{B_r}^+-Q_{B_r}^+\eta_r$ to obtain
\begin{equation*}
	\mu\left(B(x,r)\right)\geq c_1 r^{Q_{B_r}^+} r^{\frac{\beta_r}{\eta_r}}.
\end{equation*}
Using Lemma~\ref{loglemma} $(ii)$ twice with $t:=1/\gamma$ and $t:=1/s$, we obtain that 
\begin{align*}
	r^{\gamma^+_{B_r}-\gamma^-_{B_r}}&\geq e^{-C_{\log}(\gamma)}2^{-\gamma_{B_r}^+} \geq e^{-C_{\log}(\gamma)}2^{-\gamma^+},\\
	r^{s_{B_r}^+-s_{B_r}^-}&\geq e^{-C_{\log}(s)}2^{-s_{B_r}^+} \geq e^{-C_{\log}(s)}2^{-s^+}.
\end{align*}
Then, keeping in mind that $r<1$, we have
\begin{equation}\label{kwadracik3}
	r^{\frac{\beta_r}{\eta_r}}=r^{\frac{s_{B_r}^+ - s_{B_r}^-}{\eta_r}}r^{\frac{Q_{B_r}^+}{\eta_r}\left(\frac{1}{\gamma_{B_r}^-} -\frac{1}{\gamma_{B_r}^+}\right)}\geq \left(e^{-C_{\log}(s)}2^{-s^+}\right)^{\frac{1}{\eta}} \left(e^{-C_{\log}(\gamma)}2^{-\gamma^+}\right)^{\frac{Q^+}{\eta \left(\gamma^-\right)^2}}=:c_2.
\end{equation}
Therefore,
\begin{equation*}
	\mu(B(x,r))\geq c_1c_2r^{Q_{B_r}^+}.
\end{equation*}
Finally, applying Lemma~\ref{loglemma} $(i)$ with $R:=2r$, $t:=1/Q$, and $z:=x$, we get
\begin{equation}\label{kwadracik4}
	e^{C_{\log}\left(Q\right)} \left(2r\right)^{Q_{B_r}^+} \geq \left(2r\right)^{Q(x)}.
\end{equation}
This easily yields that $\mu(B(x,r))\geq br^{Q(x)}$ with $b:=c_1c_2e^{-C_{\log}\left(Q\right)}2^{Q^- - Q^+}$, and the proof of statement $(i)$ is complete.

We shall now prove $(ii)$ and $(iii).$ Fix any $x\in X$ and $r\in(0,1]$ and set $B_r:=B(x,r)$. It is enough to consider the case when $\mu\left(B_r\right)\leq 1$ and $r<\diam X$. Moreover, we can assume that $r\in(0,r')$ for some fixed $r'<\min\left\{\frac{1}{4},\frac{1}{\sigma} \right\}.$

Let $\varphi_x$ be as in  Proposition~\ref{fi}. Since $X$ is uniformly perfect, due to Lemma~\ref{pomocniczy} it suffices to consider the case when $r\leq 3\varphi_x(r)/\lambda^2$, where $\lambda \in \left(0,1/5\right)$ is taken from the definition of uniform perfectness of $X$ (see Definition~\ref{uniformly perfect}). By the virtue of Proposition~\ref{fi} and the assumption that $r \leq 3\varphi_x(r)/\lambda^2$, it follows that $0<\varphi_x(r)<r$.

For each fixed $j\in \mathbb N$, set $\tilde{r}_j:=\left(2^{-j-1}+2^{-1}\right)\varphi_x(r)$ and $\tilde{B}_j:=B(x,\tilde{r}_j)$. Note that, for all $j\in \mathbb N$,
\begin{equation}\label{nierr1}
	\frac{1}{2}\varphi_x(r)<\tilde{r}_{j+1}<\tilde{r}_j \leq \frac{3}{4}\varphi_x(r).
\end{equation}
Now, for each $j\in \mathbb N$, we define $\tilde{u}_j: \sigma B_r\to \mathbb R$ by setting for each $y\in \sigma B_r$,
\begin{equation*}
	\tilde{u}_j(y):=\left\{\begin{array}{ll} 1,& \textnormal{ if } y\in \tilde{B}_{j+1},\\ \displaystyle\frac{\tilde{r}_j-d(x,y)}{\tilde{r}_j-\tilde{r}_{j+1}},& \textnormal{ if } y\in \tilde{B}_j \setminus \tilde{B}_{j+1},\\ 0,& \textnormal{ if } y\in \sigma B_r \setminus \tilde{B}_j.\end{array} \right.
\end{equation*}
It is easy to see that $\tilde{u}_j$ is a $\left(\tilde{r}_j-\tilde{r}_{j+1}\right)^{-1}$-Lipschitz function on $\sigma B_r$. Hence, by Theorem~\ref{lipschitz}, we have that
\begin{align*}
	\left\|\tilde{u}_j\right\|_{\dot{A}^{s(\cdot)}_{p(\cdot),q(\cdot)}(\sigma B_r)}&\leq C_{\lip}\left(\tilde{r}_j-\tilde{r}_{j+1}\right)^{-s_{\tilde{B_j}}^+}\left\|\chi_{\tilde{B}_j}\right\|_{L^{p(\cdot)}(\sigma B_r)}=C_{\lip}\frac{2^{s_{\tilde{B_j}}^+(j+2)}}{\left[\varphi_x(r)\right]^{s_{\tilde{B_j}}^+}}\left\|\chi_{\tilde{B}_j}\right\|_{L^{p(\cdot)}(\sigma B_r)}\\ &\leq C_{\lip}\frac{2^{s_{B_r}^+(j+2)}}{\left[\varphi_x(r)\right]^{s_{B_r}^+}}\left\|\chi_{\tilde{B}_j}\right\|_{L^{p(\cdot)}(\sigma B_r)},
\end{align*}
where we also used the fact that $\left(\tilde{r}_j-\tilde{r}_{j+1}\right)^{-1}=2^{j+2}\varphi_{x}(r)^{-1}$.

Now, observe that $\tilde{u}_j \equiv 1 $ on $\tilde{B}_{j+1}$ and $\tilde{u}_j \equiv 0$ on $ B_r\setminus \tilde{B}_j$. It follows that for each $c \in \mathbb R$,  $\left|\tilde{u}_j -c \right|\geq \frac{1}{2}$ pointwise on at least one of sets $\tilde{B}_{j+1}$ or $B_r\setminus \tilde{B}_j$. Now, by \eqref{nierr1} and Proposition~\ref{fi}, we have that
\begin{equation*}
	\mu\left(\tilde{B}_{j+1}\right) \leq \mu\left(B(x,\varphi_x(r))\right)\leq \frac{1}{2}\mu\left(B_r\right).
\end{equation*}
On the other hand
\begin{equation*}
	\mu\left(B_r \setminus \tilde{B}_j\right) = \mu\left(B_r\right) - \mu\left(\tilde{B}_j\right) \geq \mu\left(B_r\right) -\mu\left(B(x,\varphi_x(r))\right)\geq \frac{1}{2}\mu\left(B_r\right).
\end{equation*}
Therefore $\min\left\{\mu\left(\tilde{B}_{j+1}\right),\mu\left(B_r\setminus \tilde{B}_j\right)\right\}=\mu\left(\tilde{B}_{j+1}\right)$ and hence
\begin{equation*}
	\left|\tilde{u}_j -c\right|\geq \frac{1}{2} \textnormal{ on a subset of } B_r \textnormal{ having measure at least } \mu\left(\tilde{B}_{j+1}\right).
\end{equation*}
Denote this set by $E_j$. Since $\mu(E_j)\leq\mu(B_r)\leq1$, by Proposition~\ref{rel}, we have
\begin{align*}
	\inf_{c \in \mathbb R} \left\|\tilde{u}_j-c\right\|_{L^{\gamma(\cdot)}(B_r)}&\geq \inf_{c\in \mathbb R} \left\|\left(\tilde{u}_j-c\right)\chi_{E_j}\right\|_{L^{\gamma(\cdot)}(B_r)}\\
	&\geq  \frac{1}{2}\inf_{c\in \mathbb R}\left\|\chi_{E_j}\right\|_{L^{\gamma(\cdot)}(B_r)}\\
		&\geq \frac{1}{2}\inf_{c\in \mathbb R}\left[\mu\left(E_j\right)\right]^{\frac{1}{\gamma^-_{B_r}}}\geq \frac{1}{2}\left[\mu\left(\tilde{B}_{j+1}\right)\right]^{\frac{1}{\gamma^-_{B_r}}}.
\end{align*}
We shall prove $(ii)$. Using the above estimates together with \eqref{sobnecessary}, applied with  $u=\tilde{u}_j$ and $B_0=B_r$, we obtain
\begin{equation*}
	\frac{1}{2}\left[\mu\left(\tilde{B}_{j+1}\right)\right]^{\frac{1}{\gamma^-_{B_r}}} \leq C_{\lip}C_{S}\left(\frac{\mu(B_r)}{r^{Q(x)}}\right)^{\omega} \frac{2^{s_{B_r}^+(j+2)}}{\left[\varphi_x(r)\right]^{s_{B_r}^+}}\left\|\chi_{\tilde{B}_j}\right\|_{L^{p(\cdot)}(\sigma B_r)}.
\end{equation*}
Since $s^+ \preceq_{q^-} 1$, $\displaystyle r\leq \frac{3}{\lambda^2}\varphi_x(r)$, $\mu(\tilde{B}_j)\leq\mu(B_r)\leq1$, and $\displaystyle \frac{1}{\gamma^+_{B_r}}+\frac{s^-_{B_r}}{Q^+_{B_r}}\leq \frac{1}{p^+_{B_r}}$ we have that
\begin{align*}
	\left[\mu\left(\tilde{B}_{j+1}\right)\right]^{\frac{1}{\gamma^-_{B_r}}}&\leq C_{\lip}C_{S}\left(\frac{\mu(B_r)}{r^{Q(x)}}\right)^{\omega}\frac{24}{\lambda^2}\frac{2^{js_{B_r}^+}}{r^{s_{B_r}^+}}\left\|\chi_{\tilde{B}_j}\right\|_{L^{p(\cdot)}(\sigma B_r)} \\ & \leq c_3\left(\frac{\mu(B_r)}{r^{Q(x)}}\right)^{\omega}\frac{2^{js_{B_r}^+}}{r^{s_{B_r}^+}}\left[\mu\left(\tilde{B}_j\right)\right]^{\frac{1}{\gamma^+_{B_r}}+\frac{s^-_{B_r}}{Q^+_{B_r}}},
\end{align*}
where $\displaystyle c_3:=\frac{24}{\lambda^2}C_{\lip}C_{S}.$

We define 
\begin{equation*}
	r':=\frac{1}{2}\min \left\{\frac{1}{4}, \frac{1}{2}\exp\left(-\frac{C_{\log}(1/\gamma)Q^+}{s^-}\right),\frac{1}{\sigma}\right\}
\end{equation*}
and let $\eta$, $\eta_r$ be as in \eqref{kwadracik} and \eqref{kwadracik1}. Then, by the same argument as in $(i)$ we have $\eta_r \geq \eta>0$.	Since we have $\displaystyle \frac{1}{\gamma^-_{B_r}}<\frac{s^-_{B_r}}{Q^+_{B_r}}+\frac{1}{\gamma^+_{B_r}}$ and also $\mu\left(B(x,\varphi_x(r)/2)\right)\leq \mu\left(\tilde{B}_j\right)\leq \mu\left(B(x,r)\right)$, we can apply Lemma~\ref{iterative lemma} with
\begin{equation*}
	a_j:=\mu\left(\tilde{B}_j\right),\hspace{2mm} \frac{1}{p}:=\frac{s^-_{B_r}}{Q_{B_r}^+}+\frac{1}{\gamma^+_{B_r}},\hspace{2mm} \frac{1}{q}:=\frac{1}{\gamma^-_{B_r}},\hspace{2mm} \rho:=c_3\frac{\left[\mu(B_r)\right]^{\omega}}{r^{\omega Q(x)+s_{B_r}^+}}, \hspace{2mm} \textnormal{and} \hspace{2mm}\tau:=2^{s_{B_r}^+}
\end{equation*}
to conclude that
\begin{equation*}
	\left[\mu\left(\tilde{B}_1\right)\right]^{1-\frac{1}{t_r}}\rho^{\frac{\gamma^-_{B_r}}{t_r}}\tau^{\frac{1}{\eta_r}}\geq 1, 
\end{equation*}
where we denote $\displaystyle t_r:=\gamma^-_{B_r}\left(\frac{s^-_{B_r}}{Q^+_{B_r}}+\frac{1}{\gamma^+_{B_r}}\right)=1+\eta_r\gamma^-_{B_r}$. Hence

\begin{equation}\label{r2}
	\mu\left(B_r\right)\left(c_3\frac{\left[\mu(B_r)\right]^{\omega}}{r^{\omega Q(x)+s_{B_r}^+}}\right)^{\frac{\gamma^-_{B_r}}{t_r-1}}2^{s_{B_r}^+\frac{t_r\gamma^-_{B_r}}{\left(t_r-1\right)^2}}\geq 1.
\end{equation}
%Now, by using $\eta_r\geq \eta$ and $\gamma^-_{B_r}\geq \gamma^-$ we get
%\begin{equation*}
%	\frac{\gamma^-_{B_r}}{t_r-1}=\frac{1}{\eta_r}\leq \frac{1}{\eta} \hspace{2mm} \textnormal{ and } \hspace{2mm} \frac{t_r\gamma^-_{B_r}}{\left(t_r-1\right)^2}\leq \frac{1}{\gamma^-_{B_r}\eta_r^2}+\frac{1}{\eta_r}\leq \frac{1}{\gamma^- \eta^2}+\frac{1}{\eta}.
%\end{equation*}
Therefore, if we define $c_4:=\max\left\{1,c_3\right\}$, then inequality \eqref{r2} and \eqref{fds-45} yield
\begin{equation*}
	r^{ \frac{\omega Q(x)+s_{B_r}^+}{\eta_r}} \leq c_4^{\frac{1}{\eta}}2^{\frac{1}{\gamma^-\eta^2}+\frac{1}{\eta}}\left[\mu\left(B(x,r)\right)\right]^{\frac{t_r-1+\omega \gamma_{B_r}^-}{t_r-1}}=c_4^{\frac{1}{\eta}}2^{\frac{1}{\gamma^-\eta^2}+\frac{1}{\eta}}\left[\mu\left(B(x,r)\right)\right]^{\frac{\omega+\eta_r}{\eta_r}}.
\end{equation*}
Setting $c_5^{-1}:=c_4^{\frac{1}{\eta}}2^{\frac{1}{\gamma^-\eta^2}+\frac{1}{\eta}}$ and $\beta_r:=s_{B_r}^+-Q_{B_r}^+\eta_r$, and using the fact that $r\leq1$, we obtain
\begin{equation*}
	\mu\left(B(x,r)\right)\geq c_5^{\frac{\eta_r}{\omega+\eta_r}} r^{\frac{\omega Q(x)+s_{B_r}^+}{\omega+\eta_r}} \geq \min\left\{c_5,1\right\} r^{Q_{B_r}^+}r^{\frac{\beta_r}{\eta_r}}.
\end{equation*}
Therefore, from \eqref{kwadracik3} we have
\begin{equation*}
	\mu(B(x,r))\geq \min\left\{c_5,1\right\}c_2r^{Q_{B_r}^+}.
\end{equation*}
Finally, applying \eqref{kwadracik4} we get $\mu(B(x,r))\geq br^{Q(x)}$ with
\begin{equation*}
	b:=\min\left\{c_5,1\right\}c_2e^{-C_{\log}(Q)}2^{Q^- - Q^+}.
\end{equation*}
Hence, $(ii)$ is proven.

Now we turn our attention to the proof of $(iii)$. We fix $r':=\frac{1}{2}\min\left\{\frac{1}{4},\frac{1}{\sigma}\right\}.$ Using the estimates between the proof of $(i)$ and $(ii)$ together with \eqref{mosnecessary}, applied with $u=\tilde{u}_j$, we obtain
\begin{equation}\label{mt}
	\frac{\mu(\tilde{B}_{j+1})}{\mu(B_r)}\exp\left(C_{MT1} \frac{\lambda^2}{6C_{\lip}} \frac{r^{s_{B_r}^+}}{2^{s_{B_r}^+\left(j+2\right)}\left[\mu(\tilde{B}_j)\right]^{\frac{s_{B_r}^-}{Q_{B_r}^+}}} \right)^{\omega} \leq C_{MT2}.
\end{equation}
Without loss of generality we may assume that $C_{MT2}>1$. Then, using the estimate
\begin{equation*}
	\log(y) \leq 2Q_{B_r}^+\left(s_{B_r}^- \omega\right)^{-1}y^{\frac{s_{B_r}^- \omega}{2Q_{B_r}^+}},
\end{equation*}
which holds for every $y\in (0,\infty)$, it follows from \eqref{mt} that
\begin{equation*}
	\frac{\lambda^2 C_{MT1}}{6C_{\lip}} \frac{r^{s_{B_r}^+}}{2^{s_{B_r}^+\left(j+2\right)}\left[\mu(\tilde{B}_j)\right]^{\frac{s_{B_r}^-}{Q_{B_r}^+}}} \leq \left[\log\left( C_{MT2} \frac{\mu(B_r)}{\mu(\tilde{B}_{j+1})} \right)\right]^{\frac{1}{\omega}}\leq \left(\frac{2Q^+}{s^- \omega}\right)^{\frac{1}{\omega}} C_{MT2}^{\frac{s^+}{2Q^-}} \left(\frac{\mu(B_r)}{\mu(\tilde{B}_{j+1})}\right)^{\frac{s_{B_r}^-}{2Q_{B_r}^+}}.
\end{equation*}
Therefore,
\begin{equation*}
	\left[\mu(\tilde{B}_{j+1})\right]^{\frac{s_r^-}{2Q_r^+}} \leq c_6 2^{s_{B_r}^+ j} \frac{\mu(B_r)^{\frac{s_{B_r}^-}{2Q_{B_r}^+}}}{r^{s_{B_r}^+}}  \left[\mu(\tilde{B}_j)\right]^{\frac{s_{B_r}^-}{Q_{B_r}^+}},
\end{equation*}
where
\begin{equation*}
	c_6:= \left(\frac{2Q^+}{s^- \omega}\right)^{\frac{1}{\omega}} C_{MT2}^{\frac{s^+}{2Q^-}} \frac{24}{\lambda^2}\frac{C_{\lip}}{C_{MT1}}.
\end{equation*}
Now, by the virtue of Lemma~\ref{iterative lemma}, applied with
\begin{equation*}
	a_j:=\mu(\tilde{B_j}), \hspace{4mm} \tau:=2^{s_{B_r}^+}, \hspace{4mm} \rho:=c_6 \frac{\left[\mu(B_r)\right]^{\frac{s_{B_r}^-}{2Q_{B_r}^+}}}{r^{s_{B_r}^+}}, \hspace{4mm} \frac{1}{q}:=\frac{s_{B_r}^-}{2Q_{B_r}^+},\hspace{2mm} \textnormal{and} \hspace{2mm} \frac{1}{p}:=\frac{s_{B_r}^-}{Q_{B_r}^+}
\end{equation*}
we obtain
\begin{equation}\label{poiteracji}
	\left[\mu(\tilde{B}_1)\right]^{\frac{1}{2}} c_6^{\frac{Q_{B_r}^+}{s_{B_r}^-}} \frac{\left[\mu(B_r)\right]^{\frac{1}{2}}}{r^{\frac{Q_{B_r}^+s_{B_r}^+}{s_{B_r}^-}}} 2^{\frac{2Q_{B_r}^+}{s_{B_r}^-}}\geq 1.
\end{equation}
Hence, by Lemma~\ref{loglemma} $(i)$, applied with $z:=x$, $R:=2r$ and respectively $t:=1/Q$, $t:=1/s$, and $t:=s$, we get the following three inequalities
\begin{align*}
e^{C_{\log}(Q)}(2r)^{Q_{B_r}^+} &\geq (2r)^{Q(x)},\\
e^{C_{\log}(s)}(2r)^{s_{B_r}^+} &\geq (2r)^{s(x)},\\
e^{C_{\log}(1/s)}(2r)^{\frac{1}{s_{B_r}^-}} &\geq (2r)^{\frac{1}{s(x)}},
\end{align*}
which imply that
\begin{equation}\label{r}
	r^{\frac{Q_{B_r}^+s_{B_r}^+}{s_{B_r}^-}} \geq c_7 r^{Q(x)},
\end{equation}
where
\begin{equation*}
c_7:=e^{-C_{\log}(s)\frac{Q^+}{s^-}}e^{-C_{\log}(1/s)s^+Q^+}e^{-C_{\log}(Q)\frac{s^+}{s^-}}2^{-\frac{Q^+s^+}{s^-}}.
\end{equation*}
Finally, from \eqref{poiteracji} and \eqref{r} we get
\begin{equation*}
	\mu(B_r) \geq br^{Q(x)},
\end{equation*}
where
\begin{equation*}
	b:=c_7\max\left\{c_6,1\right\}^{-\frac{Q^+}{s^-}}2^{-\frac{2s^+Q^+}{s^-}}
\end{equation*}
and the proof of $(iii)$ is complete.

Now we prove $(iv).$ Let $x\in X$ and $r\in (0,1]$. It is enough to consider the case when $\mu(B(x,r)) < \min\left\{1,\mu(X)\right\}$. Then, $X \setminus B(x,r)\neq \emptyset$. Since $X$ is uniformly perfect, there exists a constant $\lambda \in (0,1)$, independent of $B(x,r)$, such that $B(x,r) \setminus B(x,\lambda r) \neq \emptyset$. We define the function ${u}: X\to \mathbb R$ by setting for each $y\in X$,
\begin{equation*}
	u(y):= \left\{ \begin{array}{lll} 1, & \textnormal{ for } y=x, \\ \displaystyle 1-\frac{d(x,y)}{\lambda r}, & \textnormal{ for } y\in B(x, \lambda r) \setminus \left\{x\right\},\\ 0, & \textnormal{ for } y \in X \setminus B(x, \lambda r). \end{array} \right.
\end{equation*}
It is easy to see that ${u}$ is a $(\lambda r)^{-1}$-Lipschitz function on $X$ which is zero outside of $B(x, \lambda r)$. Hence, $u\in A^{s(\cdot)}_{p(\cdot),q(\cdot)}(X,d,\mu)$ (see Theorem~\ref{lipschitz}) and therefore, by the assumption in statement $(iv)$, we know that there exists a constant $D_{H,G}>0$ such that, for every $y,z\in X$,
\begin{equation*}
	\left|u(y)-u(z)\right| \leq D_{H,G} \left\|u \right\|_{A^{s(\cdot)}_{p(\cdot),q(\cdot)}(X,d,\mu)}d(y,z)^{\alpha(y)}.
\end{equation*}
Take $y=x$ and $z\in B(x, r) \setminus B(x,\lambda r)$. Then
\begin{equation*}
	1\leq D_{H,G} \left\| u \right\|_{A^{s(\cdot)}_{p(\cdot),q(\cdot)}(X,d,\mu)} r^{\alpha(x)},
\end{equation*}
which, together with Theorem~\ref{lipschitz} and Proposition~\ref{rel}, gives
\begin{align*}
	1&\leq D_{H,G}C_{\lip} \frac{1}{\left(\lambda r\right)^{s_{B_r}^+}} \left[\mu(B(x,r))\right]^{\frac{1}{p_{B_r}^+}} r^{\alpha(x)}+D_{H,G}\left[\mu(B(x,r))\right]^{\frac{1}{p_{B_r}^+}}r^{\alpha(x)} \\ &\leq \frac{D_{H,G}\left(C_{\lip}+1\right) }{\lambda^{s^+}}\left[\mu(B(x,r))\right]^{\frac{1}{p_{B_r}^+}}r^{\alpha_{B_r}^- - s_{B_r}^+}.
\end{align*}
Thus,
\begin{equation}\label{r6}
	\mu(B(x,r)) \geq \frac{\lambda^{p_{B_r}^+s^+}}{D_{H,G}^{p_{B_r}^+}\left(C_{\lip}+1\right)^{p_{B_r}^+}} r^{p_{B_r}^+(s_{B_r}^+-\alpha_{B_r}^-)} \geq c_8 r^{p_{B_r}^+(s_{B_r}^+-\alpha_{B_r}^-)},
\end{equation}
where
\begin{equation*}
	c_8:=\frac{\lambda^{p^+s^+}}{\max\left\{1,D_{H,G}\left(C_{\lip}+1\right)\right\}^{p^+}}.
\end{equation*}
Finally, applying Lemma~\ref{loglemma} three times with $z:=x$, $R=2r$ and respectively $t:=1/s$, $t:=1/\alpha$ and $t:=1/p$, we obtain
\begin{align*}
	\left(2r\right)^{s_{B_r}^+}& \geq e^{-C_{\log}(s)}\left(2r\right)^{s(x)}, \\
	\left(2r\right)^{-\alpha_{B_r}^-} & \geq e^{-C_{\log}(\alpha)} \left(2r\right)^{-\alpha(x)}, \\ 
	\left(2r\right)^{p_{B_r}^+} & \geq e^{-C_{\log}(p)} \left(2r\right)^{p(x)}.
\end{align*}
Therefore
\begin{equation*}
	r^{p_{B_r}^+(s_{B_r}^+-\alpha_{B_r}^-)} \geq \left(e^{-C_{\log}(s)-C_{\log}(\alpha)} 2^{-\alpha^+-s^+}\right)^{p_{B_r}^+} r^{p_{B_r}^+\left(s(x)-\alpha(x)\right)} \geq  c_9 r^{p_{B_r}^+\left(s(x)-\alpha(x)\right)},
\end{equation*}
where
\begin{equation*}
	c_9:=\min\left\{1,e^{-C_{\log}(s)-C_{\log}(\alpha)} 2^{-\alpha^+-s^+}\right\}^{p^+}.
\end{equation*}
Using the fact that $s\geq \alpha$, we obtain
\begin{equation*}
	\mu(B(x,r)) \geq c_8c_9 r^{p_{B_r}^+\left(s(x)-\alpha(x)\right)} \geq  br^{p(x)(s(x)-\alpha(x))}=br^{Q(x)},
\end{equation*}
where 
\begin{equation*}
	b:=c_8c_9\min\left\{e^{-C_{\log}(p)}2^{p^- - p^+},1\right\}^{s^+}.
\end{equation*}
There remains to prove the last part of $(iv)$. Assume that $s(x_0)=\alpha(x_0)$ for some $x_0\in X$. From what we have just proven,
we know that there exists $b\in(0,1]$ such that, for all $x\in X$ and $r\in (0,1]$, there holds
\begin{equation*}
	\mu\left(B(x,r)\right) \geq br^{p(x)\left(s(x)-\alpha(x)\right)}.
\end{equation*}
Taking $x=x_0$ we get
\begin{equation}\label{n5}
	\mu(B(x_0,r))\geq br^{p(x_0)(s(x_0)-\alpha(x_0))}=b
\end{equation}
for every $r\in (0,1]$. Passing to the limit in \eqref{n5} as $r\to 0^+$ we obtain that $\mu\left(\left\{x_0\right\}\right)\geq b>0$, as wanted.
This completes the proof of $(iv)$ and hence, the proof of Theorem~\ref{necessaryembedding}.
\end{proof}

%\begin{cor}
%Let $(X,d,\mu)$ be an uniformly perfect metric measure space. Assume that there exist $s,p\in \mathcal{P}^{\log}_b(X)$, $\alpha \in \mathcal{P}_b^{\log}(X)$ and $q\in \mathcal{P}(X)$ such that
%\begin{equation*}
%	\dot{A}^{s(\cdot)}_{p(\cdot),q(\cdot)}(X,d,\mu) \hookrightarrow \dot{C}^{0,\alpha(\cdot)}(X,d)
%\end{equation*}
%where $A=M$ or $A=N$. If there is $x_0 \in X$ such that $\alpha(x_0)=s(x_0)$, then $\mu\left(\left\{x_0\right\}\right)>0$.	
%\end{cor}
%
%\begin{proof}
%From Theorem~\ref{necessaryembedding} $(iv)$ we know that there exists $b>0$ such that for $x\in X$ and $r\in (0,1]$ it holds
%\begin{equation*}
%	\mu\left(B(x,r)\right) \geq br^{p(x)\left(s(x)-\alpha(x)\right)}.
%\end{equation*}
%Taking $x=x_0$ we get
%\begin{equation}\label{n5}
%	\mu(B(x_0,r))\geq br^{p(x_0)(s(x_0)-\alpha(x_0))}=b
%\end{equation}
%for every $r\in (0,1]$. Passing in \eqref{n5} to the limit with $r\to 0^+$ we obtain that $\mu\left(\left\{x_0\right\}\right)\geq b>0$, which finishes the proof.
%\end{proof}
%
\textbf{Acknowledgements:} 
Some parts of this paper were developed during P.G.’s visits to Amherst College and the Birla Institute of Technology and Science, Pilani. P.G. gratefully acknowledges the hospitality received during these visits.

\bigskip
{\small Ryan Alvarado}\\
\small{Department of Mathematics,}\\
\small{Amherst College,}\\
\small{Amherst, MA, USA} \\
{\tt rjalvarado@amherst.edu}\\
\\
{\small Micha{\l} Dymek}\\
\small{Faculty Mathematics and Information Sciences,}\\
\small{Warsaw University of Technology,}\\
\small{Pl. Politechniki 1, 00-661 Warsaw, Poland} \\
{\tt michal.dymek.dokt@pw.edu.pl}\\
\\
{\small Przemys{\l}aw  G\'orka}\\
\small{Faculty of Mathematics and Information Sciences,}\\
\small{Warsaw University of Technology,}\\
\small{Pl. Politechniki 1, 00-661 Warsaw, Poland} \\
{\tt przemyslaw.gorka@pw.edu.pl}\\
\\
{\small Nijjwal Karak}\\
\small{Department of Mathematics,}\\
\small{Birla Institute of Technology and Science-Pilani, Hyderabad Campus,}\\
\small{Hyderabad 500078, India} \\
{\tt nijjwal@hyderabad.bits-pilani.ac.in}

\begin{thebibliography}{99999}
\bibitem{AF03} {\sc R. A. Adams, J. J. F. Fournier}, Sobolev Spaces, 2nd ed., Pure and Applied Mathematics, vol. 140, Elsevier, 2003.


\bibitem{AH} {\sc A. Almeida, P. H\"ast\"o}, Besov spaces with variable smoothness and integrability, J. Funct. Anal. 258 (2010), 1628-1655.

\bibitem{AS09} {\sc A. Almeida, S. Samko}, Embeddings of variable Haj\l asz–Sobolev spaces into H\"older spaces of variable order, J. Math. Anal. Appl. (2009) 353: 489-496.

\bibitem{AGH20}  {\sc R. Alvarado, P. G\'{o}rka, P. Haj{\l}asz}, Sobolev embedding for $M^{1,p}$ spaces is equivalent to a lower bound of the measure,  J. Funct. Anal. (2020), 279: 108628.

\bibitem{zwartestale} {\sc R. Alvarado, P. Górka, A. S{\l}abuszewski}, Compact embeddings of Sobolev, Besov, and Triebel-Lizorkin
spaces, J. Differential Equations 446, (2025), 113598. 64 pp.

\bibitem{AYY24}
{\sc R. Alvarado, D. Yang, D., and W. Yuan},
Optimal embeddings for {Triebel--Lizorkin} and {Besov} spaces on quasi-metric measure spaces, Math. Z. 307, No. 3 (2024), Paper No. 50. 59 pp.

\bibitem{AYY22}
{\sc R. Alvarado, D. Yang, D., and W. Yuan}, A measure characterization of embedding and extension domains for Sobolev, Triebel–Lizorkin, and Besov spaces in spaces of homogeneous type, J. Funct. Anal. 283 (2022), Paper No. 109687, 71 pp.

\bibitem{AWYY21}  {\sc R. Alvarado, F. Wang, D. Yang and W. Yuan},
Pointwise characterization of Besov and Triebel--Lizorkin spaces on spaces of homogeneous type, Studia Math. 268 (2023),  121--166. 
%\bibitem{bandaliyev} {\sc R. Bandaliyev, P. Górka}, Relatively compact sets in variable-exponent Lebesgue spaces, Banach J. Math. Anal. 12 (2018), 331-346.

\bibitem{BK96}
{\sc S. M. Buckley and P. Koskela}, Criteria for imbeddings of Sobolev--Poincar\'e type,
Int. Math. Res. Not. IMRN, 1996,  881--901.

\bibitem{cheeger}
{\sc J. Cheeger}, Differentiability of Lipschitz functions on metric measure spaces,
 Geom.\ Funct.\ Anal. 9 (1999), 428--517.
 
\bibitem{Coifman} {\sc R. R. Coifman, G. Weiss}, Analyse harmonique non-commutative sur certains espaces	homog`enes, Lecture Notes in Math., vol. 242, Springer-Verlag, Berlin, 1971. MR 58:17690.

\bibitem{CUF13}
{\sc D. Cruz-Uribe, A. Fiorenza}, Variable Lebesgue spaces. Foundations and Harmonic analysis, Birkh\"auser/Springer, Heidelberg, 2013

\bibitem{DHHR11} {\sc L. Diening, P. Harulehto, P. H\"{a}st\"{o}, M. R\u u\^zi\^cka}, Lebesgue and Sobolev Spaces with Variable Exponents, Lecture Notes in Mathematics, Springer, Heidelberg, 2011.

\bibitem{DG} {\sc M. Dymek, P. Górka}, Compactness in the spaces of variable integrability and summability, Math. Nachr. 296 (2023), 4317-4334.

\bibitem{Gaczkowski}  {\sc M. Gaczkowski, P. G\'{o}rka}, Sobolev spaces with variable exponents on complete manifolds, J. Funct. Anal. 270, 2016, 1379--1415.

\bibitem{gag}
{\sc E. Gagliardo},
Propriet\`a di alcune classi di funzioni in pi\`u variabili,
Ricerche Mat. 7 (1958), 102--137.

\bibitem{GG} {\sc A. Ghorbanalizadeh, P. G\'{o}rka}, Completeness and separability of the spaces of variable integrability and summability, Proc. Amer. Math. Soc. 149 (2021), 3873-3879.

%\bibitem{Górka} {\sc P. G\'{o}rka}, Looking For Compactness In Sobolev Spaces On Noncompact etric Spaces, Ann. Acad. Sci. Fenn., Vol 43, 2018, 531-540.

\bibitem{G} {\sc P. G\'{o}rka}, Separability of a Metric Space Is Equivalent to the Existence of a Borel Measure. Amer. Math. Monthly 128 (2021), no. 1, 84--86. 

\bibitem{GKP} {\sc P. G\'{o}rka, N. Karak, D. J. Pons}, Variable exponent Sobolev spaces and regularity of domains,  J. Geom. Anal. 31 (2021), no. 7, 7304–7319.

\bibitem{GS} {\sc P. G\'{o}rka, A. S{\l}abuszewski}, Embeddings of the fractional Sobolev spaces on
metric-measure spaces, Nonlinear Analysis 221 (2022), s. 112867.

\bibitem{Haj03} {\sc P. Haj{\l}asz}, Sobolev spaces on metric-measure spaces, In: Heat Kernels and Analysis on Manifolds, Graphs, and Metric Spaces (Paris, 2002), pp. 173–218, Contemp. Math. 338, Amer. Math. Soc., Providence, RI, 2003.

\bibitem{Haj96} {\sc P. Haj{\l}asz}, Sobolev spaces on an arbitrary metric space, Potential Anal. 5 (1996), 403--415.

\bibitem{SMP}
{\sc P. Haj\l{}asz and P. Koskela},
Sobolev met Poincar\'e. Mem.\ Amer.\ Math.\ Soc. 145 (2000), no. 688.

\bibitem{SMP2}
{\sc P. Haj\l{}asz and P. Koskela},
Sobolev Meets Poincar\'e,
C. R. Acad.\ Sci.\ Paris S\'er.\ I Math. 320 (1995), 1211--1215.


\bibitem{hajlaszkt1}
{\sc P. Haj\l{}asz, P. Koskela and H. Tuominen},
Sobolev embeddings, extensions and measure density condition,
J. Funct. Anal. 254 (2008), 1217--1234.

\bibitem{HHL06} {\sc P. Harjulehto, P. H\"ast\"o, V. Latvala}, Sobolev embeddings in metric measure
spaces with variable dimension, Math. Z. (2006) 254: 591–609.

\bibitem{Rn} {\sc Y. He, Q. Sun, C. Zhuo}, Pointwise characterizations of variable Besov and Triebel-Lizorkin spaces via Haj{\l}asz gradients. Fract. Calc. Appl. Anal. 27, 944–969 (2024).



\bibitem{K19}
{\sc N. Karak} Measure density and embeddings of Haj\l{}asz--Besov and Haj\l{}asz--Triebel--Lizorkin
spaces, J. Math. Anal. Appl. 475 (2019), 966–984.


\bibitem{K20}
{\sc N. Karak}, Lower bound of measure and embeddings of Sobolev, Besov, and Triebel--Lizorkin spaces, Math. Nachr. 293 (2020), 120–128.


\bibitem{KYZ11} {\sc P. Koskela, D. Yang, Y. Zhou},  Pointwise characterizations of Besov and Triebel--Lizorkin spaces and quasiconformal mappings, Adv. Math. 226 (2011), 3579–3621.

\bibitem{LLP10} {\sc F. Li, Z. Li, L. Pi,} Variable exponent functionals in image restoration. Appl. Math. Comput. 216(3), 870–882 (2010)

\bibitem{lukkainen} {\sc J. Luukkainen, E. Saksman}, Every complete doubling metric space carries a doubling measure, Proc. of the Amer. Math. Soc.,	Vol. 126, No. 2, 1998, pp. 531-534

%\bibitem{MS79} {\sc R. A. Macías, C. Segovia}, A decomposition into atoms of distributions on spaces of homogeneous type, Adv. Math. (1979) 33: 271–309.

\bibitem{Niren}
{\sc L. Nirenberg},
On elliptic differential equations, Ann. Scuola Norm. Pisa (III) 13 (1959), 1--48.

\bibitem{R00}
{\sc M. R\r{u}\v{z}i\v{c}ka}, Electrorheological Fluids: Modeling and Mathematical Theory, Springer-Verlag, Berlin, 2000.


\bibitem{shanmugalingam}
{\sc N. Shanmugalingam},
Newtonian spaces: an extension of Sobolev spaces to metric measure spaces,
Rev.\ Mat.\ Iberoamericana 16 (2000), 243--279.

%\bibitem{Shv07} {\sc P. Shvartsman}, On extensions of Sobolev functions defined on regular subsets of metric measure spaces, J. Approx. Theory. (2007) 144(2): 139-161.

\bibitem{sob36}
{\sc S. L. Sobolev},
On the estimates relating to families of functions having derivatives that are square integrable,
Dokl. Akad. Nauk SSSR 1 (1936), 267--270 (Russian).


\bibitem{sob38}
{\sc S. L. Sobolev},
On a theorem of functional analysis,
Mat. Sbornik 46 (1938), 471--497 (Russian).

\bibitem{Tri} {\sc H. Triebel}, Theory of Function Spaces, Monographs in Mathematics, vol. 78, Birkh\"auser, Basel, 1983.

\bibitem{trud}
{\sc N. Trudinger},
On imbeddings into Orlicz spaces and some applications,
J. Math. Mech. 17 (1967), 473--483.

\bibitem{Z87}
{\sc V.V. Zhikov}, Averaging of functionals of the calculus of variations and elasticity theory, Math. USSR–Izv. (29) (1987), 675–710

\end{thebibliography}
\end{document}